\documentclass[11pt, english, reqno]{amsart}
\usepackage{fullpage}
\usepackage[utf8]{inputenc}
\usepackage{amsmath,amssymb,stmaryrd,array,multirow,dsfont,mathrsfs}
\usepackage{enumerate}
\usepackage{xcolor}
\usepackage{color}
\usepackage{hyperref}
\usepackage{indentfirst}
\usepackage[T1]{fontenc}
\usepackage{microtype}
\usepackage{graphics}
 
\graphicspath{{Images/}}

\xdefinecolor{darkgreen}{rgb}{0,0.4,0}

\usepackage{hyperref}
\usepackage{indentfirst}
\usepackage{cite}
\usepackage[utf8]{inputenc}
\usepackage[T1]{fontenc}
\usepackage{microtype}

\usepackage{graphics}
 
 \usepackage{amsmath,amssymb,stmaryrd,array,multirow,dsfont,mathrsfs}
\usepackage{amsthm}
\usepackage{xcolor}
\usepackage{color}
\usepackage{hyperref}
\usepackage{tikz}

\graphicspath{{Images/}}

\xdefinecolor{darkgreen}{rgb}{0,0.4,0}

\newcommand{\beq}{\begin{equation}}
\newcommand{\eeq}{\end{equation}}
\newcommand{\beqs}{\begin{equation*}}
\newcommand{\eeqs}{\end{equation*}}
\newcommand{\ben}{\begin{eqnarray}}
\newcommand{\een}{\end{eqnarray}}
\newcommand{\beno}{\begin{eqnarray*}}
\newcommand{\eeno}{\end{eqnarray*}}
\renewcommand{\Re}{{\rm Re}\,}
\renewcommand{\Im}{{\rm Im}\,}

\newcommand{\Id}{{\rm Id}}
\newcommand{\Supp}{{\rm Supp}\,}

\newcommand{\Sup}{\displaystyle \sup}

\newcommand{\Rmnum}[1]{\uppercase\expandafter{\romannumeral #1} }
 \numberwithin{equation}{section}
\usepackage{mathtools}

\DeclarePairedDelimiterX{\inp}[2]{\langle}{\rangle}{#1, #2}

\newtheorem{thm}{Theorem}[section]
\newtheorem{lem}[thm]{Lemma}
\newtheorem{prop}[thm]{Proposition}
\newtheorem{rmk}[thm]{Remark}
\newtheorem{cor}[thm]{Corollary}

\def\op {\mathrm{Op}}
\def\curl{\mathop{\rm curl}\nolimits}

\def \d {\mathrm {d}}

\def\cE{{\mathcal E}}
\def\cF{{\mathcal F}}

\def\cH{{\mathcal H}}
\def\cI{{\mathcal I}}
\def\cJ{{\mathcal J}}
\def\cK{{\mathcal K}}
\def\cL{{\mathcal L}}

\def\cO{{\mathcal O}}
\def\cP{{\mathcal P}}
\def\cQ{{\mathcal Q}}
\def\cR{{\mathcal R}}
\def\cS{{\mathcal S}}
\def\cT{{\mathcal T}}

\def\cY{{\mathcal Y}}

\let\f=\frac
\def \p {\partial}

\def\mR {\mathbb{R}}

\def\ep{\epsilon}

\def \pt {\partial_{t}}

\def \div {\,\mathrm{div}}
\def \vp {\varphi}
\def \cX {\mathcal{X}} 
\def \bI {\mathbb{I}}
\def \sech {\mathrm{sech}}
\def \kdv {\textnormal{\scalebox{0.8}{KdV}}}

\def \kp {
\textnormal{\scalebox{0.8}{KP}}} 

\def \diag {\textnormal{diag}}

\def \na{\nabla}

\title{Transverse linear stability of one-dimensional solitary gravity water waves}
\date{\today}

\begin{document}
\author{Fr\'ed\'eric Rousset}
\address{Universit\'e Paris-Saclay,  CNRS, Laboratoire de Math\'ematiques d'Orsay (UMR 8628),  91405 Orsay Cedex, France}
\email{frederic.rousset@universite-paris-saclay.fr }

\author{Changzhen Sun}
\address{Universit\'e Bourgogne Franche-Comt\'e, Laboratoire de Mathématiques de Besançon, UMR CNRS 6623}
\email{changzhen.sun@univ-fcomte.fr}

\begin{abstract}

In this paper, we establish the transverse linear asymptotic stability of  one-dimensional small-amplitude solitary waves of the  gravity water-waves system.
More precisely, we show   that the semigroup of the linearized operator about the solitary wave  decays exponentially within a spectral  subspace supplementary  to the  space generated by the spectral projection on   continuous resonant modes. The key element of the proof  is to establish suitable  uniform resolvent estimates.
To achieve this, we use  different arguments depending on the size of the transverse frequencies,
for high transverse frequencies,  we use reductions based on pseudodifferential calculus, for intermediate ones, we use 
 an energy-based approach relying on  the design  of various appropriate energy functionals for different regimes of  longitudinal 
 frequencies  and  for low frequencies, we use the KP-II approximation. As a corollary of our main result, we also get the spectral stability in the unweighted energy space.
 
%In this paper, we prove the transverse linear asymptotic stability for one-dimensional small amplitude solitary gravity water waves, in the exponentially weighted space. It is found that the semigroup of the linearized operator around the solitary wave decays exponentially on the subspace orthogonal to the space consisting of continuous resonant modes. The proof relies on the uniform resolvent estimates which is the heart of the paper. To accomplish this, we use semiclassical analysis in the uniform high transverse frequency regime, energy-based approach achieved through designation of appropriate energy functionals in the intermediate transverse frequency regime and KP-II approximation in the low transverse frequency regime.

\vspace{1em}


\end{abstract}

\maketitle
\tableofcontents

\section{Introduction}

In this paper, we study  the transverse linear asymptotic stability of line solitons for the three-dimensional gravity water-waves system  in a channel with finite depth. The equations of motions are given by 
\beq\label{o-in}
    \left\{ \begin{array}{l}
         \pt U+U \cdot \na U+\na P +g e_3=0, \\[3pt]
          \div\, U=0,\, \curl U=0.
    \end{array}
    \right. (t,X,z)\in \Omega_t \,.
\eeq
Here the fluid domain $\Omega_t$ which is also  unknown is defined as
\beqs 
\Omega_t=\{-h\leq  z\leq \zeta(t,X)\}, \quad X=(x,y)\in \mR^2\, ,
\eeqs
    where $h$ is a fixed finite constant and $\zeta$ is the free surface, which satisfies the kinematic condition 
\beq\label{o-sur}
\p_t \zeta (t,X) +U_X(t,X, \zeta(X))\cdot \na_X \zeta  (t,X)-U_3(t,X, \zeta(X))=0\,.
\eeq
We also impose that the pressure $P$ is constant, for example zero on the free boundary.
At the bottom of the fluid, we add  the boundary condition:
\beq\label{o-bot}
U_3|_{z=-h}=0\,.
\eeq
Given the mathematical challenges posed by its fully nonlinear nature, the water wave problem has been the subject of numerous studies over the past four decades. After the breakthrough works  of S. Wu, many results  
have emerged concerning the local  well-posedness and  the long-term behavior of the water wave system,  with or without surface tension,
  \cite{ A-B-Z-invention, ABZ-Gravity-nonlocalized,alazard-delort-2d-gravity, AS-Lannes,3d-surfacetension, Vicol-kukavica, GMS-3dgravity-infinite, fabio-ionescu-2dglobal, Lannes-JAMS, Tataru-almostglobal2d, Tataru-lowregu,Tataru-ctvorticity,su-2dpointvor,Chao-Mei-corner, Wang-3dgravity-finitedepth,Wang-capillary-finitedepth, wu-local-2d, Wu-3dwaterwavelocal,Wu-3dglobal,zhang-zhang,zheng-longtimegravity}. 
In this list, all the global in time  results can  be  viewed as proofs of  the asymptotic  stability of the constant solution
 to the water waves problem, that is to say the stationary  flat surface. Note that a byproduct of these results is the non-existence
  of small localized non-decaying  in time patterns  either in infinite depth or in finite depth  in dimension two
  (in the sense that the free surface is two-dimensional).

Nevertheless, there are many  physically  interesting nontrivial patterns  such as solitary  waves which are spatially localized traveling waves, we refer for example to  \cite{Lavrent-existnecesolitary, Beale, Iooss-Kirch, F-H-existence-solitary,Groves-Sun,Chen-eistence-with-vorticity}
for rigorous mathematical existence results
 or periodic  waves \cite{stokes1847theory,Struik-exis-periodic,levi1925determination,whalen-doublyperiodic,toland} (also referred to as Stokes waves). %Given their prevalence 
To justify  their observation   in nature, the  study of the dynamics around these non-trivial  patterns is  also of much importance. For example,  recently, there has been significant progress about  the spectral and modulational nonlinear instability of the Stokes waves \cite{Mielke-BF, chen-su, Nguyen-Strauss-periodic, Berti-Maspero,Hur-Yang, Haragus-Wahlen}. %Important progress has been made in recent years in understanding the spectral and modulational nonlinear instability of Stokes waves \cite{Mielke-BF, chen-su, Nguyen-Strauss-periodic, Berti-Maspero, Hur-Yang}. 
%Turning to the stability of solitary waves, which are spatially localized traveling waves,

The study of the stability of  solitary water waves has a long history in the mathematical literature.
 Let us first review some results for one-dimensional water waves (we use again  this terminology for the case that the free surface
 is one-dimensional) submitted to one-dimensional perturbations. 
For capillary-gravity water waves, the spectral stability of one-dimensional solitary waves subject to finite wave-length perturbation was established in \cite{Mariana-Arnd}. The nonlinear orbital stability  has been also  obtained conditional to  the global existence of the solution by  Mielke \cite{Mielke-conditionedorbital} and Buffoni \cite{Buffoni-B}  in the case of strong surface tension, while Groves and Wahl\'en proved it for weak surface tension \cite{Groves-whalen-conditioned-weak}. These proofs rely 
either on the general framework of Grillakis-Shatah-Strauss  \cite{GSS}, or on the variational approach of  \cite{Cazenave-Lions}
 and  hence on the definiteness up to a a finite dimensional subspace  of the energy-momentum functional. For the purely gravity water waves, such a variational approach does not apply due to the strong undefiniteness  of the energy-momentum functional so that a direct study of the linearized equation and of the spectrum of the linear operator must be performed. Lin proved in \cite{Lin-ins-largewave} the linear instability of large amplitude solitary waves. For small amplitude gravity solitary waves,  Pego and Sun \cite{Pego-Sun}, in the spirit of  the  results  previously obtained on  the KdV equation and related models  in \cite{Pego-Weinstein-kdv}, \cite{Pego-Weinstein-L}, obtained the linear asymptotic stability by relying on  the KdV approximation in the long wave region in the framework of  exponentially weighted spaces. They also deduced spectral stability in the unweighted energy space. 
The question of  nonlinear stability  for dispersive Hamiltonian PDE for which the local Cauchy problem involves the control of norms
of higher regularity than the energy space remains a widely open problem.
%{\color{red} quote recent works for semi linear equations NLS and KG:  Cuccagna, Bambusi, Schlag, Germain-Pusateri, Lindblad Rodnianski, Martel-Merle, Martel-Munoz}.
%The nonlinear orbital stability of small amplitude solitary waves still remains open, which is believed to be a challenging problem. Let us mention that for two dimensional gravity water waves in finite channel, even the global stability of the trivial solution has not been proven. The main difficulties arise from the weak dispersion and large resonances in low frequencies. Nevertheless, given the stronger dispersive effect, %

Here, we shall study the stability of  one-dimensional gravity solitary waves submitted to two-dimensional localized perturbations, 
this is usually referred as the study of transverse stability.
The study of the transverse stability of one-dimensional solitary waves also  has a long history in dispersive PDE.
In the early 1970's,  by using the theory of integrable systems Zakharov \cite{Zakharov-K} obtained the transverse instability of the soliton of the KdV equation considered as a one-dimensional solution of the (two-dimensional) KP-I equation.
A general criterion allowing to get linear transverse instability of solitary waves has now  been obtained in \cite{RT-MRL}. This criterion applies to a large class of systems like the KP-I equation and the water-waves system with strong surface tension
(Bond number bigger than $1/3$).
 One can then obtain  nonlinear instability, see  \cite{RT-JMPA}, \cite{RT-AIHP}  and in particular \cite{R-T-Transinstability} for the transverse
 nonlinear instability of solitary water-waves with strong surface tension.
 When the algebraic criterion of \cite{RT-MRL} is not matched,  there are various model situations.
  For the hyperbolic Schr\"odinger equation (where the usual  Laplacian is replaced by $\partial_{x}^2 - \Delta_{y}$), as established
  in \cite{Zakharov-R}, there is still transverse linear instability of the one-dimensional solitary wave whereas for the KP-II
  equation, there is stability \cite{Alexander-Pego-Sachs}.
  One can then conjecture that for small amplitude solitary waves, when the weakly transverse long-wave regime is given by 
  the KP-II equation, which is the case of the gravity water waves system (see \cite{Lannes-Saut} for the justification for example)  one should get transverse stability.   This heuristics was recently justified by Mizumachi \cite{Mizumachi-BL-linear} for the Benney-Luke equation which is another
  asymptotic model that can be obtained from the water-waves system while keeping the KP-II equation as
  a long-wave weakly transverse model. Note that  for  these semilinear equations,  which are globally well-posed in the energy space, 
  one can obtain nonlinear asymptotic stability  \cite{Mizumachi-KP-nonlinear}, \cite{Mizumachi-BL-nonlinear}.
  This remains a very challenging question for the gravity water-waves system.
  Here, we thus address the question of the linear transverse stability of small amplitude solitary waves for the pure gravity water waves system. Since the solitary wave does not depend on the transverse variable, one can Fourier transform in the transverse variable
  and study a family of  linear problems indexed by the transverse frequency parameter $\eta$. Due to the strongly nonlinear and nonlocal
  features of the water waves system, the dependence with respect to $\eta$ is stronger than in semilinear models like the Benney-Luke
  equation studied in  \cite{Mizumachi-BL-linear} and cannot be handled in a perturbative way especially in a  high frequency regime
  so that new ideas need to be introduced. This will be explained in more details after the statement of the main results.

%fact that the dispersion is stronger in the higher dimension,
%one could anticipate that 
%the nonlinear stability would be relatively easier to achieve. 
%This motivates us to consider three dimensional gravity water waves that admits solitary waves  depending only on one variable---referred to as line solitary waves---and study the stability of this one dimensional object. Since the perturbation depends also on the transverse variable, this is commonly referred to as 
%transverse stability.  In this paper, our main objective is to establish transverse linear stability, paving the way for subsequent exploration into nonlinear transverse stability.
\subsection{The line solitary waves} 

As the velocity $U$ is divergence and curl free, we can  write $U=\na \Phi,$ where the velocity potential $\Phi$ solves  the elliptic equation:
\beq
    \left\{ \begin{array}{l}
         \Delta \Phi=0, \quad (X,z)\in \mR^2\times \{-h<z<\zeta(t,X)\}, \\[3pt]
           \Phi|_{z=\zeta}=\varphi(t,X), \quad \p_z\Phi|_{z=-h}=0.
    \end{array}
    \right. 
\eeq
We then define the Dirichlet-Neumann operator
\beqs 
G[\zeta]\varphi=(\p_z \Phi-\na_X\zeta\cdot\na_X \Phi)|_{z=\zeta(t,X)}=\sqrt{1+|\na_X\zeta|^2}\,\na \Phi (X,\zeta(t,X))\cdot n(t,X)
\eeqs
where 
\beqs 
n(t,X)= \f{1}{\sqrt{1+|\na_X\zeta|^2}}\big(-\na_X\zeta,1\big)^t
\eeqs
is the unit outward normal vector to the free surface.
The system \eqref{o-in}-\eqref{o-bot} is thus equivalent to the following system on $(\zeta, \varphi)^t$ evaluated on the surface: 
\beq\label{ww-sur}
    \left\{ \begin{array}{l}
         \pt \zeta = G[\zeta]\varphi, \\
         \pt \varphi=-\f12 |\na_X\varphi|^2+\f12 \f{\big(G[\zeta]\varphi+\na_X\vp\cdot\na_X\zeta\big)^2}{\sqrt{1+|\na_X\zeta|^2}}-g\zeta. 
    \end{array}
    \right. 
\eeq
Since we shall study the solitary waves with speed $c>0,$ we first make change of variable $(x,y)\rightarrow (x-ct,y)$ to change \eqref{ww-sur} as: 
\beqs 
    \left\{ \begin{array}{l}
         \pt \zeta =c\p_x \zeta+ G[\zeta]\varphi, \\
         \pt \varphi=c\p_x\vp-\f12 |\na_X\varphi|^2+\f12 \f{\big(G[\zeta]\varphi+\na_X\vp\cdot\na_X\zeta\big)^2}{\sqrt{1+|\na_X\zeta|^2}}-g\zeta. 
    \end{array}
    \right. 
\eeqs
As is classical in the study of existence of solitary waves, we perform the following change of variable to make 
the equation non-dimensional:
%introduce a noon-dimensional version
\beq\label{changeofvarible} 
\hat{t}={ct}/{h}, \quad 
\hat{X}={X}/{h}, \quad  \tilde{\zeta}=\zeta/h, \quad  \tilde{\vp}=\vp/ch. %\zeta(t,X)=h \tilde{\zeta}
% t\rightarrow {ct}/{h}, \quad X \rightarrow {X}/{h}, \quad
\eeq
Then the equation becomes (we omit tilde for the sake of notational convenience) 
\beq \label{sur-mov-uni}
    \left\{ \begin{array}{l}
         \pt \zeta =\p_x \zeta+ G[\zeta]\varphi\,, \\
         \pt \varphi=\p_x\vp-\f12 |\na_X\varphi|^2+\f12 \f{(G[\zeta]\varphi+\na_X\vp\cdot\na_X\zeta)^2}{\sqrt{1+|\na_X\zeta|^2}}-\gamma\zeta\, ,
    \end{array}
    \right. 
\eeq
where $\gamma=\f{gh}{c^2}.$   Note that the above Zakharov-Craig-Sulem  formulation of the water waves system has the following Hamiltonian structure: 
\beqs
\pt \left( \begin{array}{l}
     \zeta  \\
      \varphi
\end{array} \right)=J\, \delta  \cH [ \zeta, \varphi], 
\eeqs
with 
\begin{align*}
J\,&:=\,\begin{pmatrix}0&-1\\1&0\end{pmatrix}\,,&
\mathcal{H}[\zeta, \varphi]
&=: \frac12 \int G[\eta] \varphi \cdot \varphi +\gamma \zeta^2-2\,\zeta\, \p_x \varphi\,.
\end{align*}

The existence of  one dimensional solitary waves, which are stationary solutions to \eqref{sur-mov-uni} depending only on the one-dimensional  variable $x$, is established in \cite{Beale,Kir-existence-grivatywaves}:
\begin{thm}[Restatement of Beale's result \cite{Beale}]
\label{bealetheo}
   Assuming $\gamma=1-\ep^2,$
 there exists $\ep_0$ such that for every $\ep\in (0, \ep_0)$  there is a stationary solution 
  $\zeta_{c}(x), \varphi_{c}(x)$ (independent of $y$), given by  
  \beq\label{solitarywave-def}
\zeta_{c}(x)=\ep^2\Theta(\ep x, \ep) , \qquad \vp_{c}(x)=\ep\Phi(\ep x, \ep) ,
  \eeq
  where $\Theta$ and $\Phi$ satisfy the following: \\[2pt]
  (1)
  There exists $d>0$, for any integers $k\geq 0 , \, \ell\geq 1, $ there exists $C_k, \, C_{\ell}$ such that: 
  \beqs 
|\p_{\hat{x}}^k \Theta (\hat{x},\ep)|\leq C_k \, e^{-d\,|\hat{x}|}, \qquad |\p_{\hat{x}}^{\ell} \Phi (\hat{x},\ep)|\leq C_{\ell} \, e^{-d\,|\hat{x}|} .
  \eeqs
 (2) It holds that when $\ep=0,$
  \beq \label{solitarywave-0}
\Theta (\hat{x}, 0)=\p_{\hat{x}} \Phi (\hat{x}, 0)=\Psi_{\kdv}(\hat{x})%=\sech^2\bigg(\f{\sqrt{3}}{2} \hat{x}\bigg).
  \eeq
  where $\Psi_{\kdv}(\hat{x})=\sech^2\bigg(\f{\sqrt{3}}{2} \hat{x}\bigg)$. \end{thm}
  Note that $\Psi_{\kdv}$ is the steady solution to the  KdV equation in the moving frame at unit speed: 
  \beqs 
-\p_x\Psi+3\Psi\p_x\Psi+\f13 \p_x^3 \Psi =0.
  \eeqs

Note that the existence is proven in \cite{Beale} by using the Levi-Civita formulation but the result can  be easily transformed into the Zakharov-Craig-Sulem formulation stated above. 

%In the following, we will adopt the no
%to it as two dimensional water wave system.
%Since we have reduced the original water wave problem to the equation for the surface elevation  $\zeta$ and  the trace of the velocity potential  $\vp$ \eqref{ww-sur}
%from now on we will refer the dimension of water wave system as that of the system satisfied by $(\zeta, \vp).$
As mentioned previously, the linear stability analysis of small-amplitude solitary waves of the  one dimensional water wave system has been established in \cite{Pego-Sun} by Pego and Sun. The core of the analysis is performed  within the framework of an exponentially weighted energy space. In this paper,  we shall  prove  the linear stability of the line solitary wave  within the context of the two-dimensional system described by \eqref{sur-mov-uni}. 
%Under the long-wave regime, the one-dimensional water waves system can be effectively approximated  by using the classical Korteweg-de Vries (KdV) equation, while the two-dimensional gravity water waves  can be approximated by the Kadomtsev-Petviashvili II (KP-II) equation, we refer to  {\color{red} quote Lanne book as a reference}. Recently, Mizumachi has elucidated the transverse linear and nonlinear stability of line solitons within the KP-II equation in \cite{Mizumachi-KP-nonlinear}.  
%Based on these insights, it is reasonable to anticipate that the line soliton for water waves, which can be well approximated by the line soliton of the KdV equation, exhibits transverse linear stability in the two-dimensional scenario. At this stage, it's worth mentioning that when strong surface tension is considered, the line solitary wave is found to be transversely unstable \cite{pego-sun-transinsta-linear, R-T-Transinstability}, consistent with results for model equations. That is, the KdV soliton is  transversely unstable for the KP-I equation \cite{R-T-KP-I}.

\subsection{Linearization of \eqref{sur-mov-uni} about line solitary waves}

We now linearize the system \eqref{sur-mov-uni}  about a  solitary wave $(\zeta_{c}, \vp_{c})^t$ given by Theorem \ref{bealetheo}. By using the expression of the derivative of the Dirichlet-Neumann operator with respect to the surface [Theorem 3.2.1, \cite{Lannes}], we  find the following linearized system:
\begin{align*}
\left\{ \begin{array}{l}
     \displaystyle  \pt \zeta= \p_x \zeta+ G[\zeta_c]\varphi-G[\zeta_{c}](Z_{\ep}\zeta)-\p_x(v_{c}\zeta),   \\ [5pt]
     \pt \vp =(1-v_{c})\p_x\vp+Z_{c}G[\zeta_{c}](\vp-Z_{c}\zeta)-(\gamma+Z_{c}\p_x v_{c})\eta\, ,
\end{array}\right.   
\end{align*}
with 
\beq \label{defZv}
Z_{c}=\f{G[\zeta_{c}]\varphi_{c}+\p_x\zeta_{c}\p_x\vp_{c}}{1+|\p_x\zeta_{c}|^2}, \qquad \qquad v_{c}=\p_x \vp_{c}-Z_{c}\p_x\zeta_{c}.
\eeq
We then set  ${V}_1=\zeta, 
{V}_2=\vp-Z_{c}\zeta$,  and then obtain that $
{V}=({V}_1,{V}_2)^t$ solves: 
\beq\label{def-L}
\pt {V}=%J 
{L} {V}, \qquad \qquad {L}= \left( \begin{array}{cc}
   \p_x(d_c\cdot)  &  G[\zeta_{c}]   \\[5pt]
 -  w_c& d_c\p_x
\end{array}\right)\, ,
\eeq
where
\begin{equation}
\label{defdgamm}
d_c=1-v_{c}, \quad  w_c=\gamma-d_c\p_x Z_{c}.
\end{equation}   

Let us now introduce the adapted functional framework. For any $a\in \mR,$ we shall use the weighted $L^2$ space
 $L_a^2 (\mR^2)=\colon L^2 (\mR^2; e^{2ax}\d x\d y\,) \,$  endowed with the norm:  % $L^2 (\mR^2)$  space with the norm %defined as:  
\beqs 
\|f\|_{L_a^2 (\mR^2) }=\colon \bigg(\int |f(x)|^2  e^{2ax} \d x\d y\bigg)^{\f{1}{2}}.
\eeqs
We also define the weighted space $H_{a,\star}^{{1}/{2}}(\mR^2)$ as the closure of smooth compactly supported functions endowed
 with the norm 
\begin{align*}
 % H_{a,\star}^{{1}/{2}}(\mR^2)= \bigg\{ f  \,\big|\, 
 \|f\|_{ H_{a,\star}^{{1}/{2}}(\mR^2)}= \colon \bigg(\int \big|\f{\nabla}{\langle D \rangle^{1/2} } f(x,y)\big|^2  e^{2ax} \d x\d y\bigg)^{\f{1}{2}}. %<+\infty\bigg\}. 
\end{align*}
Let us point out  that $\|\cdot\|_{H_{a,\star}^{{1}/{2}}(\mR^2)}$ is indeed a norm thanks to the Poincar\'e inequality 
 $$\|\na f\|_{L_a^2(\mR^2)}\geq |a| \|f\|_{L_a^2(\mR^2)}.$$
We shall  investigate the linear asymptotic behavior of the system \eqref{def-L} in the weighted space:
 \begin{equation}
 \label{defXaspace}X_a=\colon L_a^2(\mR^2)\times 
 H_{a,\star}^{{1}/{2}}(\mR^2).
 \end{equation} We consider $L$ as a linear operator on $X_a$ with  domain %$H_{a}^1(\mR^2)\times H_{a,\star}^{{3}/{2}}(\mR^2),$ where 
 \beqs 
%H_{a}^1(\mR^2)
H_{a}^1(\mR^2)\times H_{a,\star}^{{3}/{2}}(\mR^2)=:\big\{ F  \,\big|\, \langle D\rangle F \in X_a  \big\}.
 \eeqs
To state our main results, it is convenient to also   use    weighted energy spaces in the one-dimensional  variable $x$ only but involving
in a quantitative way a parameter $\eta \in \mathbb{R}$ which stands for the transverse frequency:
$$L_a^2(\mR)=\colon L^2(\mR; e^{2ax} \d x), $$
\begin{align}\label{Hhalfeta}
  H_{a,\star,\eta}^{{1}/{2}}(\mR)= \bigg\{ f  \,\big|\, \|f\|_{ H^{{1}/{2}}_{a,\star}(\mR)}= \colon \bigg(\int \big|\f{(\p_x, \eta)}{\langle (D_x, \eta) \rangle^{1/2} } f(x)\big|^2  e^{2ax} \d x\bigg)^{\f{1}{2}} <+\infty\bigg\}
\end{align}
and set  
\begin{equation}
\label{defspaceY}
Y_{a,\eta}=L_a^2(\mR)\times H_{a,\star,\eta}^{{1}/{2}}(\mR).
\end{equation}
We then consider the following unbounded operator  $L({\eta})$ on $Y_{a,\eta}$
 defined  for any smooth function $U(x)$ by  
$ L({\eta})(U(x))=e^{-iy\eta}L(e^{iy\eta}U(x)),$ so that $L({\eta})$ takes the form
\beq\label{def-Leta}
L({\eta})= \left( \begin{array}{cc}
   \p_x(d_c\cdot)  &  G_{\eta}[\zeta_{c}]  \\[5pt]
 -  w_c& d_c\p_x
\end{array}\right)
\eeq
with $G_{\eta}[\zeta_c]=e^{-i y\eta} G[\zeta_c] e^{i y\eta}$ being defined as 
$G_{\eta}[\zeta_c]f(x)=\sqrt{1+|\na \zeta_{c}|^2}(-\p_x\zeta_c\cdot\p_x\Phi_{\eta}+\p_z\Phi_{\eta})$ where $
\Phi_{\eta}(x,z)$ solves the elliptic problem in $\Omega=\mR\times[-1,\zeta_c]:$
$$(\Delta_{x,z}-\eta^2)\Phi_{\eta}=0, \quad \Phi_{\eta}|_{z=\zeta_c}=f, \quad \p_z \phi_{\eta}|_{z=-1}=0.$$
An important motivation for working in this weighted framework is  that the essential spectrum of the operator $L({\eta})$ which is
the imaginary axis in unweighted spaces is pushed to the left so that the first step in order to establish
 linear stability is to investigate the presence of eigenvalues close to the  imaginary axis.

We can  observe that when $\eta=0,$ $L(0)$ is the linearized operator for 1D water waves about  the line soliton $(\zeta_c, \, \vp_c-Z_c\zeta_c)$ studied in \cite{Pego-Sun}.  By the translational and Galilean invariances, there is  an eigenvalue $\lambda = 0$ with algebraic multiplicity two (see appendix). Consequently, for nonzero and small $\eta,$  we expect  that the operator $L(\eta)$  
 will have  two small  eigenvalues $\lambda(\eta)$ and $\overline{\lambda(\eta)}$ in the space $Y_{a,\eta}$ bifurcating from $\lambda(0)=0.$ 
 %{\color{red} the fact that the second eigenvalue has to be  $\lambda(-\eta)$ is hard to understand in the beginning...}
 This will be our first result:
%The first result concerns the resonant modes of $L({\eta})$ when focusing on the low transverse frequencies $\big\{\eta\big|\,|\eta|\leq {\eta}_0  \big\}. $ %\hat{\eta}_0 \ep^2\big\}.$
\begin{thm}\label{thm-resolmodes}
Let $0<a=\hat{a}\ep < \f{\sqrt{3}}{4}\ep.$ 
There exist $\epsilon_0>0,
\hat{\eta}_0=\hat{\eta}_0(\hat{a})>0$
such that for any $\ep\in (0,\ep_0], \eta \in[-\ep^2\hat{\eta}_0, \ep^2 \hat{\eta}_0],$ %any %$\hat{a}\in (0, \f12), \hat{\eta} \in (0, \hat{\eta}_0),$ 
the operators $L(\eta)$ %and $L^{*}(\eta)$
has two pairs of eigenmodes 
$(\lambda(\pm \eta), \, U(\cdot, \pm \eta))$ in the space $Y_{a, \eta},$ where
%\beqsL({\eta})U(\cdot,\eta)=\lambda(\eta)U(\cdot,\eta), \quad L({\eta})^{*}U^{*}(\cdot,\eta)=\lambda(-\eta)U^{*}(\cdot,\eta).\eeqs
\beqs
\lambda(\eta)\in C^{\infty}\big([-\ep^2\hat{\eta}_0 , \ep^2\hat{\eta}_0 ]\big) \quad U(\cdot, \eta)\in  C^{\infty}\big(%|\eta|\leq \ep^2 \hat{\eta}_0
[-\ep^2\hat{\eta}_0 , \ep^2\hat{\eta}_0 ], Y_{a,\eta}\big)\, . %\quad U^{*}(\cdot, \eta)\in  C^{\infty}(|\eta|\leq  \ep^2\hat{\eta}_0; Y_{-a,\eta})
\eeqs
Moreover, for any $\eta \in [-\ep^2\hat{\eta}_0 , \ep^2\hat{\eta}_0 ],$ it holds that
\begin{align*}
  & \lambda(\eta)=i\lambda_1\eta-\lambda_2\,\eta^2+\cO(\eta^2)\,, \\
  &   U(\cdot,\eta)=%v'_{\ep}(\cdot)+i\eta\, \vp_{1\ep}'(\cdot)
U^1+i\lambda_1\eta\, U^2+\cO(\eta^2) \, \,\,\text{ in } Y_{a,\eta}\, , \\
% \qquad \forall\, 
%& U^{*}(\cdot,\eta)=%v_{\ep}(\cdot)+i\eta \, \int_{-\infty}^x \vp'_{1\ep}(\cdot)U^{*, 1}-i\, \eta\, U^{*, 2}+\cO(\eta^2) \text{ in } Y_{-a,\eta}
&\overline{\lambda(\eta)}=\lambda(-\eta), \quad \overline{U(\cdot, \eta)}=U(\cdot, -\eta), 
\end{align*}
where $\lambda_1, \lambda_2 $ are two positive constants.  

In the same way, the adjoint operator $L^{*}(\eta)$ has  eigenmodes under  the form $(\lambda(\mp \eta), U^{*}(\cdot, \pm \eta))$ where 
$\overline{U^{*}(\cdot, \eta)}=U^{*}(\cdot, -\eta). $

\begin{rmk}
By the KP-II approximation, it is expected that the resonant modes in the long wave weakly transverse scaling  $U(\ep^{-1}\cdot, \ep^2\hat{\eta})$ should  be approximated by the resonant modes of the  linearized KP-II equation (see appendix \ref{approx-resonants}). Note that, as shown in \cite{Mizumachi-KP-nonlinear}, the latter is exponentially growing at $-\infty$ once $\hat{\eta}\neq 0$ (see \eqref{resomode-kp})
and thus  $U(\cdot, \eta)$ 
does not belong to the unweighted space $L^2(\mR).$

%We expect that $U(\cdot, \eta), U^{*}(\cdot, \eta)$ do  not lie in the unweighted space $Y_{0,\eta}.$ 
 %As will be seen in the proof, $U(\cdot, \eta), U^{*}(\cdot, \eta)$ do  not lies in the unweighted space $Y_{0,\eta}.$ %This explains  
 %and they indeed tend to some constants at $-\infty$ and $+\infty$ respectively. 
\end{rmk}

\end{thm}

As shown in the Section 2, one can then choose two smooth in $\eta$ functions
$(g_1(\cdot, \eta), g_2 (\cdot, \eta))$  such that  for $\eta \neq 0$  they form  a basis of 
the two dimensional space generated  in the space $Y_{a, \eta}$ by $U(\pm \eta)$ and for $\eta=0$ a basis of the generalized kernel of $L(0).$  Let $g_1^{*}(\cdot, \eta), g_2^{*} (\cdot, \eta)\in Y_{a, \eta}^{*}$ be the corresponding dual basis  defined in \eqref{defbasis-dual}. We can then 
define the spectral projector 
\begin{align}
    \mathbb{P}(\eta_0)f=\sum_{k=1, \, 2}\int_{-\eta_0}^{\eta_0} \langle \cF_y(f) (\cdot, \eta), g_k^{*}(\cdot,\eta)\rangle_{_{Y_{a,\eta}\times Y_{a,\eta}^{*} } }\,g_k(\cdot,\eta) \, e^{i y\eta}\, \d \eta
\end{align}
which projects the elements in $X_a$ onto the spectral subspaces associated to the continuous family of eigenvalues $\{\lambda(\pm \eta)\}_{|\eta|\leq \eta_0}.$ We refer to \eqref{def-bracketYa} for the definition of $ \langle \cdot \rangle_{Y_{{a, \eta}}\times Y_{a,\eta}^{*}}.$
Define $\mathbb{Q}(\eta_0)=\Id-  \mathbb{P}(\eta_0).$
It stems from the symmetries of basis and dual basis that 
the projections $\mathbb{P}(\eta_0)$ and $\mathbb{Q}(\eta_0)$ 
indeed send real functions to real functions.
Our  second   result  which is the main technical  one is the obtention of  uniform resolvent bounds on  the space $\mathbb{Q}(\eta_0)X_a$,  which is an infinite dimensional spectral subspace supplementary to  the continuous resonant eigenspaces. %generalized eigenspaces,
These estimates will allow  to get exponential decay  for the semigroup on this  space. %linear time asympototics of

By  using  the notation $ \| \cdot \|_{B(X)} $
for the operator norm for linear operators on $X$ and $\rho(\cdot\,; X)$ for the resolvent set of an operator on $X$,  we get the following:
\begin{thm}\label{thm-resol-L}
Let $0<a=\hat{a}\ep < \f{\sqrt{3}}{4}\ep$ and $X_a=\colon L_a^2\times 
 H_{a,*}^{{1}/{2}}.$ 
 There exist $\ep_0>0, \beta_0=\beta_0(\hat{a})>0, \hat{\eta}_0=\hat{\eta}_0(\hat{a})$ such that for any $0<\ep\leq \ep_0, \, 0<\beta\leq\beta_0,\, $ we have 
 \beqs 
 \Omega_{\beta,\,\ep}=\colon \big\{ \lambda\in \mathbb{C}\,\big|\, \Re \lambda > -\beta \ep^3\big\} \subset \rho (L; \mathbb{Q}(\ep^2\hat{\eta}_0)X_a).
 \eeqs
Moreover, there is $C_0=C_0(\hat{a}, \hat{\eta}_0, \beta_0, \ep_0)>0,$ such that for any $\lambda \in \Omega_{\beta, \ep},$  the operator $(\lambda-L)$ is invertible on the space $\mathbb{Q}(\ep^2\hat{\eta}_0)X_a$ and
 \beq\label{uni-resol}
\|(\lambda-L)^{-1} \|_{B(\mathbb{Q}(\eta_0)X_a)}\leq C_0.
 \eeq
\end{thm}

In view of the above uniform resolvent estimate \eqref{uni-resol}, we have the following 
 semigroup estimate 
 which is a  consequence of the Gearhart-Prüss theorem \cite{Gearhart,Pruss,Proof-GP}: 
 \begin{cor}
   Let $\ep_0, \beta_0, \hat{\eta}_0$ be defined in Theorem \ref{thm-resol-L}. For any $\ep\leq \ep_0,\,\beta<\beta_0,$ it holds that:
 \beq\label{semigroup}
 \|e^{tL}\mathbb{Q}(\ep^2\hat{\eta}_0)\|_{B(X_a)}\lesssim e^{-\beta\ep^3 t}.
 \eeq
 \end{cor}

 As a consequence of Theorem \ref{thm-resol-L}, and its proof,  we can also get the following spectral stability result in  unweighted spaces: %$X_0=L^2(\mR^2)\times H^{1/2}_{0,*}(\mR^2):$
\begin{thm}\label{thm-spectral-unweighted}
Let $L(\eta)$ be defined in \eqref{def-Leta} and denote $Y_0(\eta)=L^2(\mR)\times H^{1/2}_{0,*,\eta}(\mR).$
For every $\ep$, $0 <  \ep \leq \ep_{0}$, the spectrum of $L(\eta)$ %the linearized operator $L$ 
in the space  $Y_0(\eta)$ 
is contained  in the imaginary axis.
\end{thm}

Note that one has to be careful with the space  $X_a=L_a^2\times H_{a,*}^{1/2}$ and the definition of the Dirichlet-Neumann operator on it. Indeed, due to the  possible exponential growth,   this space  is not contained
 in the Schwarz space of distributions and hence one cannot use   the Fourier transform. 
 We can nevertheless  use  the   Fourier transform of $e^{ax} f$ but not the one of $f$ when $f$ is in $L^2_{a}$.
 In order to study $L$ defined in  \eqref{def-L} on the weighted space $X_a=L_a^2\times H_{a,*}^{1/2},$
 one thus need  to be careful with the definition  and the properties of 
 the Dirichlet-Neumann operator $G[\zeta_c]$ 
 on $H_{a,*}^{1/2}.$
  %which is further based on solving the elliptic problem in exponentially weighted spaces.  
  It will actually be more convenient to work on an unweighted space
 by considering  the transformed operator $G_a[\zeta_c]=\colon e^{ax}\, G[\zeta_c]\, e^{-a x}$ and study it in the usual Sobolev space:
%It is also useful to define the unweighted space 
\begin{align}
  H_{\star}^{{1}/{2}}(\mR^2)= \bigg\{ f  \,\big|\, \|f\|_{ H_{\star}^{{1}/{2}}(\mR^2)}= \colon \bigg(\int \bigg(\big|\f{\nabla}{\langle D \rangle^{1/2} } f\big|^2+a^2 |f|^2 \bigg)(x,y)  \, \d x\d y\bigg)^{\f{1}{2}} <+\infty\bigg\}. 
\end{align}
Note that $H_{\star}^{{1}/{2}}(\mR^2)$ still depends on $a$ but we neglect this dependence in the notation  in order to distinguish it with the weighted space $H_{a,\star}^{{1}/{2}}(\mR^2).$ 
An equivalent definition of $G_a[\zeta_c]$ 
is that, for any $f\in H_{*}^{1/2},$ 
\beq\label{def-DN-T}
G_{a}[\zeta_c]f(x)=\sqrt{1+|\na \zeta_{c}|^2}\big(\p_z\Phi_{a}-\p_x\zeta_c\,(\p_x-a)\Phi_{a}\big),
\eeq
where $
\Phi_{a}(x,z)$ solves the following equation in $\Omega%^{\ep}
=\mR^2\times[-1, \zeta_c]:$
\beq \label{ellipticpb}
-\big((\p_x-a)^2+\Delta_{y,z}\big)\Phi_{a}=0, \quad \Phi_{a}|_{z=\zeta_c}=f, \quad \p_z \Phi_{a}|_{z=-1}=0.
\eeq
When the surface is flat, that is when $\zeta_c = 0,$ we can  find explicitly the  solution of the above problem by using  the Fourier transform which is now legitimate: 
\begin{align}\label{def-e0f}
    E[0](f)(\cdot, z)=\f{\cosh\big({\mu_a(D)(z+1)}\big)}{\cosh{\mu_a(D)}} f\,, \quad \text{ with \, }\mu_a(D)=\sqrt{(D_x+ i a)^2+D_y^2}\, ,
\end{align}
where $D$ stands for $(D_{x}, D_{y})$ and 
for a complex number $c,$ 
we define  $\sqrt{c}$ as the limit of the square root with positive real part
of the complex numbers $c+i\delta$ when $0<\delta\rightarrow 0^{+}.$ %we  refer to page 11 for the definition of the square root of a complex number.
%define the square root of a complex number $a$ as the limit of square root of $a+\delta,$ we always take the branch with positive real part if it is not zero 
Thanks to \eqref{def-e0f}, we derive that
$G_a[0]= (\mu_a \tanh \mu_a)(D)\,.$
By using the notation $\cS=\mR^2\times [-1,0]$, 
 note that  for any positive constant $a,$ the solution $E[0]f$ belongs only to the anisotropic homogeneous space 
$\dot{H}_{x,z}^1(\mathcal{S})=\big\{f\,|\, \p_x f, \p_z f \in L^2(\cS)\big\}.$  However,
%does not belong to the homogeneous space  since 
its derivative in the transverse variable $y$ cannot be bounded in $L^2(\cS).$ Indeed, the symbol $\mu_a^{-1}(\xi, \eta)=\big(\sqrt{(\xi+ia)^2+\eta^2}\big)^{-1}$ may  exhibit singularity around  nonzero points $(\xi, \eta)=(0, \pm a).$
  This degeneracy is  a source of  new difficulties
 in the study of the transverse stability, in particular, we cannot  choose our functional framework as $L_a^2 \times \sqrt{G_a[0]}^{-1}L_a^2$ as in \cite{Pego-Sun} where the asymptotic linear stability is studied for one-dimensional water waves.

To get useful properties of the Dirichlet-Neumann operator for a non-zero surface, as  it is classically done, we can  use a change of variable to reduce the problem to 
the fixed domain $\cS$ and investigate the solvability of the difference between the resulting unknown and $E[0]f.$ An important observation is that, although the corresponding linear operator is not elliptic, one can  take benefits of 
the Poincar\'e inequality (due to the zero upper boundary condition) when  studying the perturbation  to control its $L^2(\cS)$ norm by the $\dot{H}^1(\cS)$ norm. Another key point is that 
the source term does not contain any $\p_y$ derivative of $E[0]f$ which does not belong to $L^2(\cS).$ %consisting 
This provides the solvability and uniqueness of the solution to \eqref{ellipticpb} in the new variables in the space $H_{0,up}^1(\cS)+E[0]f.$ We refer to Lemma \ref{lem-elliptic} for more details.

The rough study of $G_a[\zeta_c]$ above also motivates us to consider the transformed linear operator
\beq\label{def-La}
 {L}_a =\colon e^{ax} L \,e^{-a x}=\left( \begin{array}{cc}
   (\p_x-a)(d_c\cdot)  &  G_a[\zeta_{c}]   \\[5pt]
 -  w_c& d_c(\p_x-a)
\end{array}\right). \,\,  %\text{ with } \,
\eeq
Since the space $X=\colon L^2(\mR^2)\times H_{\star}^{{1}/{2}}(\mR^2)$ can be identified with  $e^{ax}X_a,$
we see that the study of $L$ on the weighted space $X_a$ is equivalent to the study of $L_a$ on the unweighted space $X.$ Consequently, Theorem \ref{thm-resol-L} is equivalent to the following theorem: 
\begin{thm}\label{thm-resol-La}
Let $X=\colon L^2(\mR^2)\times H_{\star}^{{1}/{2}}(\mR^2), \, \mathbb{Q}_a(\eta_0)=e^{ax}\,  \mathbb{Q}(\eta_0) \, e^{-ax}.$ There exist $\ep_0>0, \beta_0=\beta_0(\hat{a})>0, \hat{\eta}_0=\hat{\eta}_0(\hat{a})>0, C_0>0$ such that for any $0<\ep\leq \ep_0, 0<\beta\leq \beta_0,\, $ and any $\lambda \in \Omega_{\beta,\,\ep},$ %such that $\Re \lambda > -\beta\ep^3,$
 the operator
 %There exist $\ep_0>0, \beta_0>0, C_0>0$ such that for any $\beta<\beta_0,$ for any $\lambda$ satisfying $\lambda>-\beta \ep^3,$
 $(\lambda-L_a)$ is invertible on the space $\mathbb{Q}_a(\ep^2\hat{\eta}_0)X$ and
 \beq\label{uni-resol-La}
\|(\lambda-L_a)^{-1}\|_{B(\mathbb{Q}_a(\ep^2\hat{\eta}_0)X)}\leq C_0\, .
 \eeq
\end{thm}

\subsection{Outline of the proof and main difficulties }
In this subsection, we explain the main issues and ideas for  the proof of  the existence of the resonant modes in Theorem \ref{thm-resolmodes} and  the resolvent estimate in Theorem \ref{thm-resol-La}.
\subsubsection{Resonant modes}
The resonant modes $$\big(\lambda(\eta), e^{i y\,\eta}U(\cdot,\eta)\big)_{-\eta_0\leq \eta\leq \eta_0}$$ are searched by looking for solutions of 
the eigenproblems
\beqs 
(\lambda(\eta)-L(\eta))\,U(\cdot, \eta)=0.
\eeqs
As mentioned previously, $L(0)$ corresponds to 
the linearized operator of one-dimensional water waves about  the 1-d soliton $(\zeta_c, \vp_c),$ which has a nontrivial generalized kernel of dimension two. 
To find the curves $\lambda (\eta)$ bifurcating from $\eta=0$ and the eigenfunctions $U(\cdot, \eta),$ we use the classical Lyapunov-Schmidt reduction approach. While featuring some algebraic differences, the proof share some similarities with the %rely on 
 the arguments presented in \cite{Mizumachi-BL-linear} where transverse linear stability of line solitary waves  is established for the Benney-Luke equation.  

\subsubsection{Uniform resolvent estimate}
We now give more details about  the proof of the uniform resolvent estimate \eqref{uni-resol-La}, which is  the core of this paper.
In order to get   \eqref{uni-resol-La}, the first thing we have to detect is the presence of  some damping effects once we work in the weighted space $X_a$ due to the sign of the group velocity  for the  constant coefficients part of  the original linearized operator $L,$ or equivalently the transformed operator $L_a$ on the unweighted space $X.$ This has been used since the pioneer work of Pego and Weinstein \cite{Pego-Weinstein-kdv,Pego-Weinstein-Boussinesq} on the asymptotic stability of KdV equation and Boussinesq equations. To be more precise, consider the constant operator $$L_a^0=\colon 
 \left( \begin{array}{cc}
   \p_x-a  &  G_a[0]   \\[5pt]
 -  \gamma& \p_x-a
\end{array}\right), $$ which is obtained by setting $(\zeta_c, \vp_c)=0$  in the definition of $L_a$ \eqref{def-La}. The two eigenvalues of its symbol  are 
\beqs 
\lambda_{\pm}^{0}=\colon i(\xi+i a)\pm \sqrt{-\gamma \, \mu_a \tanh \mu_a(\xi, \eta)}\,, \, \qquad  \big(\,\mu_a(\xi, \eta)=\sqrt{(\xi+ i a)^2+\eta^2}\,\big)\, .
\eeqs
Some algebraic computations show that for any $(\xi, \eta)\in \mR^2,$ there exists a positive constant $\beta(\hat{\alpha})=\cO(\hat{\alpha}),$ such that 
\beqs 
\Re \lambda_{\pm}^{0} \leq - \beta(\hat{\alpha})\, \ep^3\, .
\eeqs
Consequently, for any $\lambda$ such that $\Re \lambda \geq -\beta(\hat{\alpha})\, \ep^3/2,$ it holds that 
\beqs 
\big|(\lambda-\lambda_{\pm}^0)^{-1}\big|\leq \f{2}{\beta(\hat{\alpha}) \ep^3}\, , %(\beta(\hat{\alpha}) \ep^3)^{-1},
\eeqs
which enables us to conclude that  $(\lambda-L_a^0)$  is invertible and  that  the resolvent bounds 
\beqs 
\sup_{\Re \lambda \geq -\beta(\hat{\alpha})\, \ep^3/2} %\in \Omega_{\beta, \ep}}
 \|(\lambda-L_a^0)^{-1}\|_{B(X)}< +\infty \, 
\eeqs
holds. The next step is to obtain the same  result \eqref{uni-resol-La}  for $L_a$ in which the terms depending on  the solitary waves $(\zeta_c, \vp_c)$ are  nonzero but are small. %Clearly, it depends on how far between $L_a$ and the constant operator $L_a^0.$
The main difficulties in order to obtain this result in the water-waves case are the following.

%One difficulty arises from the quasilinear structure of the linear system. 

In the case of the water waves system which is fully nonlinear,  the part of the system involving coefficients
depending on the solitary wave cannot be seen as a lower order small perturbation
 as in the study of the linearization of semilinear problems, 
see for instance \cite{Mizumachi-BL-linear}. 
More precisely, here  the appearance of $v_c \p_x$ in the diagonal terms of $L_a$ results in operators that are not bounded and  
the difference between $G_a[\zeta_c]$ and $G_a[0]$ is not a bounded operator from $H_{*}^{\f12}(\mR^2)$ to $L^2(\mR^2).$ To 
address this issue, we cut frequencies and use  %semiclassical analysis
 pseudo-differential calculus
and energy estimates in the high-frequency regions. Further details will be provided later. 

Another difficulty lies in  the presence of  the Dirichlet-Neumann operator $G_a[\zeta_c].$ If $\zeta_c$ is $ C^{\infty}(\mR)$ smooth which is the case in the current study,  one can expect to adapt well-known results for the case $a=0$ (see \cite{Alazard-Metivier}
for example) to get that 
$G_{a}[\zeta_c]$ is  a pseudo-differential operator with symbol 
\beqs 
\lambda_{\zeta_c}(x, \xi, \eta)=\lambda_{\zeta_c}^1(x, \xi, \eta)+\lambda_{\zeta_c}^0(x, \xi, \eta)+\cdots \,,
\eeqs
where the principal symbol  is 
\beqs%\label{def-prsy}
\lambda
^1_{\zeta_c}(x, \xi, \eta)=\sqrt{(\xi+ia)^2+\eta^2 (1+(\p_x\zeta_c)^2)}\,.
\eeqs
However, there is no  exact formula for the symbol $\lambda_{\zeta_c}$. 
We thus need to use  approximations of $\lambda_{\zeta_c}$ in conjunction with appropriate perturbation arguments. As will be clear in the proof, 
here an approximation $\lambda_{\zeta_c}^{app}$ is  admissible if $\op(\lambda_{\zeta_c}-\lambda_{\zeta_c}^{app})$ is a bounded operator from $H^{1/2}_{*}(\mR^2)$ to $L^2(\mR^2)$ with  operator norm sufficiently small compared to $\ep,$ for instance
%if the operator ofit is bounded  from $H^{\f12}_{*}(\mR^2)$ to  
$$\|\op(\lambda_{\zeta_c}-\lambda_{\zeta_c}^{app})\|_{ B( H_{*}^{1/2},\, L^2%(\mR^2)
)}\leq C\ep^{\alpha}, \quad (\alpha>1).$$
It turns out that the needed quality of the approximation, that is to say  the minimal  $\alpha$  that we can allow, 
is different in  different  frequency regions. 

On the  one hand,
when focusing on  uniform (in $\ep$) high transverse frequencies $\{|\eta|\geq 2\}$, one may expect that the principal symbol $\lambda _{\zeta_c}^1$ could be used as an  admissible approximation. Nevertheless, although 
 $\op(\lambda_{\zeta_c}^1)$ approximates well the operator $G_{a}[\zeta_c]$ in terms of  regularity, it is far from being a good approximation  in terms of the size of the operator norm within $B(H^{1/2}(\mR^2), L^2(\mR^2)).$ Indeed,  the  operator norm of the difference is only $\cO(1).$ %since one only has:\beqs \|\tilde{\bI}_2(D_y)(G_{a}[\zeta_c]-\op(\lambda^1_{\zeta_c}))\|_{B( H_{*}^{1/2},\, L^2%(\mR^2))}\lesssim 1, \, %\,(\text{ rather than } \ep^{\alpha})\eeqswhere $\tilde{\bI}(\cdot)$ is the characteristic function on $\{|\cdot|\geq 2\}.$
This is not enough  to prove  resolvent bounds by  perturbation arguments. Nevertheless, we will show in
Lemma \ref{lem-error} that an admissible approximation can  be obtained by using a modification  
of the principal symbol $$\tilde{\lambda}_{\zeta_c}^1=\colon \lambda
^1_{\zeta_c}\tanh \lambda
^1_{\zeta_c}.$$ The operator norm of the difference between $G_{a}[\zeta_c]$ and $\op(\tilde{\lambda}_{\zeta_c}^1)$ is now  of order $\cO(\ep^2).$ We remark that unsurprisingly, such a  modification is needed due  to the fact that we are considering water waves in a domain with finite depth.

On the other hand,  in the region of lower  
transverse frequencies $\{|\eta|\leq 2\},$ the damping effect  is  quite  weak, it thus necessitates a more accurate approximation of  %it is more convenient to approximate
the Dirichlet-Neumann operator. %It turns out that 
Here, we shall adapt in the setting of weighted spaces its first order expansion in terms of $\zeta_c,$  which is a pseudo-differential operator with the symbol \cite{Lannes} :
\beq\label{def-symb-inter}
{\lambda}_{\zeta_c}
    ^{app}(x,\xi,\eta)=%\sqrt{
    \mu_a\tanh \mu_a (\xi,\eta)
    + \zeta_c(x) \mu_a^2 (\tanh^2 \mu_a -1) (\xi,\eta), \quad  \mu_a(\xi,\eta)=\sqrt{(\xi+ia)^2+\eta^2},
\eeq
the error belonging to $B( H_{*}^{1/2},\, L^2)$ with  norm of order $\cO(\ep^3).$

We are now in position to outline the
proof of uniform resolvent bounds \eqref{uni-resol-La}.  We first use  frequency cutoffs and  different  methods to establish the resolvent bounds in different frequency regions. We divide  the frequency space $(\xi, \eta)\subset \mathbb{R}^2$ into three regions, corresponding to (uniform) high, intermediate, and low \textbf{transverse} frequencies:
\begin{align}
\label{defregions}
    R^{UH}=\big\{ (\xi,\eta)\big|\, |\eta|\geq 2 \big\}, \quad
    R^I=\big\{ (\xi,\eta)\big|\, A\ep^2\leq |\eta|\leq 2 \big\}, \quad  R^L=\big\{ (\xi,\eta)\big|\,|\eta|\leq  A\ep^2 \big\}, 
\end{align}
where $A$ is a sufficiently large constant satisfying $A^2\ep \leq 1.$ 
 We shall call the union of the  uniform high frequency and intermediate frequency regions  the high frequency region. Note that at this stage we only cut the transverse frequencies and the longitudinal frequency $\xi$ is just a fixed parameter. We shall perform further localization in $\xi$ within the different regions when needed.

$\bullet$ \textbf{Uniform high transverse frequencies.}
In the uniform high transverse frequency region $R^{UH},$ 
we define a modified operator $\tilde{L}_a$ via replacing the Dirichlet Neumann operator in $L_a$ by $\op(\tilde{\lambda}_{\zeta_c}^1).$ We then diagonalize it  by using   pseudo differential  calculus 
and  thus reduce the existence of the resolvent problem as well as the uniform resolvent bounds to the  study of $\lambda-\lambda_{\pm}^1(x, \xi, \eta),$ where   \beqs% \label{deflpm}
\lambda_{\pm}^1=i \bigg(d_c(x)(\xi+ia)\pm \sqrt{w_c}(x)\sqrt{\lambda^1_{\zeta_c}\tanh \lambda^1_{\zeta_c}}(x,\xi,\eta)\bigg).
\eeqs 
It turns out that when $|\eta|\geq 2,$ for any $(x, \xi)\in \mR^2,$  any $\lambda \in \Omega_{\beta,\,\ep}$ with $\beta$ small enough,   it holds that 
$\Re (\lambda- \lambda_{\pm}^1(x,\xi,\eta))\gtrsim \ep.$
This, together with Lemma \ref{lem-error},  allows us to establish the invertibility of $(\lambda-L_a)$ and the corresponding resolvent bounds in this region by  perturbation arguments.

$\bullet$ \textbf{Intermediate transverse frequencies.} 
The study in the intermediate transverse frequency region $R^I$ is  more involved due to a  singularity in the previous  diagonalization process. %As mentioned previously, since the damping effects in weaker in this region, it necessitates a more accurate approximation of the Dirichlet Neumann operator by using 
 Indeed,  by using   $\op({\lambda}_{\zeta_c} ^{app})$ to approximate  $G_a[\zeta_c],$ we would  involve an operator of the form $\op\big(1/\sqrt{{\lambda}_{\zeta_c} ^{app}}\big)$ which 
 unfortunately does not belong to $B(L^2, H_{*}^{1/2}).$
This is due to  the singularity of the symbol $\mu_a(\xi, \eta)$ in the region
$$\cS_{K, \delta}=\bigg\{(\xi,\eta)\in \mR^2\big|\,|\xi%+i a
|\leq K\ep, \, \delta\leq 
\f{|\eta|}{\sqrt{\xi^2+a^2}}\leq 2\bigg\}, \qquad ( \,0<\delta<1 \text{ small },\, K\sim \sqrt{A} % K(K+1) \in [A/2, A]  
).$$ 
%Since
However, since this region is indeed in low frequencies both in %$x$ and $y,$
longitudinal and transverse variables, we will show that we   approximate $G_a[\zeta_c]$ by simply $-\Delta$ and perform directly  energy estimates to obtain the desired resolvent bounds. 
Since we would have to localize in  the longitudinal  frequencies and study different sub-regions separately, we found  more convenient to adopt in this whole region $R^I$ an energy-based approach, achieved through the design of appropriate energy functionals. %prove the resolvent estimates by performing the energy estimates 

$\bullet$ \textbf{Low transverse frequencies.} 
In the low transverse frequency region $R^L,$ as in the previous case, in order not to lose derivatives,
it requires a further splitting involving frequencies dual  to the $x$ variable:  
$$R_H^L=\{|\xi|\geq K\ep, \, |\eta|\leq A\ep^2\} \qquad   R_L^L=\{|\xi|\leq K\ep, \, |\eta|\leq A\ep^2\}.$$ 
In the region $R_H^L,$ we perform energy estimates by designing an energy functional consistent with the symbol approximation \eqref{def-symb-inter} of the  Dirichlet Neumann operator. In the region 
$R_L^L,$ we take benefits of the KP-II approximation (see Subsection \ref{subsec-KP-IIapprox}) and the resolvent bounds for the linearized operator of KP-II equation around line KdV soliton proved in \cite{Mizumachi-KP-nonlinear}. Due to the weak damping effects in this region (which is of order $O(\ep^3)$), instead of using the Fourier projector as in the study of the  region $R^I,$ we use  smooth cut-offs to split the longitudinal  frequencies. 
Let us remark finally that the arguments used in this region can provide  a  proof directly in cartesian coordinates of the linear asymptotic stability of  the 1d solitary waves in the  one-dimensional water waves system, which has been previously obtained  in \cite{Pego-Sun} in   holomorphic coordinates.

\textbf{Organization of the paper:}
We establish the existence of continuous resonant modes, as stated  in Theorem \ref{thm-resolmodes}, in Section 2. This section starts  with a formal derivation of the KP-II equation from  the water waves in the long-wave regime. Uniform resolvent estimates stated in Theorem \ref{thm-resol-La} are then derived in Sections 3 and 4, in the high (uniform high and intermediate) and low transverse frequency regimes, respectively. The proof of spectral stability in the unweighted space stated in Theorem \ref{thm-spectral-unweighted} is performed  in Section 5.
Additionally, various technical results used in  the main proofs are presented in the appendix. Specifically, in Appendix \ref{appendix-DN}, we prove several properties for the transformed Dirichlet-Neumann operator $G_a[\zeta_c]$, which are  of independent interest.

%We prove the existence of continuous resonant modes stated in Theorem \ref{thm-resolmodes} in Section 2, where we begin with a formal derivation of the KP-II equation from the water waves in the longwave regime. The uniform resolvent estimates are established in Section 3 and Section 4, in the high (uniform high and intermediate) and low transverse frequency regimes respectively. Finally, we show various results in appendix that are used along the main proof. Specifically, we prove in Appendix \ref{appendix-DN} various properties for the transformed Dirichlet-Neumann operator $G_a[\zeta_c]$ which maybe of independent interest.
%\textbf{Definition:} 
%we define $\sqrt{c}=\lim_{\delta\rightarrow 0^{+}}\sqrt{c+\delta}$ as the square root of $c$ where inside of the limit, we take the branch with positive real part when we take $\sqrt{\,

\section{Spectral study  
in the regime of low transverse frequency }
In this subsection, we prove Theorem \ref{thm-resolmodes}. To be more specific, we will establish that for sufficiently small values of $\eta_0$, there exists a spectral curve $\lambda(\eta)\in C([-{\eta_0}, {\eta_0}])$ and corresponding eigenmodes $U(\cdot, \eta)\in Y_{a,\eta}=L_a^2(\mR)\times H^{1/2}_{*,a,\eta}(\mR),$ such that
\beq \label{resolnanteq-1}
L(\eta)\,U(\cdot, \eta)=\lambda(\eta)\,U(\cdot, \eta).
\eeq
Given that $L(\eta)=e^{-iy\eta}\, L\, e^{iy\eta}$, this yields a set of continuous resonant modes $$\big(\lambda(\eta), e^{i y\,\eta}U(\cdot,\eta)\big)_{-\eta_0\leq \eta\leq \eta_0}$$ for the linearized operator $L$.
 
\subsection{KP-II approximation}\label{subsec-KP-IIapprox}
As mentioned in the introduction, it is well known that the Korteweg-de Vries (KdV) equation and the Kadomtsev-Petviashvili II (KP-II) equation are  effective approximate models for describing the  water wave dynamics in the long-wave region, %specifically 
in one and two dimensions respectively. %To see this
 %formally demonstrated
We thus expect that  two dimensional water waves system linearized  about  the one dimensional solitary water waves $(\zeta_c, \vp_c)$ can be approximated by  the linearized  KP-II equation about   the line KdV soliton.  %\eqref{def-L}
This can be rigorously  shown by implementing the following scaling:
\beq \label{scaling}
\begin{aligned}
    \hat{t}=\ep^3 t, \quad \, \hat{x}=\ep x, \, \quad \hat{y}=\ep^2 y, \qquad (\hat{\zeta}, \hat{\phi})(\hat{t}, \hat{x})=(\zeta, \phi)(t, x), 
\end{aligned}
\eeq
Indeed, if $(\zeta, \phi)^t$ solves the linearized equation \eqref{def-L}, then $(\hat{\zeta}, \hat{\phi})$ solves:
\beq\label{scaled-eq}
\left\{ \begin{array}{l}
      \ep^3 \p_{\hat{t}}   \hat{\zeta} - \ep  \p_{\hat{x}} (\hat{d}_{\ep}\hat{\zeta})-\hat{G}[\ep^2\hat{\zeta}_{\ep}] \hat{\phi} =0,  \\[3pt]
   \ep^3 \p_{\hat{t}}   \hat{\phi} - \ep \hat{d}_{\ep} \p_{\hat{x}} \hat{\phi}+ \hat{w}_{\ep}  \hat{\zeta} =0,
\end{array}\right.
\eeq
where 
\beq \label{def-dwep}
(\hat{v}_{\ep}, \hat{Z}_{\ep}, \hat{\zeta}_{\ep})(\hat{t}, \hat{x})=\ep^{-2}(v_c,Z_c, \zeta_c)(t,x) , \quad \hat{d}_{\ep}= 1-\ep^2 \hat{v}_{\ep}, \, \quad 
\hat{w}_{\ep}=\gamma-\ep^3 \hat{d}_{\ep} \p_x \hat{Z}_{\ep}.
\eeq
Moreover, the scaled Dirichlet-Neumann operator is defined as:
\beqs
\hat{G}[\ep^2\hat{\zeta}_{\ep}]f= \p_{{z}}\hat{\Phi}-\ep^4\p_{\hat{x}}\hat{\zeta}_{\ep}\p_{\hat{x}}\hat{\Phi} 
\eeqs 
where $\hat{\Phi}$ solves the elliptic problem
 in $\Omega=\mR\times[-1, \ep^2\hat{\zeta}_{\ep}]:$
$$\big(\ep^2\p_{\hat{x}}^2+\ep^4 \p_{\hat{y}}^2+\p_{{z}}^2\big)\hat{\Phi} =0, \quad \hat{\Phi} |_{z=\ep^2\hat{\zeta}_{\ep}}=f, \quad \p_z \hat{\Phi}|_{z=-1}=0.$$ 

Let us now skip the subscript $\hat{}$ for the sake of notational convenience. 
Plugging $\zeta=-w_{\ep}^{-1}\big(\ep^3\p_t-\ep d_{\ep} \p_x\big)\phi$ into the first equation of \eqref{scaled-eq}, we find that $\phi$ solves
\begin{align}\label{eq-scale-s}
  %\ep^4\big[
 \bigg( \f{ \ep^2}{w_{\ep}}\p_t^2-\f{2d_{\ep}}{w_{\ep}}\p_t\p_x- \p_x\big(\f{d_{\ep}}{w_{\ep}}\big)\p_t+\ep^{-2} \p_x \big(\f{d_{\ep}^2}{w_{\ep}}\p_x\big)+\ep^{-4}\hat{G}[{\ep^2\zeta_{\ep}}]\bigg)\phi=0.
\end{align}
We will do expansions with respect to $\ep$ in  the above equation. The expansion is usually performed by  assuming that a high enough Sobolev norm of the solution is bounded
on the appropriate time scale. Here we shall assume that the problem  is localized  in a bounded  frequency region with respect
to both variables, so that we  assume that all  the differential operators 
with positive order are bounded operators in $L_{\hat{a}}^2(\mR^2).$ We shall indeed perform this localization when we
will use this approximation argument.

It follows from  the definition that 
\beqs 
\f{d_{\ep}}{w_{\ep}}=-1+\cO(\ep^2), \qquad \p_x \big(\f{d_{\ep}^2}{w_{\ep}}\p_x\big)= \p_x^2+\ep^2\p_x((1-2v_{\ep})\p_x)+\cO(\ep^3).
\eeqs
Moreover, on the one hand, when the surface is flat, the Dirichlet-Neumann operator $\hat{G}[0]$ has the form
\begin{align*}
    \hat{G}[0]=\sqrt{-\ep^2(\p_x^2+\ep^2\p_y^2)}\tanh \sqrt{-\ep^2(\p_x^2+\ep^2\p_y^2)}\, 
\end{align*}
and thus has the expansion in a bounded  frequency region:
\begin{align*} 
\hat{G}[0]&=-\ep^2(\p_x^2+\ep^2\p_y^2)-\f13 \ep^4(\p_x^2+\ep^2\p_y^2)^2+\cO(\ep^6)\\
&=   -\ep^2\p_x^2-\ep^4(\p_y^2+\f13 \p_x^4)+\cO(\ep^6).
\end{align*}
On the other hand, by the expansion of the Dirichlet-Neumann operator in terms of the small amplitude waves $\ep^2\zeta_{\ep}$ (refer to Section 3.6.2 of \cite{Lannes}), we can  write,
in  a bounded  frequency region,
\begin{align*}
      \hat{G}[{\ep^2\zeta_{\ep}}]&=\hat{G}[0]-\hat{G}[0] (\ep^2\zeta_{\ep})\hat{G}[0]-\ep^4 \p_x( \zeta_{\ep}\p_x)+\cO(\ep^6)\\
      &=%+\cO(\ep^6)
    -\ep^2\p_x^2-\ep^4(\p_y^2+\f13 \p_x^4)-\ep^4\p_x(\zeta_{\ep}\p_x)+\cO(\ep^6).
\end{align*}
%\begin{align*}   \hat{G}[{\ep^2\zeta_{\ep}}]%=\hat{G}[{\ep^2\zeta_{\ep}}]-\hat{G}[0]+\hat{G}[0]\end{align*}
Consequently, the leading order of the equation \eqref{eq-scale-s} is:
\beqs 
-2\bigg(\p_t-\p_x\big(1-(2v_0+\zeta_0)\cdot\big)+\f13 \p_x^3+\p_y^2\p_x^{-1}\big)\bigg)\p_x\phi=0,
\eeqs
where $v_0=\zeta_0=\Psi_{\kdv}=\sech^2(\f{\sqrt{3}}{2}\cdot)$ is the soliton profile of the KdV equation:
\beqs
\pt \Psi- \p_{x} \Psi+3\Psi\p_x\Psi+\f13 \p_x^3 \Psi =0.
\eeqs
We thus find the linear operator arising in   the  KP-II equation linearized about  the line soliton $\Psi_{\kdv}:$
\beq \label{linear-kpII}
L_{\kp}= \p_x\big(1-\f13 \p_x^2-3\Psi_{\kdv}\cdot\big)%-\f13 \p_x^3
-\p_y^2\p_x^{-1}.
\eeq
\subsection{Existence of resonant modes--proof of Theorem \ref{thm-resolmodes} }% }%Spectral study %of linearized operator in the regeim of low transverse frequency}%$L(\eta)$ for small} %{Resonant modes} 
%In this subsection, we focus on the low transverse frequency and search theresonant modes of the operator $L$ in the weighted space $X_a.$ %from which we obtain the corresponding modes for the transformed operator $L_a$ in unweighted space $X.$ Specifically, where  $V(\cdot,\eta)\in \cY_a^2=L_a^2\times H^{1/2}_{*,a,\eta}$, $\Re \lambda(\eta)<0$ and $|\lambda(\eta)|=\cO(|\eta|)$ for $|\eta|\leq \eta_0.$ Denotethen \eqref{resonant-modes} is equivalent to 
%\begin{proof}

Denote $U(\cdot, \eta)=(\zeta, \phi)^t(\cdot, \eta),\,\, G_{\eta}[\zeta_c]=e^{-iy\eta}G_{\eta}[\zeta_c]e^{iy\eta},$  
then $\eqref{resolnanteq-1}$ can  be written as
\beqs\label{spectralpb}
\left\{ \begin{array}{l}
\lambda  {\zeta} - \p_{{x}} ({d}_{c}{\zeta})-{G}_{\eta}[\zeta_c] {\phi} =0\,,  \\[3pt]
   \lambda   {\phi} - d_c \p_{{x}} {\phi}+{w}_{c}  {\zeta} =0\,,
\end{array}\right.
\eeqs
or equivalently
\beqs 
\big(\lambda  - \p_{{x}} ({d}_{c}\cdot)\big)\big(w_c^{-1} \big(\lambda  - {d}_{c}\p_{{x}} \big)\big)\phi+G_{\eta}[\zeta_c] {\phi}=0\,.
\eeqs
Let us define $$\lambda=\ep^3 \Lambda, \,\, \quad \hat{\eta}=\ep^{-2} \eta, \, \, \quad \hat{\vp}(\hat{x}, \hat{\eta})=\vp(x, \eta), \quad  \hat{\psi}(\hat{x}, \hat{\eta})=(\p_{\hat{x}}\hat{\phi})(\hat{x}, \hat{\eta}).$$
By the change of variable and skipping the
subscript $\hat{}$ for clarity, we find in parallel to \eqref{eq-scale-s}
\beq\label{eigenpb-real}
\cT(\ep,\eta, \Lambda) \psi=\colon \big[L_0(\ep,\eta)+\Lambda (\eta) L_1(\ep)% \eta)
+ \Lambda^2(\eta)  L_2(\ep)\big]\psi=0\, ,
\eeq
where 
\beq\label{Ljepeta}
\begin{aligned}
   &\quad L_0(\ep,\eta)= \ep^{-2}\p_x \bigg(\f{d_{\ep}^2}{w_{\ep}}\cdot\bigg)+\ep^{-4}\hat{G}_{\ep^2\eta}[{\ep^2\zeta_{\ep}}]\p_x^{-1}, \\
  &   L_1(\ep)=-\p_x \bigg(\f{d_{\ep}}{w_{\ep}}\p_x^{-1}\bigg)- \f{d_{\ep}}{w_{\ep}}, \qquad L_2(\ep)=\ep^2  w_{\ep}^{-1}\p_x^{-1}.
\end{aligned}
\eeq
We refer to \eqref{def-dwep} for the definition of $ d_{\ep}$ and $w_{\ep}.$ Let us note that when $\ep=0,$ 
\beqs 
\cT(0, \eta, \Lambda)=-2 (\Lambda-L_{\kp}),
\eeqs
where $L_{\kp}(\eta)= \p_x\big(1-3\Psi_{\kdv}\cdot\big)-\f13 \p_x^3+\eta^2\p_x^{-1}.$ 
In the following,  we use as in  \cite{Mizumachi-BL-linear}  the Lyapunov-Schmidt method to prove the existence of $(\Lambda, \psi)$ and  an  expansion in terms of $\eta$ under the form:
\beq \label{expan-modes}
\begin{aligned}
&\Lambda(\eta)=i\eta\, \Lambda_{1\ep}-\eta^2\Lambda_{2\ep}+\cO(\eta^3),\\
&\psi(\cdot,\eta)=\psi_{0\ep}(\cdot)+ i\eta \psi_{1\ep}(\cdot)+\eta^2\psi_{2\ep}(\cdot, \eta).
\end{aligned}
\eeq
Plugging the the above expansion into the equation \eqref{eigenpb-real} and sorting out in term of $\eta^j,$ we find:
\begin{align}
   & L_0({\ep,0}) \, \psi_{0\ep}=0 \,, \notag \\
   &  L_0({\ep,0}) \, \psi_{1\ep}+ \Lambda_{1\ep} 
 L_1({\ep}) \, \psi_{0\ep}=0\,, \notag \\
  & L_0({\ep,0}) \, \psi_{2\ep}-\Lambda_{1\ep}  L_1({\ep}) \, \psi_{1\ep}-\big(\Lambda_{1\ep}^2 L_2(\ep)+\Lambda_{2\ep} L_1(\ep)-L_0^1(\ep)\big)\psi_{0\ep}=0\,, \label{sec-expand}
\end{align}
where we denote $L_0^1(\ep)=\f{L_0(\ep,\eta)-L_0(\ep, 0)}{\eta^{2}}|_{\eta=0}\,.$

Since $L_0(\ep,0)$ and $L_1(\ep,0)$ are related to the linearized operator of the one dimensional water wave system linearized about  the line soliton, it is shown in Appendix D (see \eqref{id-twomodes}) that 
\beq 
\psi_{0\ep}=\ep^{-3} v_c'\bigg(\f{\cdot}{\ep}\bigg), \qquad \psi_{1\ep}=-\Lambda_{1\ep} \vp_{1c}'\bigg(\f{\cdot}{\ep}\bigg),
\eeq
where $v_c$ is defined in \eqref{defZv} and $\vp_{1c}=\p_c(c\vp_c)-cZ_c\p_c\zeta_c\,.$ 
 Denote $L_{0}(\ep,0)^{*}\in B(L_{-\hat{a}}^2(\mR))$ the adjoint operator associated to  $L_{0}(\ep,0)\in B(L_{\hat{a}}^2(\mR))$ for  the inner product $\langle \cdot, \cdot\rangle_{L_{-\hat{a}}^2(\mR)\times L_{\hat{a}}^2(\mR)}$. By using the explicit expression  %of $L_{0}(\ep,0)$
in \eqref{Ljepeta},
it is direct to check that the operator $L_{0}(\ep, 0)$ has the structure 
%{\color{red} I don't really understand that} 
\beq 
L_{0}(\ep,0)^{*}=-\p_x^{-1}L_{0}(\ep,0)\p_x\,,
\eeq
from which we obtain that $L_{0}(\ep,0)^{*}v_{\ep}=0$, with $$v_{\ep}(\cdot)=\colon\p_x^{-1}\psi_{0,\ep}(\cdot)=\ep^{-2} v_c\big({\cdot}/{\ep}\big).$$ Consequently, it holds that
\beqs 
\langle L_{0}(\ep,0)\psi_{0\ep}, v_{\ep} \rangle=0\,,
\eeqs
 where throughout this proof,  we denote $\langle \cdot,\cdot\rangle=\langle \cdot,\cdot\rangle_{(L_{-\hat{a}}^2(\mR), L_{\hat{a}}^2(\mR))}.$
Moreover, since $ d_{\ep}=1-\ep^2 v_{\ep}, \, w_{\ep}=\gamma-\ep^3 Z_{\ep}\p_x d_{\ep},\, v_{\ep}$ are all even functions, we get by integration by parts that
\beqs 
\langle L_{1}(\ep)\psi_{0\ep}, v_{\ep} \rangle=0\, .
\eeqs
Therefore, it follows from \eqref{sec-expand}  that
\beq\label{eq-labda-1}
\Lambda_{1\ep}^2 \langle-L_1(\ep) \p_x\vp_{1\ep} +L_2(\ep) \psi_{0\ep}, v_{\ep} \rangle= \langle  L_0^1(\ep)\p_xv_{\ep}
, v_{\ep} \rangle ,
\eeq
where $\vp_{1\ep}(\cdot) = \ep \vp_{1c}(\cdot/\ep).$ 
We first note that when restricting to $\ep=0,$ %and noticing that 
\beq\label{L10-L01}
L_{1}(0)=-2\,\Id,  \qquad L_0^1(0)\,\p_x=\ep^{-4}\f{\hat{G}_{\ep^2\eta}[{\ep^2\zeta_{\ep}}] - \hat{G}_{0}[{\ep^2\zeta_{\ep}}]}{\eta^2}\big|_{\eta=\ep=0}=\Id.
\eeq
Moreover, thanks to \eqref{solitarywave-def}, \eqref{solitarywave-0} and the relation $\ep^2=1-\f{gh}{c^2},$ it holds that:
\begin{align*}
&\qquad \qquad \qquad v_0(x)=\Psi_{\kdv}(x)=\sech^2\bigg(\f{\sqrt{3}}{2}x\bigg), \\
&\p_x \vp_{10}\,(x)=c\p_c \vp_c'\big(\f{x}{\ep}\big)|_{\ep=0}=c\f{\p(\ep^2)}{\p c}\p_{\ep^2}\vp_c'\big(\f{x}{\ep}\big)|_{\ep=0}=2 \big(\Id+\f{x}{2}\p_x\big)\Psi_{\kdv}(x).
\end{align*}
Consequently, since $L_2(\ep)=\cO_{B(H^1_{\hat{a}}(\mR), L_{\hat{a}}^2(\mR))}(\ep^2)\,,$
\begin{align*}
    \langle-L_1(\ep) \p_x\vp_{1\ep} +L_2(\ep) \psi_{0\ep}, v_{\ep} \rangle&=2\langle\p_x \vp_{10}, v_0 \rangle+\cO(\ep)\\
   & =3\| v_{0}\|_{L^2}^2+\cO(\ep)>0\, ,
\end{align*}
provided that $\ep$ is small enough. %Moreover,
It then follows from \eqref{eq-labda-1} that:    %We thus find: 
%\beqs \Lambda_{10}^2=\f{\| v_{0}\|_{L^2}^2}{2\langle\p_x \vp_{10}, v_0 \rangle}=\f{\| v_{0}\|_{L^2}^2}{3\|v_{0}\|_{L^2}^2}=\f13,\eeqs and thus
\beq \label{Lbda1ep}
\Lambda_{1\ep}^2=\f{\langle  L_0^1(\ep)\p_xv_{\ep}, v_{\ep} \rangle } {\langle-L_1(\ep) \p_x\vp_{1\ep} +L_2(\ep) \psi_{0\ep}, v_{\ep} \rangle}= \f{\| v_{0}\|_{L^2}^2}{3\| v_{0}\|_{L^2}^2}+\cO(\ep)=\f13+\cO(\ep)>0.
\eeq
It now remains to show the existence of  the number $\Lambda_{2\ep}<0$ and the  function $\psi_{2\ep}(\cdot, \eta)\in L_{\hat{a}}^2(\mR),$ such that \eqref{expan-modes} holds and 
$$ \overline{\Lambda(\eta)}= {\Lambda(-\eta)},\qquad \overline{\psi_{2\ep}(\cdot, \eta)}=\psi_{2\ep}(\cdot, -\eta).\, $$
To use  the Lyapunov-Schmidt method,
it is more convenient to set  $\Lambda(\eta)=i\eta\,\tilde{\Lambda}(\eta)$ and look for $\tilde{\Lambda}(\eta),\, \kappa(\eta), \tilde{\psi}_{2\ep}(\cdot, \eta)$ such that 
\beq \label{realexp}
\psi(\cdot,\eta)=v_{\ep}'-\big(i\eta \tilde{\Lambda}(\eta)-\kappa(\eta)\eta^2 \big)\vp_{1\ep}'+\eta^2\tilde{\psi}_{2\ep}(\cdot, \eta)
\eeq
and 
\beqs 
\tilde{\psi}_{2\ep}(\cdot, \eta) \perp \mathrm{Span}\bigg\{v_{\ep},\, \int_{-\infty}^x \tilde{\psi}_{1\ep}(x') \,\d x'\bigg\}=\ker_g L_{0}^{*}(\ep, 0),
\eeqs
where $\tilde{\psi}_{1\ep}\in \ker_g L_{0}(\ep, 0)$ satisfies $L_{0}(\ep, 0)\tilde{\psi}_{1\ep}=-v_{\ep}', \, \tilde{\psi}_{1\ep}|_{\ep=0}=\f12 \vp'_{1\ep}|_{\ep=0}.$
It is  worth noting that the determination of $\tilde{\psi}_{1\ep}$ can be achieved through the Lyapunov-Schmidt approach. That is, one starts from the generalized kernel of $L_0(0, 0),$ does expansions in $\ep,$ and determines the remainders of order $\ep^2$ by the Implicit Function Theorem.
Since the procedure will be quite similar as in the subsequent discussion, we omit the specific details  for brevity. 

Plugging the expansion \eqref{realexp} into the equation \eqref{eigenpb-real} and looking at the  $\cO(\eta^2)$ term, we find the equation satisfied by $\tilde{\psi}_{2\ep}:$ 
\beq \label{eqpsi2ep}
\cT_0(\ep,\eta, \tilde{\Lambda}) \,\tilde{\psi}_{2\ep}+  \cT_1(\ep, \eta,  \tilde{\Lambda}, \kappa ) + i\eta\, \cT_2(\ep, \eta,  \tilde{\Lambda}, \kappa ) =0
\eeq
where 
\begin{align*}
&\cT_0(\ep,\eta, \tilde{\Lambda}) =  L_0(\ep,\eta)+ i\eta\,\tilde{\Lambda} L_1(\ep)
-\eta^2 \tilde{\Lambda}^2  L_2(\ep)\, ,\\
&\cT_1(\ep, \eta,  \tilde{\Lambda}, \kappa ) =\tilde{\Lambda}^2 \big( L_{1}(\ep)\vp_{1\ep}'-L_2(\ep)v_{\ep}'\big)+\kappa(\eta)L_0({\ep}, \eta)\vp_{1\ep}'+ \f{L_{0}(\ep,\eta)-L_{0}(\ep, 0)}{\eta^2}v_{\ep}'\,, \\
&\cT_2(\ep, \eta,  \tilde{\Lambda}, \kappa ) =\tilde{\Lambda}\bigg(\kappa(\eta) L_1(\ep)-\f{L_{0}(\ep,\eta)-L_{0}(\ep, 0)}{\eta^2} \bigg)\vp_{1\ep}'\, .
\end{align*}
Let $\mathbb{P}_{\ep}$ be the spectral projection in $L_{\hat{\alpha}}^2(\mR)$ onto $\ker_g L_{0}(\ep,0)$ and $ \mathbb{Q}_{\ep}=\Id-\mathbb{P}_{\ep},$ then the equation \eqref{eqpsi2ep} implies that
\beqs 
 \mathbb{Q}_{\ep} \cT_0(\ep,\eta, \tilde{\Lambda})  \mathbb{Q}_{\ep}\,\tilde{\psi}_{2\ep}+  \mathbb{Q}_{\ep}  \cT_1(\ep, \eta,  \tilde{\Lambda}, \kappa ) + i\eta\,  \mathbb{Q}_{\ep}  \cT_2(\ep, \eta,  \tilde{\Lambda}, \kappa ) =0.
\eeqs
On the one hand, it holds by  the definition of $\mathbb{Q}_{\ep}$ that 
\beqs 
\| \mathbb{Q}_{\ep} \cT_0(\ep, 0, \tilde{\Lambda}) ^{-1}\mathbb{Q}_{\ep} \|_{B(L_{\hat{\alpha}}^2(\mR), H_{\hat{\alpha}}^1(\mR))}<+\infty.
\eeqs
On the other hand, in view of the expressions of 
$L_0(\ep,\eta), L_1(\ep), L_2(\ep)$ in \eqref{Ljepeta}, one gets that
\begin{align*}
    \|\cT_0(\ep, \eta, \tilde{\Lambda})-\cT_0(\ep, 0, \tilde{\Lambda})\|_{B(L_a^2)}&\lesssim a^{-1}(\ep\eta)^2(1+|\tilde{\Lambda}(\eta)|^2)+\eta |\tilde{\Lambda}(\eta)|\\
    & \lesssim \eta_0^2(1+M^2%(\eta_0)
    )+\eta_0 M,
\end{align*}
where  $M=\sup_{\eta\in[-\eta_0, \eta_0]}| \tilde{\Lambda}(\eta)|.$ We thus conclude that, upon choosing $\eta_0$ small enough,
\beqs 
\| \mathbb{Q}_{\ep} \cT_0(\ep, \eta, \tilde{\Lambda}) ^{-1}\mathbb{Q}_{\ep} \|_{B(L_{\hat{\alpha}}^2(\mR), H_{\hat{\alpha}}^1(\mR))}<+\infty.
\eeqs
 The existence of $\tilde{\psi}_{2\ep}$ is thus guaranteed upon establishing the existence of the  smooth curves $\tilde{\Lambda}(\cdot), \kappa(\cdot): [-\eta_0,\, \eta_0]\rightarrow \mathbb{C},$ which is the task of the following.

Since $\tilde{\psi}_{2\ep} \perp \ker_g L_{0}^{*}(\ep, 0),$ we get from \eqref{eqpsi2ep} that 
%\beqs L_{0}^{*}(\ep, 0) v_{\ep}=0, \quad L_{0}^{*}(\ep, 0) \bigg( -\int_x^{+\infty} \tilde{\psi}_{1\ep}(x') \,\d x'\bigg)=v_{\ep}, \quad \eeqs
\beq \label{Implicitthm}
\begin{aligned}
   & \cH_1 (\ep, \eta,  \tilde{\Lambda}, \kappa  )=\colon
 \big\langle 
     \cT_1(\ep, \eta,  \tilde{\Lambda}, \kappa ) + i\eta\, \cT_2(\ep, \eta,  \tilde{\Lambda}, \kappa ) +\big(\cT_0(\ep, \eta)-\cT_0(\ep, 0)\big) \tilde{\psi}_{2\ep}, \, v_{\ep}\big\rangle=0\,,\\
    & \cH_2 (\ep, \eta,  \tilde{\Lambda}, \kappa  )=\colon
\big \langle 
     \cT_1(\ep, \eta,  \tilde{\Lambda}, \kappa ) + i\eta\, \cT_2(\ep, \eta,  \tilde{\Lambda}, \kappa ) \\
     & \qquad \qquad\qquad\qquad+\big(\cT_0(\ep, \eta)-\cT_0(\ep, 0)\big) \tilde{\psi}_{2\ep}, \, \int_{-\infty}^x \tilde{\psi}_{1\ep}(x')  \,\d x' \big\rangle=0\,.
\end{aligned}
\eeq
Looking at these two identities at $\eta=0$, we obtain the following relations 
 \begin{align*}
&0=\cH_1|_{\eta=0}=  \big\langle  \cT_1|_{\eta=0}, \, v_{\ep}\big\rangle=  \big\langle \tilde{\Lambda}^2(0) \big( L_{1}(\ep)\vp_{1\ep}'-L_2(\ep)v_{\ep}'\big)+ L_0^1(\ep) v'_{\ep},  \, v_{\ep}\big\rangle, \\
& 0= \cH_2|_{\eta=0}=  \big\langle  \cT_1|_{\eta=0}, \, \int_{-\infty}^x \tilde{\psi}_{1\ep}(x') \,\d x'\big\rangle\\
& \qquad\qquad =  \big\langle \tilde{\Lambda}^2(0) \big( L_{1}(\ep)\vp_{1\ep}'-L_2(\ep)v_{\ep}'\big)+ L_0^1(\ep) v'_{\ep}+\kappa(0)L_0(\ep, 0)\vp_{1\ep}',  \,\int_{-\infty}^x \tilde{\psi}_{1\ep}(x') \,\d x'\big\rangle,
 \end{align*} 
The first relation yields $$\tilde{\Lambda}^2(0) = \f{\langle  L_0^1(\ep)v'_{\ep}, v_{\ep} \rangle } {\langle-L_1(\ep) \vp_{1\ep}' +L_2(\ep) \psi_{0\ep}, v_{\ep} \rangle}=\f13+\cO(\ep)>0, $$ which is consistent with \eqref{Lbda1ep}. Next, since 
$$\vp_{10}'=2 v_0+ x\p_x v_0,\, \quad \tilde{\psi}_{10}=\f12 \vp_{10},\qquad L_0^*(\ep, 0)\int_{-\infty}^x \tilde{\psi}_{1\ep}(x') \,\d x'=v_{\ep},$$ it holds that
\beqs 
\big\langle L_0(\ep, 0)\vp_{1\ep}'\,,  \,\int_{-\infty}^x \tilde{\psi}_{1\ep}(x') \,\d x'\big\rangle
=\big\langle \vp_{1\ep}',  \, v_{\ep} \big\rangle=\big\langle \vp_{10}',  \, v_{0} \big\rangle+\cO(\ep)=\f32\,\|v_0\|_{L^2(\mR)}^2+\cO(\ep)\neq 0.
\eeqs
Moreover, in view of \eqref{L10-L01}, we have that
\begin{align*}
&\big\langle \tilde{\Lambda}^2(0) \big( L_{1}(\ep)\vp_{1\ep}'-L_2(\ep)v_{\ep}'\big)+ L_0^1(\ep) v'_{\ep},  \,\int_{-\infty}^x \tilde{\psi}_{1\ep}(x') \,\d x'\big\rangle\\
&=\big\langle v_0-\f23\,\vp_{10}',  \,\int_{-\infty}^x \tilde{\psi}_{1\ep}(x') \,\d x'\big\rangle+\cO(\ep)=\f{1}{12}\, \|v_0\|_{L^1(\mR)}^2+\cO(\ep).
\end{align*} 
It thus follows from the second relation that
\beq\label{kappa0}
\kappa(0)=-\f{1}{18} \f{\|v_0\|_{L^1(\mR)}^2}{\|v_0\|_{L^2(\mR)}^2}+\cO(\ep)\neq 0.
\eeq 
To determine the curves $\tilde{\Lambda}(\eta),\, \kappa(\eta)$ around $\eta=0,$ we compute  %by using the implicit theorem. Since 
\begin{align}
& \p_{\kappa} \cH_1|_{\eta=0}=\big\langle L_0(\ep, 0)\vp_{1\ep}', v_{\ep} \big\rangle=0, \notag \\
 &    \p_{\kappa} \cH_2|_{\eta=0}=\big\langle L_0(\ep, 0)\vp_{1\ep}',\, -\int_x^{+\infty} \tilde{\psi}_{1\ep}(x') \,\d x'  \big\rangle=\f32\,\|v_0\|_{L^2}^2+\cO(\ep)\neq 0, \notag\\
  &    \p_{\tilde{\Lambda}} \cH_2|_{\eta=0}= 2 \tilde{\Lambda}(0)\big\langle L_{1}(\ep)\vp_{1\ep}'-L_2(\ep)v_{\ep}',\, v_{\ep}  \big\rangle=-6 \tilde{\Lambda}(0) \|v_0\|_{L^2}^2+\cO(\ep)\neq 0, \label{ptLbda}
\end{align}
which lead to
\beqs
\det\big(D_{\kappa,  \tilde{\Lambda} }(\cH_1, \cH_2)\big)|_{\eta=0} \neq 0. 
\eeqs
Consequently, by the Implicit Function Theorem, there exists $\eta_0$ and two $C^1$ curves $\tilde{\Lambda}(\eta),\, \kappa(\eta): [-\eta_0,\, \eta_0]\rightarrow \mathbb{C},$ such that \eqref{Implicitthm} holds true.
%and $$
Moreover, it holds that
\beqs
%\mathrm{i}
i\,\Lambda_{2\ep}=\colon \p_{\eta} \tilde{\Lambda}\big|_{\eta=0}=-\f{\p_{\eta}\cH_1}{\p_{\tilde{\Lambda}} \cH_1}\big|_{\eta=0}\,.
\eeqs
In light of the expression of $\cH_1$ in \eqref{Implicitthm}, we compute
\begin{align*}
%\mathrm
{i}\,\p_{\eta}\cH_1\big|_{\eta=0}&=-\tilde{\Lambda}(0)\big\langle 
\big(\kappa(0) L_1(\ep)-%\f{L_{0}(\ep,\eta)-L_{0}(\ep, 0)}{\eta^2} 
L_0^1(\ep)\big) \,\vp_{1\ep}'+
L_1(\ep)\,\tilde{\psi}_{2\ep},\, v_{\ep}
\big\rangle  \\
&=\tilde{\Lambda}(0)\big\langle (2\kappa(0)+\p_x^{-1})\,\vp_{10}',\, v_{0}
\big\rangle+\cO(\ep) \\
&=\tilde{\Lambda}(0) \bigg( 3\kappa(0)\|v_0\|_{L^2(\mR)}^2-\f12\|v_0\|_{L^1(\mR)}^2 \bigg)+\cO(\ep).
\end{align*} 
Therefore, it follows from \eqref{kappa0} and \eqref{ptLbda} that for $\ep$ small enough, %As a result,  %by applying \eqref{ptLbda}
\beqs 
\Lambda_{2\ep}=\f{1}{9} \f{\|v_0\|_{L^1(\mR)}^2}{\|v_0\|_{L^2(\mR)}^2}+\cO(\ep)>0\, .
\eeqs
Finally, by the fact $\overline{L_0(\ep, \eta )}= {L_0(\ep, -\eta )},\,$ it holds that 
\beqs 
\overline{\cH_k(\ep,\eta, \kappa, \tilde{\Lambda})}=\cH_k\big(\ep, -\eta, \overline{\kappa}, \overline{\tilde{\Lambda}}\big), \qquad k=1,2. 
\eeqs
We thus get from the uniqueness of the curves that:
\beqs
\overline{\tilde{\Lambda}(\eta)}=\tilde{\Lambda}(-\eta), \qquad \overline{\kappa(\eta)}=\kappa(-\eta),
\eeqs
which further implies that
\beqs 
\overline{{\Lambda}(\eta)}={\Lambda}(-\eta), \qquad \overline{\tilde{\psi}_{2\ep}(\cdot, \eta)}=\tilde{\psi}_{2\ep}(\cdot, -\eta).
\eeqs

Let us summarize what we have obtained and their consequences, which lead to Theorem \ref{thm-resolmodes}.
\begin{lem}
There exists $\ep_0, \hat{\eta}_0>0,$  such that for any $\ep\leq \ep_0,$ 

(1) For any $\hat{\eta}\leq \hat{\eta}_0, $ one can find  
$\Lambda(\hat{\eta}), %\in C^{\infty}([-\hat{\eta_0}, {\eta_0}]), $ a function $
\, {\psi}(\cdot, \hat{\eta}) \in H_{\hat{a},*, \hat{\eta}}^{1/2} $  which  solve \eqref{eigenpb-real} and enjoy the expansion: 
\beq \label{expan-modes-1}
\begin{aligned}
&\Lambda(\hat{\eta})=i\hat{\eta}\, \Lambda_{1\ep}-\hat{\eta}^2\Lambda_{2\ep}+\cO(\hat{\eta}^3),  \\
&\psi(\cdot,\hat{\eta})=\p_x \bigg(v_{\ep}- \big(i\,\Lambda_{1\ep} \hat{\eta} - \tilde{\Lambda}_{2\ep}\hat{\eta}^2\big)\,\vp_{1\ep}\bigg)(\cdot)+\hat{\eta}^2 \tilde{\psi}_{2\ep}(\cdot, \hat{\eta}) %\hat{\eta}^2\psi_{2\ep}(\cdot, \hat{\eta}).
\end{aligned}
\eeq
where 
\begin{align*}
& \Lambda_{1\ep}=\f{1}{\sqrt{3}}+\cO(\ep)\,, \qquad
\Lambda_{2\ep}=\f{1}{9} \f{\|v_0\|_{L^1(\mR)}^2}{\|v_0\|_{L^2(\mR)}^2}+\cO(\ep), \,    \quad \tilde{\Lambda}_{2\ep}=\f{1}{18} \f{\|v_0\|_{L^1(\mR)}^2}{\|v_0\|_{L^2(\mR)}^2}+\cO(\ep),\\
& v_{\ep}(\cdot)=\ep^{-2}v_c (\ep^{-1} \,\cdot)\, , \, 
\qquad \vp_{1\ep}(\cdot)=\ep\, 
\big(\p_c(c\vp_c)-cZ_c\p_c\zeta_c\big)(\ep^{-1} \,\cdot)\,, \qquad \tilde{\psi}_{2\ep}(\cdot, \hat{\eta}) \perp v_{\ep}\,.
\end{align*}

(2) For any $\eta\in [-\ep^2\hat{\eta}_0, \ep^2 \hat{\eta}_0],$ the operator $L(\eta)$ has  two pairs of eigenmodes $\big(\lambda(\eta), U(\cdot, \eta)\big) $ and $\big(\lambda(-\eta), U(\cdot, -\eta)\big)$ in the space $Y_{a, \eta}$
where 
\begin{align}\label{defU-base}
   & \lambda(\eta)=\ep^3 \Lambda(\ep^{-2}\eta)=i \ep\Lambda_{1\ep}\eta-\ep^{-1}\Lambda_{2\ep}\eta^2+\cO(\eta^3)\, , \notag\\ 
     U(\cdot, \eta)&= \left( \begin{array}{c}
-{w}^{-1}_c\big(\lambda(\eta)+\ep \,d_c\,\p_{\hat{x}}\big)(\p_{\hat{x}}^{-1} \psi)    \\[3pt]
        \p_{\hat{x}}^{-1} \psi 
    \end{array} 
    \right)(\ep\, \cdot, \, \ep^{-2}\eta )=\colon \left( \begin{array}{c}
U_1 (\cdot, \, \eta )    \\[3pt]
U_2 (\cdot, \, \eta )
    \end{array} 
    \right)  \, .
\end{align}

(3) The eigenmode $(\lambda^{*}(\pm \eta), U^{*}(\cdot, \pm\eta))$ of $L^{*}_a(\eta)$ in the space $Y_{a, \eta}^{*}$ are characterized by 
\begin{align}\label{dualmodes}
    \lambda^{*}(\eta)=\lambda(-\eta), \qquad  U^{*}(\cdot, \eta)=(U_2, U_1)^t(-\,\cdot, -\eta)\,.
\end{align}
Moreover, it holds that 
\beqs
\overline{\lambda(\eta)}=\lambda(-\eta), \quad \overline{U(\cdot, \eta)}=U(\cdot, -\eta), \quad \big\langle U(\cdot, \eta),\, U^{*}(\cdot, -\eta)  \big\rangle_{_{Y_{a,\eta}\times Y_{a,\eta}^{*}}}=0 ,
\eeqs
where the inner product is defined as follows:
 for $f=(f_1,f_2)^t\in Y_{a,\eta}, \, f^{*}=(f_1^*, f_2^*)^t\in  Y_{a,\eta}^{*},$
\beq \label{def-bracketYa}
\big\langle f, f^{*}\rangle_{_{Y_{a,\eta}\times Y_{a,\eta}^{*}}} =: \langle f_1, f_1^{*} \rangle_{L_a^2(\mR)\times L_{-a}^2(\mR)}+\left  \langle  \f{(\p_x,\eta)}{\langle (D_x, \eta) \rangle^{1/2}} f_2, \f{(\p_x,\eta)}{\langle (D_x, \eta) \rangle^{1/2}} f^{*}_2\right \rangle_{L_a^2(\mR)\times L_{-a}^2(\mR)}.
%\left \langle  \f{(\p_x,\eta)}{\langle (D_x, \eta) \rangle^{1/2}} e^{ax} f, \f{(\p_x,\eta)}{\langle (D_x, \eta) \rangle^{1/2}} e^{-ax} f^{*}\right \rangle_{L^2(\mR)\times L^2(\mR)}
\eeq
\end{lem}

\begin{proof}
The first point and the existence of the  eigenmode $\big(\lambda(\eta), U(\cdot, \eta)\big) $ for $L(\eta)$ in the second point have been proven. The existence of the other eigenmode follows from the evenness of $L(\eta)$ in $\eta.$ %$\big(\lambda(\eta), U(\cdot, \eta)\big) $
%By the definition, $G_{\eta}[\zeta_c]$ is even in $\eta,$ so that $L(\eta)=L(-\eta).$ 
Now only the last point requires some clarification. In fact, as an operator taken on the space $Y_{a}^{*}(\eta),$ the dual operator $L^{\star}(\eta)$ takes the form
\beqs 
L^{*}(\eta)=\left(\begin{array}{cc}
  -d_c \p_x   & -w_c  \\[3pt]
  G_{\eta}[\zeta_c]   & -\p_x(d_c \cdot) 
\end{array}
\right).
\eeqs
Note that we have used the fact that $G^{*}_{a,\eta}[\zeta_c]=G_{-a,\eta}[\zeta_c].$  
The identities in \eqref{dualmodes} then follow from the evenness of $\zeta_c, \, d_c,\, w_c$ in $x$ and $G_{\eta}[\zeta_c]$ in $\eta.$ 
\end{proof}

We are now in position to define precisely  the basis of $L(\eta).$ %Given the evenness of $L(\eta)$ is $\eta,$ we see that it admits two eigenvalues $\lambda(\eta)$ and $\lambda(-\eta)$ and the associated eigenfunctions are
We already know that 
 $U(\cdot, \eta)$ and $\overline{U(\cdot, \eta)}$ are two eigenfunctions. 
 Nevertheless, their imaginary parts  vanishes when $\eta=0$ and thus they do not provide a basis of the two-dimensional algebraic kernel
  of $L(0)$. To resolve this degeneracy, inspired by
\cite{Mizumachi-BL-linear}, 
we first introduce a pair of functions with  purely real inner product:
\begin{align*}
     g(\cdot,\eta)=(\alpha(\eta)-{i}) \, U(\cdot, \eta), \quad g^{*}(\cdot, \eta)=-\f{\sqrt{3}}{4}\, U^{*} (\cdot, \eta)
\end{align*}
where $\alpha(\eta)=\f{\Re \langle U(\cdot, \eta), U^{*} (\cdot, \eta)\rangle}{\Im \langle U(\cdot, \eta), U^{*} (\cdot, \eta)\rangle}.$
%$c_0\in\mR$ is a constant to be chosen. 
Note that 
\beq \label{prop-geta}
\begin{aligned}
&\quad \overline{ g(\cdot,\eta)}=  -g(\cdot, -\eta), \qquad \overline{ g^{*}(\cdot,\eta)}=  g^{*}(\cdot, -\eta), \\
 & \langle g(\cdot, \eta), g^{*}(\cdot, -\eta) \rangle =0, \qquad \Im \langle g(\cdot, \eta), g^{*}(\cdot, \eta) \rangle =0.
\end{aligned}
\eeq
We then define the basis and dual basis:
\beq \label{defbasis-dual}
\begin{aligned}
   & g_1(\cdot, \eta)= -\f{1}{2i
}\big( g(\cdot, \eta)+ g(\cdot, -\eta)\big) =\big(\Re U-\alpha(\eta)\,\Im U\big)(\cdot, \eta), \\  & g_2(\cdot,\eta)=\f{1}{2
\kappa(\eta)}\big( g(\cdot, \eta)-g(\cdot, -\eta)\big)=\f{1}{\kappa(\eta)}\big(\Im U+\alpha(\eta)\Re U\big)(\cdot, \eta), \\
  &  g^{*}_1(\cdot, \eta)=-\f{1}{2i\kappa(\eta)} (g^{*}(\cdot, \eta)-g^{*}(\cdot, -\eta))=
    \f{\sqrt{3}}{4\,\kappa(\eta)}\Im U^{*} (\cdot, \eta),\\  & g_2^{*}(\cdot, \eta)=  \f{1}{2} (g^{*}(\cdot, \eta)+g^{*}(\cdot, -\eta))=\f{\sqrt{3}}{4}\, \Re  U^{*} (\cdot, \eta),
\end{aligned}
\eeq
where 
\begin{align*}
    \kappa(\eta)=\f12 \,\Re \langle g(\cdot, \eta), g^{*}(\cdot,\eta)\rangle_{_{Y_{a,\eta}\times Y_{a,\eta}^{*}}} =-\f{\sqrt{3}}{8} \bigg( \Im \langle U, U^{*} \rangle+\f{(\Re \langle U, U^{*} \rangle)^2}{\Im \langle U, U^{*} \rangle } \bigg)(\eta).
\end{align*}
Note that by \eqref{prop-geta}, it holds that 
\beqs 
\langle g_{j}(\cdot, \eta), g_k^{*}(\cdot, \eta)\rangle_{_{Y_{a,\eta}\times Y_{a,\eta}^{*}}} =\delta_{j k}, \quad \forall\, 1\leq j, k \leq 2; \, \eta\in [-\ep^2 \hat{\eta}_0, \ep^2 \hat{\eta}_0 ] \, .
\eeqs

%Once the basis are defined, we could define the projection of $L$ in the weighted space $X_a$ onto the continuous eigenmodes $(e^{iy\eta} g_k(\cdot, \eta))_{-\eta_0\leq \eta\leq \eta_0}:$
%\beqs \mathbb{P}(\eta_0) f=\sum_{k=1}^2\int_{-\eta_0}^{\eta_0} \inp{\mathcal{F}_yf(\cdot,\eta)}{g_k^{*}(\cdot,\eta)}\,g_k(\cdot,\eta)\,e^{iy\eta}\,\d \eta\eeqs

\section{ Spectral stability in the uniform high  and intermediate transverse frequency regions.}

\label{paragraph3}

In this section, we focus on the uniform  high  and intermediate  transverse frequency estimates
that is the regions $R^{UH}$ and $R^I$ defined in \ref{defregions}.
 Let $\mathbb{I}_K(\cdot):\mR\rightarrow \mR$ be the characteristic function of $[-K, K]$ and $\tilde{\mathbb{I}}_K=1-\mathbb{I}_K.$ 
For a large number $A$ such that $A\ep^2<1,$ we define the spaces $X^{UH}$ and $X^I_A$ as:
$$X^{UH}=\tilde{\mathbb{I}}_{2}(D_y) X ,
\quad 
X^I_{A}=\mathbb{I}_2\tilde{\mathbb{I}}
%\mathbb{I}
_{A\ep^2}(D_y)X,$$
    which are subspaces of $X$  that as introduced above, we refer to as the  uniform  high transverse frequency  region  (when  $\{|\eta|\geq 2\}$) and  the intermediate  transverse frequency region  (ie. $\{ A\ep^2\leq |\eta|\leq 2\}$). Denote $X_A^H= X^{UH}\cup X_A^I,$
 %\tilde{\mathbb{I}}_{{A\ep^2}}(D_y)X,$
%and high horizontal frequency (ie. $\{|\xi|\geq K\ep\}$) respectively.
our goal  in this section is to show the following proposition: 

\begin{prop}\label{prop-TH}
Assume that $0<\beta\leq \f12$ and $A$ sufficiently large. 
There exists $\ep_0$ small enough, such that for any $\ep\in(0,\ep_0],$ it holds that
$$\Omega_{\beta, \ep}=\colon\{\lambda\in \mathbb{C} \,|\, \Re \lambda\geq -\beta\ep^3\}\subset \rho(L_a; X_A%_{{A\ep^2}}
^H).$$ 
Moreover, for any $\lambda\in \Omega_{\beta,\ep}, $ 
\beqs 
 \|\big(\lambda- %\tchi_{_{A\ep^2}}(|D_y|)
 L_a\big)^{-1}\|_{B(X_A
 ^H)}  \lesssim A^{-1}%\f{1}{{A}
 \ep^{-3}, \label{HTF-0}
\eeqs
where $\lesssim \cdot$ denotes $\leq C\cdot $ for some constant $C$ that is independent of $A, \ep\,.$
\end{prop}

%$$X_{a,A}^H=\tilde{\mathbb{I}}%\tilde{\chi}_{{A\ep^2}}(|D_y|)X_a, \quad  X_{a,K}^H =\tilde{\chi}_{_{K\ep}}(|D_x|){\mathbb{I}} %\tilde{\chi}_{{A\ep^2}}(|D_y|)X_a$$ 
This proposition is the direct consequence of the following two propositions concerning the resolvent estimates on the space $X^{UH}$ and $X_A^I$ respectively.

\begin{prop}\label{prop-UTH}
Suppose that $0<\beta\leq \f12$
and $\ep$ is sufficiently small. 
 For any $\lambda\in \Omega_{\beta, \ep},$ the operator
    $\big(\lambda- 
 L_a\big)$ is invertible on the space $X^{UH}$ and its inverse has the bound: 
 \beq 
 \|\big(\lambda-
 L_a\big)^{-1}\|_{B(X^{UH})} \leq  C\ep^{-1}\,,
 \eeq
 where the constant $C>0$ is independent of $\ep.$
\end{prop}

\begin{prop}\label{prop-ITF}
Suppose that $0<\beta\leq \f12.$
There exists a large number $A>0$ 
such that for any $\ep$ sufficiently small so that 
$A\ep^2<1,$
 the following hold true: for any $\lambda\in \Omega_{\beta,\ep},$ the operator
$\big(\lambda- 
 L_a\big)$ is invertible on the space $X_{A}^I$ and its inverse has the bound: 
\begin{align}
 \|\big(\lambda- %\tchi_{_{A\ep^2}}(|D_y|)
 L_a\big)^{-1}\|_{B(X_{A}^I)}  \leq C A^{-1}%\f{1}{{A}
 \ep^{-3}, \label{ITF} % \\
 % \|\big(\lambda-%\tchi_{_{1, K\ep}}(|D_x|)  L_a\big)^{-1}\|_{B(X_{K}^H)}  \lesssim \f{1} {K \ep^{3}}, \label{HHF-0}
\end{align}
where $C$ is a constant independent of $A, \ep.$
\end{prop}

We will prove Propositions \ref{prop-UTH} and \ref{prop-ITF} in the following two subsections.
\subsection{Proof of Proposition \ref{prop-UTH}}
\begin{proof} [Proof of Proposition \ref{prop-UTH}]
  For any $f\in H^{1/2}_{*}(\mR^2)$  with transverse frequency $|\eta|\geq 2,$ it holds that 
$\|f\|_{H_{%a,
\star}^{{1}/{2}}(\mR^2)}\approx \|\sqrt{G_a[0]} f\|_{L^2(\mR^2)}$ since $a=\hat{a}\ep\leq 1/2.$ Therefore, to prove Proposition \ref{prop-UTH} , it suffices to show that  
$\Omega_{\beta, \ep}\subset \rho(\cL_a; \cX^{UH}),$
with the bound
\beq 
 \|\big(\lambda-
 \cL_a\big)^{-1}\|_{B(\cX^{UH})} \leq  C\ep^{-1}, \quad \forall \, \lambda \in 
 \Omega_{\beta, \ep},
 \eeq
where $\cX^{UH}=\tilde{\mathbb{I}}_2(D_y) (L^2(\mR^2)\times L^2(\mR^2))$ and 
\beq\label{defcL} 
\begin{aligned}
\cL_a&=\colon \diag(1, \sqrt{G_a[0]} )\, L_a \,\diag\big(1, \sqrt{G_a[0]}^{-1} \big)\\
&=\left( \begin{array}{cc}
   (\p_x-a)(d_c\cdot)  &  G_a[\zeta_{c}] \sqrt{G_a[0]}^{-1}  \\[5pt]
 -\sqrt{G_a[0]}  w_c&\qquad d_c(\p_x-a)+\big[\sqrt{G_a[0]}, d_c\big](\p_x-a) \sqrt{G_a[0]}^{-1}
\end{array}\right)
\end{aligned}
\eeq 
where $\left[ \cdot, \cdot \right]$ stands for the commutator of two operators.

For any $F\in \cX^{UH},$ and $\lambda\in\Omega_{\beta, \ep},$ consider the resolvent problem: 
%Let  $H_{%a,\star}^{{1}/{2}}(\mR^2)$ is equivalent to $\|\sqrt{G_a[0]} \cdot \|_{L^2(\mR^2)}.$
% For any $f\in L_a^2,$ consider the resolvent problem: 
    \beqs 
\big( \lambda-\cL_a\big) U=F,
    \eeqs
we shall prove the existence of $U\in \cX^{UH}$ and the resolvent estimate simultaneously. %As the background waves does not depend on the transverse variable, 
Let us set $\cL_a({\eta})=e^{-iy \eta}\cL_a e^{iy \eta}.$ By taking the Fourier transform in the transverse variable, 
we reduce the matter to solving the resolvent problem 
\beq \label{resol-HTF}
\big( \lambda-\cL_a(\eta)\big) u=f, \quad \, \forall \,|\eta|\geq 2
\eeq
in the space $\cY= L^2(\mR)\times L^2(\mR),$ and proving the resolvent bounds  $\|u\|_{\cY}\lesssim \ep^{-1}\|f\|_{\cY}$ uniformly for $\eta$ such that $|\eta|\geq 2.$

From its  definition, $\cL_a(\eta)$ takes the same  form as in  \eqref{defcL} with $G_a[\cdot]$ replaced by 
%\beq\label{defcLeta} \cL_a(\eta)= \left( \begin{array}{cc}  (\p_x-a)(d_c\cdot)  &  G_{a,\eta}[\zeta_{c}] \sqrt{G_{a,\eta}[0]}^{-1}  \\[5pt] -\sqrt{G_{a,\eta}[0]}  w_c& d_c(\p_x-a)+\big[\sqrt{G_{a,\eta}[0]}, d_c\big](\p_x-a) \sqrt{G_{a,\eta}[0]}^{-1}\end{array}\right)\eeq
$$G_{a,\eta}[\cdot]=\colon e^{-iy \eta}\,G_a[\cdot]\, e^{iy \eta}.$$ 
As $\zeta_c\in C^{\infty}(\mR)$ is  sufficiently smooth, we could 
translated in our framework the results of  \cite{Alazard-Metivier,Sylvester} for example,
to  get an approximation  of  the operator  $G_{a,\eta}[\zeta_c]$ by isolating the principal symbol
\beq\label{def-prsy}
\lambda
^1_{\zeta_c}(x, \xi, \eta)=\sqrt{(\xi+ia)^2+\eta^2 (1+(\p_x\zeta_c)^2)}\,.
\eeq
Note that for any $|\eta|\geq 2$, we have that $\lambda_{\zeta_c}^1$ does not vanish for any $(a, x,\xi)\in [0,\f12]\times  \mR^2,$ upon choosing $\ep$ small enough. Nevertheless, as mentioned in the introduction, although $\op(\lambda_{\zeta_c}^1)$ approximates well the operator $G_{a,\eta}[\zeta_c]$ in terms of the regularity, it is far from being a good approximation in terms of the size of the  operator norm within $B(H^{1/2}, L^2),$ since one only has:
\beqs 
\|G_{a,\eta}[\zeta_c]-\op(\lambda
^1_{\zeta_c})\|_{B(L^2(\mR))}\lesssim 1 \, \,(\text{rather than } \ep^2 )\,,
\eeqs
which prevents us to prove the resolvent bounds by  perturbation arguments. Nevertheless, as proved in 
Lemma \ref{lem-error}, the desired estimate holds true by taking a modification of the principal symbol which is more accurate
for low frequencies: 
\beqs
\|G_{a,\eta}[\zeta_c]-\op(\lambda
^1_{\zeta_c}\tanh \lambda
^1_{\zeta_c})\|_{B(L^2(\mR))}\lesssim \ep^2.
\eeqs

We thus define an approximate operator $\tilde{\cL}_{a}(\eta)$ by  replacing $G_{a,\eta}[\zeta_c]$ by $\op(\lambda_{\zeta_c}^1
\tanh \lambda
^1_{\zeta_c})$ in the definition of ${\cL}_{a}(\eta),$  that is: 
\beq\label{deftcLeta} 
\tilde{\cL}_a(\eta)=\colon \left( \begin{array}{cc}
   (\p_x-a)(d_c\cdot)  &  \op( \lambda^1_{\zeta_c}\tanh \lambda^1_{\zeta_c})\sqrt{G_{a, \eta}[0]}^{-1}  \\[5pt]
 -\sqrt{G_{a,\eta}[0]}  w_c& \qquad d_c(\p_x-a)+\big[\sqrt{G_{a,\eta}[0]}, d_c\big](\p_x-a) \sqrt{G_{a,\eta}[0]}^{-1}
\end{array}\right).
\eeq
The strategy is thus to  first prove  the invertibility  and  the resolvent bounds for  $\lambda-\tilde{\cL}_a(\eta)$ and  then, 
in a second step to  extend these results to $\lambda-{\cL}_a(\eta)$ through perturbation arguments.

Define the function 
\beq \label{def-g}
g_{\eta}(x,\xi)=\sqrt{\f{w_c(x) (\mu_a\tanh \mu_a)(\xi,\eta)}{\lambda^1_{\zeta_c}\tanh \lambda^1_{\zeta_c}(x,\xi,\eta)}},
\eeq 
where $(\mu_a\tanh \mu_a)(\xi,\eta)$ is the symbol of the operator $G_{a,\eta}[0]$  and $\mu_a(\xi,\eta)=\sqrt{(\xi+ia)^2+\eta^2}.$ 
It is shown in Lemma \ref{lem-g} that for $\ep$ small enough, both $g_{\eta}$ and ${1}/{g_{\eta}}$ belong to the symbol class with zero order ($S^0$ as defined in Appendix A).
Denote $\op(f)$ the operator  
associated to the symbol $f(x,\xi)$ by the standard quantization (see Appendix A).  Then by using a standard result for the $L^2$ continuity of
pseudodifferential operators with 
$S^0$ symbols %(in $S^0$ as defined in Appendix A), 
we get that 
 for any $|\eta|\geq 2,$  $\op(g_{\eta}), \op(1/g_{\eta})$ are  bounded operators on $L^2(\mR)$ with norm  uniformly bounded in  $\eta$ and $\ep.$ Moreover, by using  standard pseudo-differential calculus and the fact that $\|\p_x(\zeta_c, \p_x \varphi_c)\|_{W_x^{k,\infty}(\mR)}\lesssim \ep^3$ for $k$  as large as we need
   to approximate the 
composition of pseudo-differential operators \cite{Zworski-book}, it holds that: 
\beqs 
\op(g_{\eta})\circ  \op(1/g_{\eta})=\Id_{L^2(\mR)} + R_{1\eta}\, , \quad  \op(1/g_{\eta})\circ \op(g_{\eta}) =\Id_{L^2(\mR)}+ R_{2\eta} \,,
\eeqs
where  $R_{1\eta}, R_{2\eta}$ are two bounded operators in $L^2(\mR)$ with  operator norms not larger than $C\ep^3$ (C being independent of $\eta$).
%$\|R_{1\eta}, R_{2,\eta}\|_{B(L^2(\mR))}\lesssim \ep^3$
%(we denote its symbol as $\lambda_G (x, \xi, \eta)).$ 
%By the Fourier transform in the transverse variable, 
%it suffices to show that for any $f\in \cY= L^2\times L^2,$ any $\lambda\in\Omega_{\ep},$ there exists a unique solution $u_{\eta}$ such that $\big( \lambda-\cL_a(\eta)\big) u=f,$ and $\|u_{\eta}\|_{\cY}\lesssim \ep^{-1}\|f\|_{\cY}.$ 

 Let the operator valued matrices $P_1(\eta)$ and $P_2(\eta)$ be defined by:
  \begin{align*}
  P_1(\eta)=   \left(\begin{array}{cc}
   1  & 1 \\
   \op(g_{\eta})
   & -  \op(g_{\eta})
\end{array}  \right), \,\, P_2= \f12 \left(\begin{array}{cc} 1  &  \op({1}/{g_{\eta}})  \\ 1  & -\op(1/g_{\eta}) \end{array}  \right),
\end{align*}
%which satisfies $P_1(\eta)\circ P_2(\eta)=\Id +\cR_{\eta}$ 
 %  \begin{align*}
%  P(\eta)=   \left(\begin{array}{cc}   1  & 1 \\
  % \sqrt{w_c} \sqrt{-G_{a,\eta}[0]} \circ \op (\f{1}{\sqrt{\lambda_G}})%\sqrt{-G[\zeta_c]}^{-1} ,  & - \sqrt{w_c}\sqrt{-G_{a,\eta}[0]} \circ \op (\f{1}{\sqrt{\lambda_G}})\end{array}  \right), \,\, P^{-1}= \f12 \left(\begin{array}{cc} 1  &  \sqrt{w_c}^{-1}\op(\sqrt{\lambda_G})\circ\sqrt{-G_{a,\eta}[0]}^{-1} \\ 1  & - \sqrt{w_c}^{-1}\op(\sqrt{\lambda_G})\circ\sqrt{-G_{a,\eta}[0]}^{-1}\end{array}  \right)\end{align*}
%By using the fact $|\p_x(\zeta_c, \p_x \varphi_c)|\lesssim \ep^3$ and Lemma [?], 
we can  then diagonalize $\tilde{\cL}_a(\eta)$ as:
\beq\label{diagnolzation}
\tilde{\cL}_a(\eta)= P_1(\eta) \,\,\op \left(\begin{array}{cc}
   \lambda_{-}^1(x,\xi,\eta)  &  \\
     &   \lambda_{+}^1(x,\xi,\eta)
\end{array}  \right) P_2(\eta)+\cR_{\eta}\, ,
\eeq
where  \beq \label{deflpm}
\lambda_{\pm}^1=i \bigg(d_c(x)(\xi+ia)\pm \sqrt{w_c}(x)\sqrt{\lambda^1_{\zeta_c}\tanh \lambda^1_{\zeta_c}}(x,\xi,\eta)\bigg).
\eeq 
Hereafter, 
$\cR_{\eta}$  will always stand for a generic bounded linear operator in
$\mathcal{Y}=L^2(\mR)\times L^2(\mR)$ with  norm proportional to $\ep^3$: 
\beqs
\|\cR_{\eta}\|_{B(\cY)}\leq C \ep^3,\quad  \, \forall \, |\eta|\geq 2\, 
\eeqs
and which may change from line to line.
Applying $P_2(\eta)$ to %to both sides of %\eqref{resol-HTF},
the resolvent equation $(\lambda-\tilde{\cL}_a(\eta))u=\tilde{f}$ and noticing that $P_2(\eta) P_1(\eta)=\Id_{\cY}+\cR_{\eta},$ we find:  \beqs\label{resol-HTF-1} 
%\bigg(
 \op \left(\begin{array}{cc}
  \lambda- \lambda_{-}^1(x,\xi,\eta)   &  \\
     &  \lambda- \lambda_{+}^1(x,\xi,\eta) 
\end{array}  \right) %\bigg)
(P_2(\eta)u)=P_2(\eta)\tilde{ f}+\cR_{\eta} u\,.
\eeqs
To conclude, it will be enough  to show the following claim:

{\bf{Claim 1:}\label{claim-inverse}}
   % There exists $M>0$ large enough, 
 Let $\ep$ be sufficiently small. For any $|\eta|\geq 2,$ any $\lambda\in \Omega_{\beta,\ep},$ the operators $\op \big(\lambda-
    \lambda_{\pm}^1(\cdot, \cdot, \eta) \big) $ are invertible and their inverses have the bounds %in $L=B(\cY)$
\begin{align}
    \|\big(\op \big(\lambda-
    \lambda_{\pm}^1(\cdot, \cdot, \eta) %({\eta})
    \big) \big)^{-1}\|_{B(L^2(\mR))}\lesssim  %C(\hat{a}_0) 
    \ep^{-1}. \label{ev-HT-p} 
  %   \|\big(\op \big(\lambda-  \tchi_{_{A\ep^2}}\lambda_{-}     \big)^{-1}\|_{B(L_a^2(\mR^2))}\lesssim   \f{1}{\sqrt{A}\ep^3}. \label{ev-HT-m}  
    \end{align}
Indeed, once this claim is shown, we derive from \eqref{ev-HT-p} that
\beqs 
\|P_2(\eta)u\|_{\cY}\lesssim \ep^{-1}\bigg(\|P_2(\eta)\tilde{f}\|_{\cY}+C\ep^3 \|u\|_{\cY}
\bigg).
\eeqs 
%we could find a solution $u$ for the problem $(\lambda-\tilde{\cL}_a(\eta))u=\tilde{f}$by using the invertibility of the operator $P_2(\eta)$ and 
Upon choosing $\ep$ sufficiently small and using   the fact that  $P_1(\eta)P_2(\eta)=\Id+\cR_{\eta},$ we get  that $\|u\|_{\cY}\lesssim \ep^{-1}\|\tilde{f}\|_{\cY}\,.$ The existence %of resolvent problem  $(\lambda-\tilde{\cL}_a(\eta))u=\tilde{f}$ 
also follows from \eqref{ev-HT-p} and the invertibility of 
$P_2(\eta).$
We thus proved the invertibility  of $\lambda-\tilde{\cL}_a(\eta)$  as well the  resolvent bound
$\|(\lambda-\tilde{\cL}_a(\eta))^{-1}\|_{\cY}\lesssim \ep^{-1}.$ 

To get similar results for  the original operator, it suffices to show that $\tilde{\cL}_a(\eta)$  is a small perturbation of  ${\cL}_a(\eta)$ in the sense that, %for any $\eta,$ such that $|\eta|\leq 2,$
\beqs
\|{\cL}_a(\eta)-\tilde{\cL}_a(\eta)\|_{\cY}\lesssim \ep^2
\eeqs
uniformly for $\eta$ such that $|\eta|\geq 2.$ This is shown in Lemma \ref{lem-error}. The proof of Proposition \ref{prop-UTH} is  thus finished up to the proof of the above claim.

\end{proof}

We now come back to the proof of the Claim 1. It will be a  consequence of the following lemma:
\begin{lem}\label{lem-ev-HT}
%Let $\lambda^{1}_{\pm}$ be defined as in \eqref{deflpm} with $\sqrt{\lambda_{\zeta_c}}$ replaced by $\sqrt{\lambda_{\zeta_c}^1}.$
Assume %$M$ large enough and 
$\ep$ to be sufficiently small.  Then for any $\eta$ such that $|\eta|\geq 2$, any $(x, \xi)\in \mR^2,$
it holds that: 
\beq \label{es-ev}
\Re \lambda_{\pm}^1(x, \xi, \eta)\leq -\f{a}{4}=-\f{\hat{a}\ep}{4}.
\eeq
\end{lem}

Let us postpone the proof of this lemma and first finish  the proof of Claim 1. By Lemma \ref{lem-ev-HT}, for any $\lambda$ such that $\Re \lambda >- \hat{a} \ep^3,$ we 
have for $\ep$ small enough, $|\lambda-\lambda_{\pm}|\geq |\Re( \lambda-\lambda_{\pm}^1)|>\f{\hat{a}}{8}\ep.$ By applying Lemma \ref{lem-inverse} in the appendix, we have, upon choosing $\ep$ smaller if necessary, that for any $|\eta|\geq 2,$ the operators $\op\big((\lambda-\lambda_{\pm}^1)(\cdot,\cdot, \eta)\big)$ are invertible and 
\beqs 
 \|\big(\op \big(\lambda-
    \lambda_{\pm}^1(\cdot, \cdot, \eta) %({\eta})
    \big) \big)^{-1}\|_{B(L^2(\mR))}\leq 
   \f{16}{\hat{a}\ep}\, .
\eeqs

\begin{proof}[Proof of Lemma \ref{lem-ev-HT}]

    It suffices to show that  for any $|\eta|\geq 2,$ 
    it holds that 
    \beq\label{Im-lambda1}
    \big|\Im (\sqrt{\lambda_{\zeta_c}^1\tanh \lambda_{\zeta_c}^1})\big|\leq \f{3}{5}a\, .
    \eeq
Indeed, once this is shown, we have 
\begin{align*}
\Re \lambda^1_{\pm}&=-a\, d_c(x)\mp \sqrt{w_c} \, \Im \big(\sqrt{\lambda_{\zeta_c}^1\tanh \lambda_{\zeta_c}^1}\big) \\
&= -a (1-v_c)\mp \sqrt{\gamma-(1-v_c)\p_x Z_c} \, \Im \big(\sqrt{\lambda_{\zeta_c}^1\tanh \lambda_{\zeta_c}^1}\big)\\
&=-a \mp \Im \big(\sqrt{\lambda_{\zeta_c}^1\tanh \lambda_{\zeta_c}^1}\big)+\cO(\ep^3)\leq -\f{a}{4}
\end{align*}
upon choosing $\ep$ small enough.
Note that in the above estimate  we have  used that as consequences of  Theorem \ref{bealetheo}:
    \beqs 
\|v_c\|_{L_x^{\infty}}\lesssim \ep^2, \quad \|\p_x Z_c\|_{L_x^{\infty}}\lesssim \ep^3, \quad \gamma={1-\ep^2}.
\eeqs
We now come back to the proof of \eqref{Im-lambda1} which follows from algebraic computations. By the definition of $\lambda_{\zeta_c}^1$ in \eqref{def-prsy}, one has: 
\beqs 
{\lambda_{\zeta_c}^1}=\sqrt{R+i I}=\colon u+iv, \, \text{  } \,
\eeqs
with $R=\xi^2+\eta^2(1+(\p_x\zeta_c)^2)-a^2, \, I=2 a \xi$
and 
\beq \label{defuv}
u=\sqrt{\f{R+\sqrt{R^2+I^2}}{2}}, \quad v=\f{I}{2u}.
\eeq
 Without loss of generality,  we assume $\xi>0$ so that $I>0, v>0.$ By choosing $\ep$ small enough, so that 
$\|\p_x \zeta_c\|_{L^{\infty}_x}\leq 1/4, \, a=\hat{a}\ep \leq 1/2,$ one has that $R\geq \xi^2+\f{3}{4}\eta^2\geq \xi^2+3$ for any $|\eta|\geq 2.$ Therefore $u\geq \sqrt{R}\geq \sqrt{\xi^2+3},\, 0< v\leq a.$

Since
\beqs
\tanh (u+i v)=\f{\sinh u \cos v+i\cosh u \sin v}{\cosh u \cos v+i\sinh u \sin v}=\f{\sinh (2u)+i\sin(2v)}{2\,(\cosh^2 u-\sin^2 v)},
\eeqs
we have that $\lambda_{\zeta_c}^1\tanh \lambda_{\zeta_c}^1=\tilde{u}+i\tilde{v},$ with
\beqs
\tilde{u}=\f{u\sinh 2u-v\sin(2v)}{2\,(\cosh^2 u-\sin^2 v)}, \qquad \tilde{v}= \f{v\sinh(2u)+u\sin (2v)}{2\,(\cosh^2 u-\sin^2 v)}.
\eeqs
As $a\leq \f12<\f{\pi}{4},$  and $0<v \leq a$, one has
that  $2\,(\cosh^2 u-\sin^2 v)=%\f12 
\cosh 2 u+(1-2\sin^2 v)>\cosh 2 u.$ %we have, as long as $\sin^2 v\leq  \f12,$ 
%which is true for $\ep$ (and thus $a$)small enough, that
which leads to 
\beq\label{prop-tv} 
\tilde{v}\leq v \big(\tanh 2u+\f{2u} {\cosh 2u}\big)\leq 2v.
\eeq
Moreover, by using the fact that $\f{\sinh x}{1+\cosh x}$ is an increasing function on $\mR_{+}$, it holds that 
\beq\label{proptu}
\tilde{u}\geq \f{u\sinh 2u-v\sin 2v}{1+\cosh 2u}\geq \f{\sinh 2\sqrt{3}}{1+\cosh 2\sqrt{3}}u+\cO(\ep^2).
\eeq
Let $\sqrt{{\lambda_{\zeta_c}^1}\tanh \lambda_{\zeta_c}^1}=\alpha+i\beta,$ then
\beqs 
\alpha=\sqrt{\f{\tilde{u}+\sqrt{\tilde{u}^2+\tilde{v}^2}}{2}}, \quad 
\beta=\f{\tilde{v}}{2\alpha}\, .
\eeqs
In view of the  identity \eqref{defuv} and the estimates \eqref{prop-tv}, \eqref{proptu},
%Since $\alpha\geq \sqrt{\tilde{u}}\geq \sqrt{\f{2 u\sinh 2u}{2+\cosh (2u)}},$ %\geq R^{\f14}\geq (2+\xi^2)^{1/4},$  it holds that
we have that
\beqs 
\beta\leq \f{I}{2\alpha {u}} \leq \sqrt{\f{1+\cosh 2\sqrt{3} }%\sqrt{2}
{\sinh %\sqrt{2}
2\sqrt{3}}} \f{I}{2u^{3/2}}+\cO(\ep^2)
\leq \sqrt{\f{1+\cosh 2\sqrt{3}}{\sinh 2\sqrt{3}} }\f{|\xi|}{(3+\xi^2)^{{3}/{4}}} a+\cO(\ep^2).%\leq \f{3a}{5}.
\eeqs
Since $\cosh 2\sqrt{3}\approx 16, \sinh 2\sqrt{3}\approx 16 $ and $\f{|\xi|}{(3+\xi^2)^{{3}/{4}}}\leq \f{\sqrt{6}}{9^{3/4}}\approx 0.47,$ one sees that, 
upon choosing $\ep$ smaller if necessary, 
$\beta \leq \f{3}{5}a,$ which is the desired result.
\end{proof}

\subsection{Proof of Proposition \ref{prop-ITF}}\label{subsec-ITF}
\begin{proof} [Proof of Proposition \ref{prop-ITF}] 
When focusing on the intermediate transverse frequencies 
$\{A\ep^2\leq |\eta|\leq 2\},$ there are two new issues.  The symbol of $\sqrt{G_a[0]}$ 
may vanish and the diagonalization like \eqref{diagnolzation} does not work anymore since $P_2(\eta)$ is not a bounded operator on $L^2(\mR)\times L^2(\mR).$ Moreover, it is not realistic to first prove the invertibility of the operator with constant coefficients and then use perturbation arguments. Indeed, due to the appearance of $v_c \p_x$ terms in the diagonal entries of $L_a,$ there would  be a  loss of derivatives.
In order to prove the resolvent bounds \eqref{ITF}, we shall use  an energy-based approach, achieved through appropriate energy functionals.

For any $\lambda \in \Omega_{\beta, \ep},$ $F\in X_A^I,$ consider the resolvent equation 
 \beq\label{reolveq-inter}
(\lambda I -L_a)\, U=F\, .  %\quad U, F\in X_A^I.
 \eeq
Let us first assume the existence of $U \in X_A^I$ and prove the resolvent estimates: 
\beq\label{resolvent-inter}
\|U\|_{X}\lesssim A^{-1}\ep^{-3}\,\|F\|_{X}. 
\eeq
To get this estimate, we shall  perform energy estimates 
which control quantities that are equivalent to $\|\cdot\|_X,$
%it is necessary to find different 
but are  different  according to frequency regions (both in the longitudinal frequency $\xi$ and the transverse frequency $\eta$).
On the  one hand, when handling the region (most of cases) where the 
the operator $|\f{\na_a}{\langle \na \rangle^{1/2} }|$ 
$(\na_a=(\p_x-a, \p_y)^t)$
is equivalent to $|\sqrt{G_a}[0]|=|\sqrt{\mu_a\tanh \mu_a}|,$ 
it is convenient to use a norm which is equivalent to 
$\|\cdot\|_{L^2}\times \|\sqrt{G_a}[0]\cdot\|_{L^2}.$
On the other hand, when focusing on the region where $\mu_a(\xi,\eta)$ may vanish, 
for instance the set 
\beq \label{def-singset}
\cS_{K, \delta}=\bigg\{(\xi,\eta)\in \mR^2\big|\,|\xi%+i a
|\leq K\ep, \, \delta\leq 
\f{|\eta|}{\sqrt{\xi^2+a^2}}\leq 2\bigg\}, \quad ( \,0<\delta<1 \text{ small, } K(K+1) \in [A/2, A]  ),
\eeq 
we have to instead introduce an alternative norm.
Indeed, in this region, we have  %$\sqrt{G_a[0]}^{-1}\approx \mu_a^{-1}$
 $$(G_a[\zeta_c]-G_a[0])\sqrt{G_a[0]}^{-1}=-\p_x\zeta_c \f{\p_x-a}{\sqrt{G_a[0]}}+\text{good operator},$$ which is not a bounded operator on $L^2(\mR^2
).$ Nevertheless, since the singularity happens only in a  low $\xi$  frequency region, 
we can simply use the norm  $\|\cdot\|_{L^2}\times \|\na_a\cdot\|_{L^2}.$
%{\color{red} the above paragraph is not very clear}

Let $\tilde{\mathbb{I}}_{\delta}:\mR\rightarrow \mR,$  %be the characteristic function on  
$\pi_{s}(\cdot), \, \pi_r(\cdot): \mR^2\rightarrow \mR$ be  respectively the characteristic functions  of the sets 
\beq \label{def-reguset}
\mR\backslash[-\delta, \delta] \,, \qquad  \cS_{K, \delta}\, , \qquad \cS_{K, \delta}^c=\big\{(\xi,\eta)\in \mR^2\big|\,  |\xi|\leq  \delta, \, A\ep^2\leq 
{|\eta|}\leq 2\big\} \backslash \cS_{K, \delta} \eeq
respectively. We split the solution $U$ to the problem \eqref{reolveq-inter} into three parts $U=U_h+U_s+U_r,$ where
\beqs 
U_h= \tilde{\mathbb{I}}_{\delta}(D_x) U, \quad U_s= \pi_{s}(D) U, \quad U_r= \pi_{r}(D)U 
\eeqs
denote a  high longitudinal  frequency part, a  low frequency ‘singular' part and a  low frequency ‘regular’ part. In order to prove \eqref{resolvent-inter}, it is enough  to show
%The resolvent estimate \eqref{resolvent-inter} could be derived from 
the following estimates on $U_h, U_s $ and $ U_r:$
\begin{align}
    \|U_h\|_{X}&\lesssim_{\delta} \ep^{-1}\|F\|_{X}+\ep \|(U_s, U_r)\|_{X}, \label{es-uh} \\
     \|U_s\|_{X}&\lesssim_{\delta}  \ep^{-1}\|F\|_{X}+K^3\ep^2 \|(U_h, U_r)\|_{X}, \label{es-us}\\
       \|U_r\|_{X}&\lesssim K^{-2} \ep^{-3}\big(\|F\|_{X}+\ep^2 \|(U_h, U_s)\|_{X}\big). \label{es-ur}
\end{align}
Indeed, we derive  from %once have been shown, we have
the above estimates  that
\beqs 
\|U\|_{X}\leq C K^{-2}\ep^{-3}  \|F\|_{X} + C(K\ep+K^{-2})\|U\|_{X}
\eeqs
where $C>0$ is independent of $K$ and $\ep.$  The estimate \eqref{resolvent-inter} then follows by choosing first $K$ large enough, and then $\ep$ sufficiently small.
We will prove \eqref{es-uh}-\eqref{es-ur} in the following subsections. 

Once the above estimates are obtained, to finish the proof of the Proposition \ref{prop-ITF}, we just need to prove  the existence for the resolvent equation \eqref{reolveq-inter}, that is that  any element of $%\lambda\in 
\Omega_{\beta,\ep}$  lies in the resolvent set of $L_a$ in the space $X_A^I.$  As in  \cite{Pego-Sun}, we can invoke
 an abstract result from Spectral Theory (Theorem III.6.7, \cite{Book-Kato}) in order to reduce the problem to the
 proof of  the existence of one element of the resolvent set in $\Omega_{\beta, \ep}$.
We shall thus prove that $\lambda=1$  lies in the resolvent set of $L_a.$
We split the operator $L_{a}$ into two parts: 
$L_a=L_{a}^0+L_{a}^1,$ where
%the matrix with constant coefficient: 
\beq\label{def-La0-1}
 {L}_{a}^0
 =\left( \begin{array}{cc}
   \p_x-a  &  G_a[0]   \\[5pt]
 -  \gamma& \p_x-a
\end{array}\right), \,\, L_a^1
 =\left( \begin{array}{cc}
   -(\p_x-a)(v_c\cdot)  &  G_a[\zeta_{c}] -G_a[0]  \\[5pt]
   d_c \p_x Z_c & -v_c(\p_x-a)
\end{array}\right).
\eeq
It is then sufficient to prove on the  one hand that $\Id- L_{a}^0$ is invertible in $X_A^I$ with the norm of the inverse uniformly bounded  in $\ep$ and on the other hand that $(\Id- L_{a}^0)^{-1}L_a^1\in B(X_A^I)$ with the norm 
\beq\label{smallperLa1}
\|(\Id- L_{a}^0)^{-1}L_a^1\|_{ B(X_A^I)}\lesssim \ep^2.
\eeq
Since $L_a^0$ is an operator with constant coefficient, the invertibility of the operator $\Id-L_{a}^0$ is equivalent to the invertibility of the complex valued matrix 
$\Id_{2} - L_{a}^0(\xi,\eta),$ $L_{a}^0(\xi,\eta)$ being the symbol of $L_{a}^0(D).$ 
%\beqs  L_{a}^0(\xi,\eta)= \eeqs

Denote $\lambda_0(\xi,\eta)=(\mu_a \tanh\mu_a)(\xi,\eta)$ the symbol associated to  $G_a[0],$ and 
$$\lambda_{\pm}^0(\xi,\eta)=i(\xi-ia)\pm \sqrt{-\gamma \lambda_0(\xi,\eta)}$$
the eigenvalues of $L_{a}^0(\xi,\eta).$ By \eqref{Relambdapm} in Appendix B, 
it holds that $\Re \lambda_{\pm}^0\leq a,$ for any $(\xi,\eta)\in \mR,$ provided $0<a<\f{\pi}{6}.$ Consequently, we find that
$\Re( 1- \lambda_{\pm}^0)$ never vanishes, which ensures the invertibility of 
$\Id_{2} - L_{a}^0(\xi,\eta).$ Moreover, straightforward computation show that
\beq\label{inverse1-La0}
\begin{aligned}
\big(\Id_{2} - L_{a}^0(\xi,\eta)\big)^{-1}&= \f{1}{(1-\lambda^0_+)(1-\lambda^0_-)}\left( \begin{array}{cc}
  \f{(1-\lambda^0_+)+(1-\lambda^0_-)}{2}  & \f{(1-\lambda^0_-)-(1-\lambda^0_+)}{2}  \sqrt{-\lambda_0/\gamma}   \\[5pt]
 -  \gamma&  \f{(1-\lambda^0_+)+(1-\lambda^0_-)}{2} 
\end{array}\right) \\
&=\colon \left(  \begin{array}{cc}
A_{11}, & A_{12}\\
A_{21} & A_{22}
\end{array}\right). 
\end{aligned}
\eeq
In light of this explicit expression, we have, by remembering  $X_A^I=\mathbb{I}_2\tilde{\mathbb{I}}
_{A\ep^2}(D_y)(L^2(\mR^2)\times H_{\star}^{\f12}(\mR^2)),$  that %Thanks to the fact $\|\sqrt{G_a[0]}f\|_{L^2}\lesssim \|f\|_{H_{*}^{1/2}},$ we have that:
\begin{align}\label{id-la0-inverse}
    \|\big(\Id_{2} - L_{a}^0(\xi,\eta)\big)^{-1}\|_{B(X_A^I)}\lesssim \max \sup_{\xi\in \mR, |\eta|\leq 2} \bigg\{ \f{1}{|1-\lambda^0_+|}\,,  \f{1}{|1-\lambda^0_-|}\,, \f{(1+\xi^2)^{1/4}}{|(1-\lambda^0_+)(1-\lambda^0_-)|}\bigg\}<+\infty.
\end{align}
Note that we have used the facts $\|\sqrt{G_a[0]}f\|_{L^2}\lesssim \|f\|_{H_{*}^{1/2}},$
and $|\Im (1-\lambda^0_{\pm}) |\thickapprox (1+\xi^2+\eta^2)^{\f12}$ for $|\xi|\gg 1.$ 
%and $|\eta|\leq 2.$  

To show \eqref{smallperLa1}, it suffices to prove the following estimates
\begin{align*}
    \|A_{11}(\p_x-a)+A_{12}(d_c\p_xZ_c  )\|_{B(L^2,L^2)}\lesssim \ep^2,  \\
    \|A_{11}(G_a[\zeta_c]-G_a[0])+A_{12}(v_c(\p_x-a))\|_{B(H_{*}^{\f12},L^2)}\lesssim \ep^2, \\
    \|A_{21}(\p_x-a)+A_{22}(d_c\p_x Z_c)\|_{B(L^2,H_{*}^{\f12})}\lesssim \ep^2, \\
   \|A_{21}(G_a[\zeta_c]-G_a[0])+A_{22}(v_c(\p_x-a))\|_{B(H_{*}^{\f12},H_{*}^{\f12})} \lesssim \ep^2.
\end{align*} 
Thanks to the the explicit expressions \eqref{inverse1-La0} and the estimates%$\|v_c\|_{W_x^{1,\infty}}\lesssim \ep^2,$ 
\beqs
\|v_c\|_{W_x^{1,\infty}}\lesssim \ep^2, \quad \|d_c\p_x Z_c\|_{{W_x^{1,\infty}}}\lesssim \ep^3, %\quad \big\|\f{1}{1-\lambda^0_{\pm}}(D)\big\|_{B(H^{-1},L^2)}\lesssim 1,
\eeqs
as well as the estimate
\begin{align*}
 \qquad     
\|\bI_2(D_y) (G_a[\zeta_c]-G_a[0])\|_{B(H_{*}^{\f12}(\mR^2), L^2(\mR^2))}\lesssim \ep^2
\end{align*}
whose proof is detailed in Lemma \ref{lem-Ga-Ga0-g},
the above four estimates are the consequences of the following algebraic computation:
%\, \|\p_x v_c\|_{L_x^{\infty0}}\lesssim \ep^3 $ %$\|\sqrt{G_a[0]}f\|_{L^2}\lesssim \|f\|_{H_{*}^{1/2}},$ we have, {\color{blue} by recalling that that
\begin{align*}
   % \|\big(\Id_{2} - L_{a}^0(\xi,\eta)\big)^{-1}\|_{B(X_A^I)}\lesssim 
   \sup_{\xi\in \mR, |\eta|\leq 2}  \big(\langle \xi\rangle+\sqrt{\lambda_0}\langle \xi\rangle^{1/2}%(1+\xi^2)^{1/4}\rangle
   \big)\max\bigg\{ \f{1}{|1-\lambda^0_+|}\,,  \f{1}{|1-\lambda^0_-|}\,, \f{%(1+\xi^2)^{1/4}
   1}{|(1-\lambda^0_+)(1-\lambda^0_-)|}\bigg\}<+\infty.
\end{align*}
%Note that for the boundedness of the last quantity, we have used that $|\Im (1-\lambda^0_{\pm}) |\thickapprox \langle \xi\rangle$ for $|\xi|\gg 1$ and $|\eta|\leq 2.$  %Thus \eqref{smallperLa1} 
 
 To summarize, we have shown that $1\in \rho(L_a; X_A^I)$ and thus the proof of Proposition \ref{prop-ITF} is complete.
\end{proof}

Before going into the details for the proof of \eqref{es-uh}-\eqref{es-ur} in the following three subsections, we first state  some algebraic properties for $\lambda_{\pm}^{0}$ whose proof is given in the appendix. Since they  depend on the localization in  frequencies, we  split the frequency domain  $R_{\eta}^L=\{(\xi, \eta)\in \mR^2\,| |\eta|\leq 2\}$ into the following subsets:
\begin{align*}
   & R_{\xi}^{H}=\{ (\xi, \eta)\in \mR^2 \,\big||\xi|\geq \delta, \, |\eta|\leq 2 \}, \\
   &  R_{\xi}^{I}=\{ (\xi, \eta)\in \mR^2 \,\big|K\ep\leq |\xi|\leq \delta, \, |\eta|\leq 2 \}, \\
    &  R_{\eta,1}^{I}=\{ (\xi, \eta)\in \mR^2 \,\big| |\xi|\leq K\ep, \, \delta |\xi+i a| \leq |\eta|\leq 2 \}, \\
    &  R_{\eta,2}^{I}=\{ (\xi, \eta)\in \mR^2 \,\big| |\xi|\leq K\ep, \, A\ep^2\leq |\eta|\leq  \delta |\xi+i a| \}, 
\end{align*}
where $K(K+1)\in [{A}/{2}, A].$ 
We begin with a lemma about  the properties of the  symbol of $\sqrt{-G_a[0]} \,.$
\begin{lem}\label{lem-sym-Ga0}
Let $\mu_a(\xi,\eta)=\sqrt{(\xi+ia)^2+\eta^2}.$ It holds that: 
\begin{align}\label{essym-Ga0}
    0\leq \Re \sqrt{-\mu_a \tanh \mu_a} \leq   
    \left\{ \begin{array}{c}
        a(1-C\delta), \qquad \forall \, (\xi,\eta)\in R_{\xi}^H \cup  R_{\eta,1}^{I}\, , \\[5pt]
         a(1-C A \ep^{2}), \qquad \forall \,(\xi,\eta)\in  R_{\xi}^{I}\cup R_{\eta,2}^{I}\,.
     \end{array} \right.
\end{align}
\end{lem}
As a direct consequence of the above lemma, we have the following:
\begin{prop}\label{lem-lampm1}
    Let $\lambda_{\pm}^0(\xi,\eta)= i(\xi+ia)\pm \sqrt{1-\ep^2} \sqrt{-\mu_a \tanh \mu_a}(\xi,\eta),$ for any 
  $\lambda \in \Omega_{\beta, \ep},$
   such that $\Re \lambda>-\beta \ep^3,$  it holds that, 
  \beq\label{lda-ldap}
     \Re (\lambda-\lambda_{-}^0) \geq a-\beta\ep^3, \qquad  \forall \, (\xi,\eta)\in R_{\eta}^{L}.
  \eeq
 \beq
      \begin{aligned}
     \Re (\lambda-\lambda_{+}^0) \gtrsim \left\{ \begin{array}{c}
        \delta   a, \qquad \forall\, (\xi,\eta)\in R_{\xi}^H \cup  R_{\eta,1}^{I} ,  \\[4pt]
          A \ep^{3}, \qquad \forall\, (\xi,\eta)\in  R_{\xi}^{I}\cup R_{\eta,2}^{I} \, .
     \end{array} \right. 
    \end{aligned}
\eeq
\end{prop}
\begin{proof}
It follows from Lemma \ref{lem-sym-Ga0} %the fact $\gamma=1-\ep^2$ %direct computation 
and the identity 
\beqs
      \Re (\lambda-\lambda_{\pm})=\Re \lambda -\bigg(-a\pm \sqrt{1-\ep^2\,}\, \Re \sqrt{-\mu_a \tanh \mu_a} \bigg).
\eeqs
\end{proof}
\subsubsection{Proof of \eqref{es-uh} }\label{subsub-h}
 We focus in this subsection on the %region that we  called the 
 uniform high longitudinal frequency region
 $\{|\xi|\geq \delta\}$ on which $m\big(\sqrt{G_a[0]}\big)=\sqrt{\mu_a\tanh \mu_a}(\xi)\approx_{\delta} m\big(\f{\na_a}{\langle \na \rangle^{1/2}}\big)$ (we denote $m(\cdot)$ the symbol of a Fourier multiplier). Therefore, for any $f\in \tilde{\mathbb{I}}_{\delta} H_{*}^{1/2},$
it holds that $\|f\|_{H_{*}^{1/2}}\approx \|\sqrt{G_a[0] }f\|_{L^2}.$ 
We thus define $(\tilde{U}, \tilde{F}) =\diag \,(1, \sqrt{G_a[0]}\,)(U,F)$ and aim to show the following estimate for $\tilde{U}_h=: \tilde{\mathbb{I}}_{\delta}\, \tilde{U}:$ 
%\diag \,(1, \sqrt{G_a[0]}\,)U^h:$
%To show \eqref{es-uh}, it suffices to prove: 
\beq \label{es-uh-1}
 \|\tilde{U}_h\|_{\cX}\lesssim_{\delta} \ep^{-1}\|\tilde{F}\|_{\cX}+\ep \|{U}\|_{X}, \quad (\cX=L^2(\mR^2)\times L^2(\mR^2)).
\eeq

 Taking  $\tilde{\mathbb{I}}_{\delta}(D_x)\,\diag \,(1, \sqrt{G_a[0]}\,)$ on both sides of \eqref{reolveq-inter} and using 
 $\tilde{\mathbb{I}}_{\delta}\circ \tilde{\mathbb{I}}_{\delta}=\tilde{\mathbb{I}}_{\delta},$
 we find that $\tilde{U}^h$ satisfies
 \beq\label{eq-tildeU}
\big(\lambda-\tilde{\mathbb{I}}_{\delta}\,\cL_a^0\,\tilde{\mathbb{I}}_{\delta}\big)\, \tilde{U}_h=\tilde{F}+ H(U) 
\eeq
where 
$$H(U)=\tilde{\mathbb{I}}_{\delta} \, \cL_a^1 \, {U}+\tilde{\mathbb{I}}_{\delta}\, \diag \,(1, \sqrt{G_a[0]})\,  \cL_a^2 \,{\mathbb{I}}_{\delta}\,   {U}$$ and 
\beq\label{defcL0-1} 
\small \cL_a^0= \left( \begin{array}{cc}
   (\p_x-a)(d_c\cdot)  &  \sqrt{G_a[0]} \\[5pt]
 -\sqrt{G_a[0]} ( w_c \cdot )& d_c(\p_x-a)%+
\end{array}\right), %
\quad \cL_a^1= \left( \begin{array}{cc} 0  & G_a[\zeta_c]-G_a[0]  \\[5pt] 0 & 0\end{array}\right), 
\quad \cL_a^2= \left( \begin{array}{cc}
   -(\p_x-a)(v_c \cdot)  & 0\\[5pt]
d_c \p_x Z_c  &  -v_c (\p_x-a)\end{array}\right).
\eeq
Denote $\tilde{U}_h=(\tilde{\zeta}_h, \tilde{\vp}_h).$
Taking the real part of the inner product between the equation \eqref{eq-tildeU} and $(w_c\, \tilde{\zeta}_h, \tilde{\vp}_h)$ in $\cX,$ we find the identity
\beq\label{EI-high}
\begin{aligned}
&\Re \lambda \int w_c |\tilde{\zeta}_h|^2+ |\tilde{\vp}_h|^2\, \d x \d y + a \int d_c \big(w_c |\tilde{\zeta}_h|^2+ |\tilde{\vp}_h|^2\big) \, \d x \d y\\
& \qquad  +\Re\int \big(\sqrt{G_a[0]}-\sqrt{G_{-a}[0]}\big) ( w_c \tilde{\zeta}_h) \, \overline{\tilde{\vp}_h} \, \d x\d y
   \\
   & = \cJ_1+\cJ_2+\cJ_3 
\end{aligned}
\eeq
where 
\begin{align*}
   \cJ_1&= \f12 \,\Re \int \big(2  w_c\p_x d_c-\p_x({d_c w_c})\big)|\tilde{\zeta}_h|^2-\p_x d_c| \tilde{\vp}_h|^2\, \d x\d y, \\
   \cJ_2&= \Re \int {\tilde{\vp}_h } \,\, \overline{\sqrt{G_{-a}[0]} \big[\tilde{\mathbb{I}}_{\delta} , w_c\big] \tilde{\mathbb{I}}_{\delta}\tilde{\zeta}_h}-d_c \tilde{\vp}_h \,\, \p_x\big[\tilde{\mathbb{I}}_{\delta} , w_c\big] \tilde{\mathbb{I}}_{\delta} \overline{\tilde{\zeta}_h}\, \d x \d y ,\\
   \cJ_3& =\Re \int (\tilde{F}+ H(U))\cdot \overline{(w_c \tilde{\zeta}_h, \tilde{\vp}_h)}\, \d x\d y\, .
\end{align*}
By using the Cauchy-Schwarz inequality, the real part of the left hand side of \eqref{EI-high} is larger than 
\begin{align*}
   \big( \Re \lambda+ a -\sup_{x\in \mR} \sqrt{w_c}(x)\sup_{ |\xi|\geq \delta, \, |\eta|\leq 2 } \, \Im \sqrt{\mu_a \tanh \mu_a}(\xi,\eta)\big) \big(\|\sqrt{w_c}\tilde{\zeta}_h\|_{L^2}^2+\|\tilde{\vp}_h\|_{L^2}^2\big),
\end{align*}
which  thanks to {Proposition \ref{lem-lampm1}, %provided $\ep$ be sufficiently small, 
is greater than
\beqs 
  a(1-C\delta) \big(\|\sqrt{w_c}\tilde{\zeta}_h\|_{L^2}^2+\|\tilde{\vp}_h\|_{L^2}^2\big)\geq \f{\hat{a}\ep}{2} \big(\|\sqrt{w_c}\tilde{\zeta}_h\|_{L^2}^2+\|\tilde{\vp}_h\|_{L^2}^2\big)
\eeqs
for any $ \lambda \in \Omega_{\beta,\ep}.$ We now control the terms $\cJ_1-\cJ_3$ in the right hand side of \eqref{EI-high}.
Since $\|\p_x(d_c, w_c)\|_{L_x^{\infty}}\lesssim \ep^{3},$ it holds that for sufficiently small $\ep,$
\beqs 
w_c(x)=\gamma-d_c\p_x Z_c={1-\ep^2}+\cO(\ep^3)\in [1/2, 1], \quad \forall\, x \in \mR\,,
\eeqs
we thus have that 
\beq 
\cJ_1\lesssim \ep^3 \big(\|\sqrt{w_c}\tilde{\zeta}_h\|_{L^2}^2+\|\tilde{\vp}_h\|_{L^2}^2\big).
\eeq
Next, it follows from the commutator estimate \eqref{es-commutator} stated in Lemma \ref{lem-commutator}  below,  that we use for $s=0$,  that
\beqs 
\|\big(\mathbb{I}_{2}(D_{y}) \sqrt{G_{-a}[0]}, \p_x\big) \big[\tilde{\mathbb{I}}_{\delta} , w_c\big] \tilde{\mathbb{I}}_{\delta}\|_{B(L^2)}\leq C \|\cF_{x\rightarrow \xi}(w_c)\|_{L_{\xi}^1}\lesssim \ep^2,
\eeqs 
where $C=\delta+\sup_{|\xi|\leq \delta, |\eta|\leq 2} |\sqrt{\mu_a\tanh \mu_a}(\xi,\eta)|.$
Consequently, by again the Cauchy-Schwarz inequality, it holds that
\beqs 
\cJ_2 \lesssim \ep^2 \big(\|\sqrt{w_c}\tilde{\zeta}_h\|_{L^2}^2+\|\tilde{\vp}_h\|_{L^2}^2\big).
\eeqs
Finally, thanks to Lemma \ref{lem-Ga-Ga0-g} and Lemma \ref{lem-clatla} in the appendix, we have
\beqs 
\|H(U)\|_{\cX}\lesssim \ep^2 \|U\|_{X}.
\eeqs
The Cauchy-Schwarz inequality then stems: 
\beqs 
\cJ_3\lesssim  (\|\tilde{F}\|_{\cX}+\ep^2 \|U\|_{X})\big(\|\sqrt{w_c}\tilde{\zeta}_h\|_{L^2}+\|\tilde{\vp}_h\|_{L^2}\big).
\eeqs
To summarize, it holds that
\beqs 
\cJ_1+\cJ_2+\cJ_3\lesssim \ep^2 \big(\|\sqrt{w_c}\tilde{\zeta}_h\|_{L^2}^2+\|\tilde{\vp}_h\|_{L^2}^2\big)+ (\|\tilde{F}\|_{\cX}+\ep^2 \|U\|_{X})\big(\|\sqrt{w_c}\tilde{\zeta}_h\|_{L^2}+\|\tilde{\vp}_h\|_{L^2}\big).
\eeqs
Since $\|\tilde{U}_h\|_{\cX}^2\approx \|\sqrt{w_c}\tilde{\zeta}_h\|_{L^2}^2+\|\tilde{\vp}_h\|_{L^2}^2,$ we find eventually that 
\beqs 
\|\tilde{U}_h\|_{\cX}\lesssim \ep \|\tilde{U}_h\|_{\cX}+ \ep^{-1}\|\tilde{F}\|_{\cX}+\ep \|U\|_{X},
\eeqs
 which yields  \eqref{es-uh-1} by assuming that  $\ep$ is  sufficiently small. 
 
In the next lemma, we state a commutator estimate that is used in the above proof.
It is a simple generalization of the Proposition 5.1 of \cite{Pego-Sun} in one dimension:
\begin{lem}\label{lem-commutator}
    Let $\mathcal{A}, \mathcal{B}, \mathcal{W} $ be three Fourier multipliers on the space $L^2(\mR^2)$ with symbols $A, B, W$ and let $f\in C^{\infty}(\mR)$ and $\cF(f)$ be sufficiently localized, for any $s\in \mR$ such that $C_{s}$
     and $C_{f,s}$ below are finite,  it holds that 
    \beq \label{es-commutator}
\|\mathcal{A}[ \mathcal{B}, f] \mathcal{W} \|_{B(L^2(\mR^2))}\lesssim C_s \, C_{f,s}
    \eeq
   where    
$$C_s=\sup_{(\xi,\xi',\eta)\in \mR^3}\f{A(\xi,\eta) \big(B(\xi,\eta)-B(\xi',\eta)\big)W(\xi',\eta)}{|\xi-\xi'|^s}, \qquad  C_{f,s}= \int_{\mR} |\xi|^s |\cF(f)|(\xi) \,\d \xi.$$
\end{lem}
%Let us mention that the estimate \eqref{es-commutator} will be used only for those $s$ such that $C_s<+\infty.$

\subsubsection{Proof of \eqref{es-us} }
In this subsection, we prove the estimate \eqref{es-us} for the singular part $U_s=\pi_s({D}) U.$ Since $\pi_s$ localizes in  the low frequency region and $\|\pi_s(D) f\|_{H_{*}^{1/2}}\approx \|\na_a \pi_s(D) f \|_{L^2},$ it is convenient to split $L_a$ into two parts: 
\beqs 
 {L}_a= \left( \begin{array}{cc}
   \p_x-a  &  -\Delta_a   \\[5pt]
 -\gamma & \p_x-a
\end{array}\right)+\left( \begin{array}{cc}
   -(\p_x-a)(v_c\cdot)  &  G_a[\zeta_{c}] +\Delta_a  \\[5pt]
 - \p_x Z_c & -v_c(\p_x-a)
\end{array}\right)=\colon \tilde{L}_a^0+\tilde{L}_a^1\, ,
\eeqs
where $\Delta_a=(\p_x-a)^2+\p_y^2.$
Applying  the Fourier multiplier $\pi_s(D)$ on both sides of \eqref{reolveq-inter}, we find that $U_s$ solves the equation: 
\beqs 
  \big(\lambda-\tilde{L}_a^0 \big)U_{s}=\pi_s(D)(F+\tilde{L}_a^1 U)\, .
\eeqs
%Let $V_s=\colon e^{\lambda t} U_S
Denote $(U_{\lambda}, F_{\lambda})=e^{\lambda t}(U, F)\in  C([0,\infty); X)^2.$
Since $e^{\lambda t} U_s$ solves the evolution problem
 \beq\label{evoeq}
  \big( \pt  -  \tilde{L}_a^0 \big)V=\pi_s(D)(F_{\lambda} + \tilde{L}_a^1 U_{\lambda} )\,,
    \eeq
we first prove some a priori estimates for the equation \eqref{evoeq}.
Note that we could perform the energy estimates directly on the eigenvalue problem, 
but the structure is somewhat more clear if we go through the time evolution problem.

Let us set  $V=(V_1, V_2)^t$ and define the energy functional $$\cE_s (V(t))=\colon \f12 \int \big(\gamma|V_1(t)|^2+|\na_a V_2(t)|^2\big)\, %e^{2ax}
\, 
\d x \d y\, . $$ 
We get, after performing the straightforward energy estimate for the equation
\eqref{evoeq} in the space $L^2\times \na_a ^{-1}L^2,$ that
\begin{align}\label{EI}
\pt \,\cE_s(V%^{\ell}
)& = \cI_1+\cI_2 +\cI_3
\end{align} 
where
\begin{align*}
    \cI_1&=-2a\, \cE_s(V)+ 2 a \gamma\, \Re\int  V_1\, (\p_x-a)\overline{ V_2} \,
    \d x \d y,\\
    \cI_2&=\Re \int 
\big(F_{\lambda 1}  \overline{ V_1}+ \na_a F_{\lambda 2} \cdot \na_a  \overline{ V_2} \big) \,
\d x\d y,\\
\cI_3&= \Re \int 
\big(S_1 \overline{ V_1}+ \na_a S_2 \cdot \na_a \overline{ V_2}\big) \,
\d x\d y,
\end{align*}
and 
\beqs
S_1=\pi_s({D}) \big(-(\p_x-a)(v_c U_{\lambda 1})+(G_a[\zeta_c]+\Delta_a)U_{\lambda 2}\big)  ; \quad  S_2=-\pi_s({D})\big(d_c \p_x Z_c U_{\lambda 1}+v_c (\p_x -a)U_{\lambda 2}\big) .
\eeqs
Since $\pi_s(\xi,\eta)$ is supported in   $\{|\eta|\geq \f{\delta}{2}|\xi+ia|\},$ we have
$$\|(\p_x-a) V_2\|_{L^2}\leq \f{1}{\sqrt{1+{\delta^2}/{4}}}\|\na_a V_2\|_{L^2}.$$
It then follows from the Cauchy-Schwarz inequality that 
\beq\label{es-I1}
\cI_1\leq -2a\bigg(1-\sqrt{\f{\gamma}{1+\delta^2/4}}\bigg) \cE_s (V)\leq- \f{\delta^2a} {4} \cE_s(V). 
\eeq
Next, by using again the Cauchy-Schwarz inequality 
\beq \label{es-I2}
\cI_2\lesssim  \|(F_{\lambda1},\na_a F_{\lambda2})\|_{L^2} \|(V^{\ell}, \na_a V_2^{\ell})\|_{L^2}\lesssim_{\gamma}\|F_{\lambda}\|_{X}  \sqrt{\cE_a}(V). 
\eeq
We now estimate the term $I_3$ which could be bounded as $I_3\leq \f{1}{\gamma} \|(S_1, \na_a S_2)\|_{L^2} \sqrt{\cE_s}(V).$   
Recalling that $v_c=\cO_{L_x^{\infty}}(\ep^2), \, \p_x Z_c= \cO_{L_x^{\infty}}(\ep^3)$ %$\|v_c\|_{W_x^{1,\infty}}\lesssim \ep^2, \quad \|\p_x Z_c\|_{L_x^{\infty}}\lesssim \ep^3$
and  that $\pi_s(\xi, \eta)$ is supported on the region $\cS_{K,\,\delta}=\{|\xi|\leq K\ep,\, \delta|\xi|\leq |\eta|\leq 2|\xi+ia| \},$
 one readily gets that:
\beq\label{es-1}
\begin{aligned}
&\|\pi_s({D}) (\p_x-a)(v_c U_{\lambda 1})\|_{L^2}+\|\pi_s({D})\na_a\big(d_c \p_x Z_c U_{\lambda 1}\big)\|_{L^2}\lesssim K\ep^3 \|U_{\lambda1}\|_{L^2} ,\\
&\quad \|\pi_s({D})\na_a (v_c (\p_x -a)U_{\lambda 2})\|_{L^2}\\
&\lesssim \|\pi_s({D})\na_a\p_x\big(  v_c U_{\lambda 2}\big)\|_{L^2}+\|\pi_s({D})\na_a(\p_x v_c U_{\lambda 2})\|_{L^2}+K\ep^2\|v_c U_{\lambda 2}\|_{L^2}\\
&\lesssim
(K^2\ep^2 \|v_c\|_{L_x^{\infty}}+K\ep \|\p_x v_c\|_{L_x^{\infty}})\|U_{\lambda2}\|_{L^2}\lesssim K^2\ep^3 \|U_{\lambda2}\|_{H_{*}^{\f12}}.
\end{aligned}
\eeq
Moreover, we prove in Lemma  \ref{lem-G+Delta} that: 
\beqs 
\|\pi_{s}({D})  (G_a[\zeta_c]+\Delta_a)\|_{B(L^2)}\lesssim (K\ep)^3,
\eeqs
which, combined with \eqref{es-1}, yields 
\beqs 
\|(S_1, \na_a S_2)\|_{L^2}\lesssim (K\ep)^3 \|U_{\lambda}\|_{X}.
\eeqs
We thus conclude that 
\beq\label{es-I3}
\cI_3\lesssim (K\ep)^3 \|U_{\lambda}\|_{X} \sqrt{\cE_s}(V).
\eeq
Plugging the estimates \eqref{es-I1}-\eqref{es-I3} into \eqref{EI} and 
 applying the Gr\"onwall inequality, we find
\beq\label{EnergyIneq}
%\|V^{\ell}(t)\|_{X_a}\approx 
\sqrt{\cE_s}(V(t)) \leq e^{-\delta^2 a t/4} \sqrt{\cE_a}(V(0)) %\|V^{\ell}(0)\|_{X_a}
+C \int_0^t  e^{-\delta^2 a (t-s)/4} \big( \|F_{\lambda}(s)\|_{X}+ (K\ep)^3 \|U_{\lambda}(s)\|_{X}\big) \,\d s.
\eeq
As $e^{\lambda t} U_s$ solves \eqref{evoeq} with $(U_{\lambda}, F_{\lambda})=e^{\lambda t}(U, F),$
we then derive that %$U_s,$
\beqs
\sqrt{\cE_s}(U_s) 
%\|e^{\lambda t} u^{\ell}\|_{X_a}
\leq e^{-\alpha_{\lambda}t}\sqrt{\cE_s}(U_s) +C \big(\|F\|_{X}+ K^3 \ep^3 \|U\|_{X}\big)  \int_0^t e^{-\alpha_{\lambda}(t-s)}\, \d s\, 
\eeqs
where $\alpha_{\lambda}=\Re \lambda+\delta^2 a/4. $ As $a=\hat{a}\ep$ and $\Re \lambda>-\beta \ep^3,$ one has that $\alpha_{\lambda}\geq \delta^2 a/8>0 $ as long as $\ep$ being small enough.
Consequently, by taking the limit $t \rightarrow + \infty$, we find
\beqs 
\sqrt{\cE_s}(U_s) \lesssim \ep^{-1}\|F\|_{X}+K^3\ep^2 \|U\|_{X},
\eeqs
which leads to \eqref{es-us} since  
$\sqrt{\cE_s}(U_s)\approx \|U_s\|_{X}.$

\subsubsection{Proof of \eqref{es-ur}}\label{subsub-l}
In this subsection, we  estimate $U_r=\pi_r(D)U$ which is localized in   $\{(\xi,\eta)\in \mR\big|\, |\xi|\leq \delta, A\ep^2\leq |\eta|\leq 2\} \backslash \cS_{K, \delta}$ and $\|U_{r}\|_{X}\approx \|\diag \big(1, \sqrt{G_a[0]}\big)U_r\|_{\cX}.$ %In order to detect more carefully
Since the damping mechanism is weaker in this region (proportional to $\ep^3$), 
we need to use  an energy functional  which involves the background surface wave $\zeta_c$ and is equivalent to 
the norm $\|\diag \big(1, \sqrt{G_a[0]}\big)U_r\|_{\cX}.$

Taking the Fourier multiplier $\pi_r(D)$ on both sides of \eqref{reolveq-inter}, we find that $U_r$ solves the equation
\beqs 
  \big(\lambda-\pi_r(D) L_a \pi_r(D) \big)U_{r}=\pi_r(D)\big(F+%\tilde
  {L}_a^1 (U_h+U_s)\big),
\eeqs
where $L_a^1$ is defined in \eqref{def-La0-1}.
%\beq\label{deftla}  \tilde{L}_a= \left( \begin{array}{cc}   -(\p_x-a)(v_c \cdot)  & G_a[\zeta_c]-G_a[0]\\[5pt]d_c \p_x Z_c  &  -v_c (\p_x-a)\end{array}\right). \eeq
Denote $W_{\lambda}=\pi_r(D)  e^{\lambda t}\big(F+{L}_a^1 (U_h+U_s)\big).$ As in the estimate for $U_s,$ we consider the evolution problem: 
\beq\label{evo-ur}
  \big(\p_t- \pi_r (D) L_a \pi_r (D)\big)V= W_{\lambda}. %\pi_r(D)\big(F+\tilde{L}_a (U_h+U_s)\big),
\eeq
We shall use  the   multiplier $ Q_a(D)=\sqrt{1-\tanh^2 \mu_a(D)}$
and  the energy functional
\beqs 
\cE_r (V(t))=\colon \f12 \int \bigg(\gamma\,|V_1(t)|^2+|\sqrt{G_a [0]}V_2(t)|^2+\zeta_c\big|\na_a Q_a(D)V_2(t)\big|^2\bigg)\, 
\, \d x \d y.
\eeqs
By computing the time derivative of  the above functional and then by using  the equation \eqref{evo-ur}, we find the energy identity
\begin{align*}
    \pt \,\cE_r (V)=\cK_1+\cdots \cK_6\, ,
\end{align*}
where 
\begin{align*}
    \cK_1&=-a \int( 1-v_c(x))\bigg(\gamma\,|V_1(t)|^2+|\sqrt{G_a [0]}V_2(t)|^2+\zeta_c\big|\na_a Q_a(D)V_2(t)\big|^2\bigg)\, 
\, \d x \d y \\
  & \qquad\qquad   - \gamma\, \Re \int \big( \sqrt{G_a[0]}-\sqrt{G_{-a}[0]}\big) V_1 \, \overline{\sqrt{G_a[0]} V_2} \,
    \d x \d y,\\
    \cK_2&=\gamma \,  \Re \int \pi_r(D)\big(G_a[\zeta_c]-G_a[0]+Q_a(D)\div_a(\zeta_c\na_a Q_a(D))\big) \pi_r(D)V_2\,\,\overline{V_1} \, \d x\d y,\\
     \cK_3&=\gamma \,  \Re  \int  \div_a(\zeta_c\na_aQ_a(D)V_2) \,\,\overline{( Q_a(D)- Q_{-a}(D))V_1
 }\, \d x\d y,\\
    \cK_4&= 2a  \gamma\, \Re\int   \,\zeta_c (\p_x-a) Q_a(D)V_2 \, \,\overline{Q_a(D)V_1} \, \d x\d y ,\\
\cK_5&=\f{1}{2} \int \p_x d_c \big(\gamma |V_1|^2-|\sqrt{G_a[0]}V_2|^2\big) - \p_x (d_c\,\zeta_c) |\na_a Q_a(D)V_2|^2 \,\d x\d y, \\%-2 \Re  \na_a Q_a(D)\p_x(v_c V_2) \cdot \overline{\na_a Q_a(D)V_2 }, \\
\cK_6&= \Re \int  \pi_r(D)\sqrt{G_a[0]}(d_c\p_x Z_c V_1)\overline{\sqrt{G_a[0]} V_2} +  \zeta_c \pi_r(D) \na_a Q_a(D)(d_c\p_x Z_c V_1) \overline{\na_a Q_a(D)V_2 }\, \d x\d y, \\ 
\cK_7&= \Re \int 
\big(W_{\lambda1} \overline{\gamma V_1}+ \sqrt{G_a[0]}W_{\lambda2} \, \overline{\sqrt{G_a[0]} V_2} + \zeta_c \na_a Q_a(D) W_{\lambda2} \cdot \overline{\na_a %\sqrt{1-\tanh^2\mu_
Q_a(D) V_2}\big) \,\d x\d y .
\end{align*}
We now control $K_1-K_7$ term by term. At first,  by  the Cauchy-Schwarz inequality along with Lemma \ref{lem-sym-Ga0} and  the fact that  $\|v_c\|_{L^{\infty}_x}\lesssim \ep^2,$ we can find a constant $C>0$ such that, for $A$ large enough and $\ep$ sufficiently small,
\beqs
\cK_1\leq -2\bigg(a \inf_{x\in\mR} (1-v_c(x))-\gamma \sup_{|\xi|\leq \delta,\, A\ep^2\leq |\eta|\leq 2}\big|\Im \sqrt{\mu_a\tanh \mu_a}(\xi,\eta)\big|\bigg)\cE_r(V(t))\leq -C A \ep^3 \cE_r(V(t)).
\eeqs
Next, we apply  Lemma \ref{lem-remaider-sec} and the Cauchy-Schwarz inequality to find that
\beqs 
\cK_2 \lesssim \ep^3 \,\cE_r(V(t))\,.
\eeqs
Hereafter, we denote $\lesssim$ as $\leq C$ for some constant that is independent of $A$ and $\ep.$
Moreover, let us notice that  $Q_a(D)\in S^{\infty}$ and 
by \eqref{remga0},
\beqs 
|Q_a(\xi,\eta)-Q_{-a}(\xi,\eta)|\lesssim |\Im \mu_a | \lesssim a\lesssim \ep,  \qquad \forall \, (\xi,\eta)\in \Supp \pi_r\, ,
\eeqs
which, together with the fact
$ \|\zeta_c\|_{L_x^{\infty}}\lesssim \ep^2,\, $ yield: 
\beqs 
\cK_3 \lesssim \ep^3\, \cE_r(V(t))\,.
\eeqs
In addition, thanks to the facts $a=\hat{a}\ep$ and $\|\p_x (\zeta_c, d_c)\|_{L_x^{\infty}}\lesssim \ep^3,$ one readily gets that:
\beqs 
\cK_4+\cK_5 \lesssim \ep^3\, \cE_r(V(t))\,.
\eeqs
Furthermore, as $\|\pi_r(D)\sqrt{G_a[0]}\|_{B(L^2)}\lesssim 1,$ and $\|\p_x Z_c \|_{L_x^{\infty}}\lesssim \ep^3,$ we find also
\beqs 
\cK_6 \lesssim \ep^3 \cE_r(V(t))\, .
\eeqs
Finally, it stems from the Cauchy-Schwarz inequality and the fact  that 
$$\|\pi_r(D)\sqrt{G_a[0]}W_{\lambda 2}\|_{L^2}\approx \|\pi_r(D)\na_a W_{\lambda 2}\|_{L^2} \approx\| \pi_r(D)W_{\lambda 2}\|_{H_{*}^{1/2}} $$
that
\beqs 
\cK_7\lesssim \|W_{\lambda}\|_{X} \sqrt{\cE_r(V(t))}\,.
\eeqs
Therefore, upon choosing $A$ large enough, it holds that:
\beqs
\pt \,\cE_r(V(t))\leq -\f{C}{2} A\ep^3 \cE_r(V(t))+C_1\|W_{\lambda}\|_{X} \sqrt{\cE_r(V(t))}\,.
\eeqs
It then follows from the Gr\"onwall inequality that: 
\beqs\label{EnergyIneq}
\sqrt{\cE_r}(V(t)) \leq e^{-CA\ep^3 t/2} \sqrt{\cE_a}(V(0)) 
+C_1 \int_0^t  e^{-CA\ep^3  (t-s)/2}  \|W_{\lambda}(s)\|_{X} \,\d s\,.
\eeqs
Plugging $V(t)=e^{\lambda t} U_r$ into the above estimate and since we had set  $W_{\lambda}=\pi_r(D)  e^{\lambda t}\big(F+\tilde{L}_a (U_h+U_s)\big),$ we find the inequality
\beqs
\sqrt{\cE_r}(U_r) 
\leq e^{-\beta_{\lambda}t}\sqrt{\cE_r}(U_r) +C_1\big(\|F\|_{X}+  \|\pi_r(D){L}_a^1 (U_h+U_s)\|_{X}\big)  \int_0^t e^{-\beta_{\lambda}(t-s)}\, \d s\, ,
\eeqs
where $\beta_{\lambda}=\Re \lambda+CA\ep^3/2> CA\ep^3/4$ 
for any $ \lambda \in \Omega_{\beta,\ep}$ provided that $A$ is large enough. 
As a result, we deduce 
\beqs 
\sqrt{\cE_r}(U_r) \lesssim (A\ep^3)^{-1} \big(\|F\|_{X}+  \|\pi_r(D){L}_a^1(U_h+U_s)\|_{X}\big)\,.
\eeqs
The estimate \eqref{es-ur} then follows by applying estimate \eqref{es-tla}. 
 
 %We will prove in Lemma \ref{} that $\|\pi_r(D)\tilde{L}_a \|_{B(L^2)}\lesssim \sqrt{A}\ep^2,$ which, combined with the above estimate, yields \eqref{es-ur}.
%\section{Properties for the Dirichlet-Neumann operator}
\section{Spectral stability for low transverse frequencies}\label{sec-lowtransverse}
In this section, we focus on the low transverse frequency region $R^L$ defined in \eqref{defregions}.  Let us set  %Denote $X_A^I= ,$ and 
$Z=\mathbb{Q}_a(\ep^2\eta_0)\bI_{A\ep^2}(D_y)X,$ we want to show that for $\ep$ small enough, $\Omega_{\beta,\ep}\subset \rho( L_a; Z)$ and that there is a constant $C>0,$ such that for any $\lambda \in \Omega_{\beta, \ep},$ 
\beq\label{resolbdd-lowtran}
\|(\lambda I-L_a)^{-1}\|_{B(Z)}\leq C.
\eeq

We begin with a lemma for the resolvent estimate in the high longitudinal frequency region $\{|\xi|\geq K\ep\}:$
\begin{lem}\label{lem-HHF}
Let $U \in \tilde{\bI}_{K\ep}(D_x)X,\, F\in X$ which satisfy
%assume there exists a $u\in Z$ such that %$\tilde{\bI}_{A\ep^2}(D_x)(\lambda I-L_a)\tilde{\bI}_{A\ep^2}(D_x) u$
$\big(\lambda I-%\tilde{\bI}_{K\ep}(D_x)
L_a\big) U=F,$ then it holds that, for $K>1$ sufficiently large and $\ep$ small such that $K^4\ep\leq 1,$
\begin{align} \label{es-HHF}
\|U\|_{X}\lesssim K^{-2}\ep^{-3}\,\|F\|_{X}.
\end{align}  
\end{lem}
\begin{proof}
    The proof is close to the one  of Proposition \ref{prop-ITF}, detailed in Subsection \ref{subsec-ITF}. %It consists in
    The general idea is to split $U$ into two parts $U=U^h+U^m,$ where $U^h, U^m$ are localized respectively on the (uniform) high and intermediate longitudinal  frequencies : 
    \beqs 
   U^h=\tilde{\bI}_{\delta}(D_x) U, \qquad U^m=(\tilde{\bI}_{K\ep}{\bI}_{\delta})(D_x)U.
    \eeqs
By the definition, $U^h, U^m$ solves the equations: 
\begin{align*}
\left\{ \begin{array}{l}
\tilde{\bI}_{\delta}(D_x)\, (\lambda I-L_a)\, U^h=\tilde{\bI}_{\delta}(D_x) \,(L_a^1 U^m+ F), \\[4pt]
  (\tilde{\bI}_{K\ep}{\bI}_{\delta})(D_x) \,(\lambda I-L_a)\, U^m=(\tilde{\bI}_{K\ep}{\bI}_{\delta})(D_x) \,(L_a^1 U^h+ F)\, ,
\end{array} \right.
\end{align*}
where $L_a^1$ is defined in \eqref{def-La0-1}.
    As in the treatment of $U_h$ and $U_r$ in subsections \ref{subsub-h}, \ref{subsub-l}, we introduce the  functionals:
    \begin{align*}
      &  \cE_h (U^h)=\int |\sqrt{w_c} U_1^h|^2 + \|\sqrt{G_a[0]}U_2^h\| \,\d x\d y \, ,\\
&\cE_r (U^m)=\colon \f12 \int \bigg(\gamma\,|U_1^m|^2+|\sqrt{G_a [0]}U_2^m|^2+\zeta_c\big|\na_a Q_a(D)U_2^m\big|^2\bigg)\, 
\, \d x \d y\,,
    \end{align*}
and then  find by performing energy estimates that: 
\begin{align*}
&\sqrt{\cE_h (U^h)}\lesssim \ep^{-1} \big(\|F\|_{X}+\|L_a^1 U^m\|_{X}\big), \\
&   \sqrt{\cE_m (U^m)}\lesssim  K^{-2}\ep^{-3} \big(\|F\|_{X}+\|{\bI}_{\delta}(D_x) L_a^1 U^h\|_{X}\big).
\end{align*}
Note that we have used the algebraic property of  $\Re \sqrt{-\mu_a\tanh \mu_a}$ in the region $R_{\xi}^I$ stated in Lemma \ref{lem-sym-Ga0}. The estimate \eqref{es-HHF} can then be derived
 from the fact $\|{\bI}_{\delta}(D_x)L_a^1 \|_{B(X)}+\|L_a^1 {\bI}_{\delta}(D_x) \|_{B(X)}\lesssim \ep^{2}$ which is shown in Lemma \ref{lem-La1} .
\end{proof}

We are now in position to prove the resolvent bound \eqref{resolbdd-lowtran}:
\begin{prop}
 Assume that $\ep$ and $\eta_0$ are small enough.
  For any $\lambda \in \Omega_{\beta, \ep},\,$ any $ U, F \in Z=\colon\mathbb{Q}_a(\ep^2\eta_0)\bI_{A\ep^2}(D_y)X,$ such that $(\lambda I-L_a) U=F,$ it holds that 
    \beq\label{resol-lowfreq}
\|U\|_{X}\lesssim \ep^{-3}\|F\|_{X}.
    \eeq
\end{prop}
\begin{proof}
For the proof, we shall perform different estimates  for  low longitudinal frequencies (${|\xi|\leq K\epsilon}$) and high   longitudinal frequencies (${|\xi|\geq K\epsilon}$). In the high-frequency regime, we apply Lemma \ref{lem-HHF}, while in the low-frequency regime, we use  the KP-II approximation. Due to the excessively large resolvent bound in the low longitudinal  frequency regime, which is of  order $\epsilon^{-3}$, we need to use a smooth cutoff function to localize in frequency.

Let $\chi, \chi_1, \chi_2:\mR \rightarrow [0,1]$ 
be two cut-off functions with the properties: 
\beq\label{def-smoothcutoff}
\begin{aligned}
 & \Supp \chi\subset [-2K\ep, 2K\ep] \, , \quad   \chi_1 \equiv 1 \text{ on } [-K\ep, K\ep] \,, \\
&\Supp \chi_1 \subset [-4K\ep, 4K\ep] \, , \quad   \chi \equiv 1 \text{ on } [-3K\ep, 3K\ep]\, ,  \\
 & \Supp \chi_2\subset [-8K\ep, 8K\ep]\, ,  \quad   \chi_2 \equiv 1 \text{ on } [-4K\ep, 4K\ep]\, ,
\end{aligned}
\eeq
where $K\geq 1$ is large and $\ep$ is small, such that $K^4\ep\leq 1.$
%and denote  %$\chi_{{K}}(\cdot)=\chi(\cdot/K\ep)$ and $\chi_{{1,K}}(\cdot)=\chi_1(\cdot/K\ep), \quad $\tilde{\chi}=1-\chi.$

We  define then
$$ U^H=(1-{\chi})(D_x)\, U,\, \qquad U^L=\chi_1(D_x) \,U,$$
which solve the equations
\begin{align}\label{eqULUH}
\left\{ \begin{array}{l}
\big (\lambda I- %\tilde{\bI}_{K\ep}(D_x)\,
L_a\big)\, U^H=(1-\chi)(D_x) F-[{\chi}(D_x), \, L_a^1] \,U, \\[5pt]
\big(\lambda- L_a\big)\,  %{\bI}_{2K\ep}(D_x) 
{\chi}_{2}(D_x)\, U^L= [{\chi}_1(D_x) , \, L_a^1] \,U+ {\chi}(D_x) F\, .
\end{array} \right.
\end{align}
We refer to \eqref{def-La0-1} for the definition of $L_a^1.$ 

Applying Lemma \ref{lem-HHF} and Lemma \ref{lem-Ga-ll}, we first have that:
\beqs
\begin{aligned}
    \| U^H \|_{X}&\lesssim K^{-2}\ep^{-3}\big(
    \big\|[{\chi}(D_x), \, L_a^1] \, U\big\|_{X} +\|F\|_{X}\big)\\
    & \lesssim K^{-2}\ep^{-3}\big( (K\ep^3+K^4\ep^5)\|U\|_{X} +\|F\|_{X}\big)\\
    & \lesssim (K^{-1}+\ep) \, \|U\|_{X}+K^{-2}\ep^{-3}\|F\|_{X}\,.
\end{aligned}
\eeqs
Note that we have used the assumption $K^4 \ep \leq 1$ for the final inequality. It follows from the fact $(1-\chi)(1-\chi_1)=(1-\chi_1)$ that $\|U\|_{X}\lesssim \|U^L\|_{X}+\|U^H\|_{X}.$ Therefore, upon choosing $K$ large enough and $\ep$ sufficiently small, 
\beq\label{es-UH}
 \| U^H \|_{X}  \lesssim (K^{-1}+\ep) \, \|U^L\|_{X}+K^{-2}\ep^{-3}\|F\|_{X}\,.
\eeq
It now remains to bound  $U^L.$ Since it is localized  in  low frequencies in both variables, %both in $x$ and $y$,
we can diagonalize smoothly  the constant matrix 
\begin{align}\label{def-inverseP}
{L}_{a}^0
 =\left( \begin{array}{cc}
   \p_x-a  &  G_a[0]   \\[5pt]
 -  \gamma& \p_x-a
\end{array}\right) =  P  \left( \begin{array}{cc}
   \lambda_{+}^0  & 0   \\[5pt]
     0   &  \lambda_{-}^0
\end{array}\right) P^{-1}, \quad  P^{-1}= \f12 \left(\begin{array}{cc}
   1  & -\sqrt{- G_a[0] /\gamma}  \\[5pt]
  1& \sqrt{- G_a[0] /\gamma}
\end{array}\right).
\end{align}
where 
$$\lambda_{\pm}^0(\xi,\eta)=i(\xi-ia)\pm \sqrt{-\gamma \mu_a \tanh\mu_a(\xi,\eta)}\,.$$  %$\lambda_0(\xi,\eta)=(\mu_a \tanh\mu_a)(\xi,\eta)$ being the symbol of  $G_a[0].$
%It is direct to check that 
Note that we have,  
\beq\label{equivalence-ll}
\|%\chi_1(D_x)
\bI_{A\ep^2}(D_y) f\|_{X}\approx \| %\chi(D_x)_1
\bI_{A\ep^2}(D_y)P^{-1} f \|_{\cX}\,, \quad (\cX=L^2(\mR^2)\times L^2(\mR^2)) ,
\eeq
and it thus suffices  to get an estimate for $\tilde{U}^L=\colon P^{-1}U^L $ in the space $\cX.$

Let us set  %$\tilde{U}^L=P^{-1}U^L, $ %and 
\beqs
R(x, D)=P^{-1} L_a^1\, P=\big( R_{ij}\big)_{1\leq i, j\leq 2} \, , \qquad \tilde{F}=P^{-1}\big([{\chi}(D_x) , \, L_a^1] \,U+ {\chi}(D_x) F \big).
\eeqs
In particular, $R_{11}$ takes the form
\beq\label{defR}
\begin{aligned}
    R_{11}(x, D)&=-\f12 \bigg( (\p_x-a) (v_c\, \cdot)+\sqrt{-G_a[0]} \,v_c\, (\p_x-a) \sqrt{-G_a[0]} ^{-1} \\
    &\qquad\qquad  +\sqrt{\gamma\,}\big(G_a[\zeta_c]-G_a[0]\big)\sqrt{-G_a[0]} ^{-1} -\sqrt{\gamma}^{-1}\sqrt{-G_a[0]} (d_c \, \p_x Z_c)\bigg).
\end{aligned}
\eeq
Moreover, thanks to   \eqref{equivalence-ll} and  \eqref{es-R}, we have that 
\beq\label{estimate-RxD}
\|\,R\,\|_{B(\cX)}\lesssim (K\ep^3+K^4\ep^5)\lesssim K\ep^3.
\eeq
We take $P^{-1}$ on both sides of  $\eqref{eqULUH}_2$ and translate it into the following system:
%Taking Equation is equivalent to %translates into
\beq \label{eq-tUL}
\begin{aligned}
\left\{ \begin{array}{l}
     \big(\lambda- \lambda_{+}^0(D)+R_{11}(x, D) \big)\,  %{\bI}_{2K\ep}(D_x) 
{\chi}_{2}(D_x)\,\tilde{ U}_1^L= -R_{12}\, \tilde{ U}_2^L + \tilde{F}_1\,,   \\[5pt]
      \big(\lambda- \lambda_{-}^0(D) \big)\,  %{\bI}_{2K\ep}(D_x) 
\,\tilde{ U}_2^L=-\big( R_{21}\, \tilde{ U}_1^L+ R_{22}\, \tilde{ U}_2^L \big)+ \tilde{F}_2\,.
\end{array}\right.
\end{aligned}
\eeq
Thanks to \eqref{lda-ldap} and \eqref{estimate-RxD}, we get by energy estimates that
\beq\label{U2tL}
\|\tilde{ U}_2^L\|_{L^2}\lesssim \ep^{-1}\big( \|\tilde{F}_2\|_{L^2}+ K\ep^3 \|\tilde{U}^L\|_{\cX} \big).
\eeq
To estimate $U_1^L,$ we need to use the KP-II approximation. 
As a preparation, we
introduce the scaling operator 
$S_{\ep}: L^2(\mR^2)\rightarrow L^2(\mR^2)$ as $S_{\ep}f(x,y)=\ep^{-3/2}f(\ep^{-1}x, \ep^{-2}y).$ 
Define then $$\cP=S_{\ep}^{-1}\,\mathbb{P}^{\hat{a}}_{\kp}(\hat{\eta}_0)\,S_{\ep}\,,\qquad  \cQ=S_{\ep}^{-1} \,\mathbb{Q}_{\kp}^{\hat{a}}(\hat{\eta}_0)\,S_{\ep}\, .$$
  where $\mathbb{P}^{\hat{a}}_{\kp}$ is the projection (see \eqref{projection-KP} for precise definition) to the resonant space associated to the operator $L_{\kp}$ and  $\mathbb{Q}^{\hat{a}}_{\kp}=\Id -\mathbb{P}^{\hat{a}}_{\kp}.$  We first claim that it suffices to prove the following two properties: 
%\begin{claim}
\begin{align}\label{prop-puL}
    \|\cP \, \tilde{U}_1^L\|_{L^2}\lesssim (\hat{\eta}_0^2+ \ep+K^3\ep^2+K^{-3}) \,\|\tilde{U}_1^L\|_{L^2}+\|\tilde{U}^L_2\|_{L^2}+\| U^H\|_{X}\, ,
\end{align}
%\end{claim}
%\begin{claim}
    \begin{align}\label{prop-QuL}
    \|\cQ \, \tilde{U}_1^L\|_{L^2}\lesssim %\big(+K^5(\ep^2+\hat{\eta}_0^2+\ep)\big)\, 
    K^{-3}\,\|\tilde{U}_1^L\|_{L^2} + K\|\tilde{U}^L_2\|_{L^2}+%(\hat{\eta}_0+ \ep^2) 
    \ep^{-3}\,\|\tilde{F}_1\|_{L^2}\, .
\end{align}
%\end{claim}
Indeed, once these two estimates are proven, we can  conclude that for $K$ large enough, $\ep, \hat{\eta}_0$ small enough, 
\beqs 
\|\tilde{U}_1^L\|_{L^2}\lesssim K  \|\tilde{U}^L_2\|_{L^2}+\|U^H\|_{X}+ \ep^{-3}\,\|\tilde{F}_1\|_{L^2}.
\eeqs
This, combined with the estimate \eqref{U2tL}, yields, for $K^4\ep\leq 1 $ and $\ep$ small enough, 
\begin{align*}
    \|U^L\|_{X}\lesssim \|\tilde{U}^L\|_{\cX}&\lesssim \|U^H\|_{X}+ \ep^{-3}\,\|\tilde{F}\|_{\cX}, \\
    & \lesssim \|U^H\|_{X}+ \ep^{-3}\,\big(\|F\|_{X}+ \|[{\chi}_1(D_x) , \, L_a^1] \,U\|_{X}\big).
\end{align*}
As is shown in \eqref{La1-commu}, %Lemma \ref{}, %for $K\ep^2\leq 1,$ we have 
\beqs 
\|[{\chi}_1(D_x) , \, L_a^1] \,\bI_2(D_y)\,U\|_{X}\lesssim \ep^3\,\big( \|U^H\|_{X}+ K^{-1} %e^{-\alpha K}
%K^3\ep^2
\|U^L\|_{L^2}\big).
\eeqs
%for some positive constant $\alpha.$ 
Therefore, choosing $K$ large enough, we find that
\beqs 
 \|U^L\|_{X}\lesssim \|U^H\|_{X}+ \ep^{-3}\,\|F\|_{X},
\eeqs
which, together with \eqref{es-UH}, allows us to conclude \eqref{resol-lowfreq} upon choosing $K$ larger and $\ep$ smaller if necessary. The proof is   thus finished  once \eqref{prop-puL}, \eqref{prop-QuL} are established. This is  done in the following two subsections.
\end{proof}

\subsection{Proof of  the estimate \eqref{prop-puL}}
Recall that $Y_{a,\eta}=L_a^2(\mR)\times H^{1/2}_{a,*,\eta}(\mR),$ where $H^{1/2}_{a,*,\eta}(\mR)$ is defined in \eqref{Hhalfeta}. 
The projection operator of $L$ in the weighted space $X_a$ is defined as
\beqs 
\mathbb{P}(\eta_0) f=\sum_{k=1}^2\int_{-\eta_0}^{\eta_0} 
\big\langle\mathcal{F}_yf(\cdot,\eta),\, g_k^{*}(\cdot,\eta)\big\rangle_{{Y_{a,\eta}\times Y_{a,\eta}^{*}}}\,
g_k(\cdot,\eta)\,e^{iy\eta}\,\d \eta\, ,
\eeqs
 where $g_k(x,\eta), g_k^{*}(x, \eta) $ are the basis and dual basis of $L_a(\eta)$ defined in \eqref{defbasis-dual}. 
Let $$Y_{\eta}= e^{a x} Y_{a, \eta}, \qquad 
g_{k,a}(x, \eta)=e^{a x} g_k(x, \eta),\, \qquad g^{*}_{k,a}(x, \eta)=e^{-a x} g_k^{*}(x, \eta), $$ 
the corresponding projection of the transformed operator $L_a =e^{a x}\,  L\, e^{-a x} $ in the unweighted space $X$  has the form
\beqs
\mathbb{P}_a(\eta_0)  f =e^{-a x} \mathbb{P}_a(\eta_0) e^{ax} f=\sum_{k=1}^2\int_{-\eta_0}^{\eta_0} 
\big\langle \mathcal{F}_y f(\cdot,\eta),\, g_{k,a}^{*}(\cdot,\eta)\big\rangle_{Y_{\eta}\times Y_{\eta}^{*}}\,
g_{k,a}(\cdot,\eta)\,e^{iy\eta}\,\d \eta\, .
\eeqs
At this stage,  it is useful to 
introduce the projection operator on the continuous resonant modes of  the  linearized operator $L_{\kp}$ defined in \eqref{linear-kpII} 
\begin{align}\label{projection-KP}
    \mathbb{P}_{\kp}^{\hat{a}}(\hat{\eta}_0)  f =\sum_{k=1}^2\int_{-\hat{\eta}_0}^{\hat{\eta}_0} 
\big\langle \mathcal{F}_y f(\cdot, \hat{\eta}),\, g_{0k,\hat{a}}^{*}(\cdot,\hat{\eta})\big\rangle_{L^2(\mR)\times L^2(\mR)}\,
g_{0k,\hat{a}}(\cdot,\hat{\eta})\,e^{iy\hat{\eta}}\,\d\hat{ \eta}.
\end{align}
where $g_{0k,\hat{a}}(x,\hat{\eta})=e^{\hat{a} x} g_{0k}(x,\hat{\eta}),\, g^{*}_{0k,\hat{a}}(x,\hat{\eta})=e^{-\hat{a} x} g^{*}_{0k}(x,\hat{\eta}) $ 
and $g_{0k,\hat{a}}(x,\hat{\eta}),  \,g^{*}_{0k,\hat{a}}(x,\hat{\eta})$ are generalized eigenmodes for $L_{\kp}(\hat{\eta})$ defined in \eqref{base-kp}.

We are now ready to prove \eqref{prop-puL}. Using the fact that $\mathbb{P}_a(\eta_0) U=0, $ we write
\beqs 
\cP \tilde{U}_1^L=\cP U_1^L-\big[ P^{-1}(D)\, \mathbb{P}_a(\eta_0) P(D)  \tilde{U}^L\big]_1-\big[ P^{-1}(D)\, \mathbb{P}_a(\eta_0) (1-\chi_1)(D)U \big]_1
\eeqs
where we use $\big[\,\,\big]_1$ to denote the first element of a vector. The last term in the above identity can be controlled directly, thanks to the equivalence of norms \eqref{equivalence-ll}, by $\|(1-\chi_1)(D)U \|_{X}\lesssim \|U^H\|_X.$ It thus remains to control the first two terms.
By the definition and the change of variables, we have on the  one hand
\begin{align*}
   & \qquad \big(S_{\ep} \, P^{-1}(D) \,\mathbb{P}_a(\eta_0) P(D) \tilde{U}^L\big)(x, y)\\
   & =\sum_{k=1}^2 \ep^{-{3}/{2}}\int_{-\eta_0}^{\eta_0} e^{i y\,\eta/\ep^2}  \bigg\langle(\cF_y \tilde{U}^L)(\cdot, \eta),\, P^{*}(D, \eta)\, g_{k,a}^{*}(\cdot, \eta)\bigg\rangle P^{-1}(\ep D_x, \eta) \, g_{k,a}(\ep^{-1}x, \eta)\, \d \eta\\
   &=  \sum_{k=1}^2 \int_{-\hat{\eta}_0}^{\hat{\eta}_0} e^{i {y}\,\hat{\eta}}  \bigg\langle\cF_y( S_{\ep}\tilde{U}^L)(\cdot, \hat{\eta}),\,P^{*}(\ep D_{{x}}, \ep^2\hat{\eta})\, g_{k,a}^{*}(\ep^{-1}\cdot, \ep^2\hat{\eta}) \bigg\rangle\, \ep^{-1} P^{-1}(\ep D_{{x}}, \ep^2\hat{\eta}) \, g_{k,a}(\ep^{-1}{x}, \ep^2\hat{\eta})\, \d \hat{\eta}\, ,
\end{align*}
and on the other hand, 
\begin{align*}
    &\qquad \big(\,\mathbb{P}_{\kp}^{\hat{a}}(\eta_0) \, S_{\ep}\tilde{U}_1^L\big)(x, y) \\
   &=   \sum_{k=1}^2 \int_{-\hat{\eta}_0}^{\hat{\eta}_0} e^{i {y}\,\hat{\eta}}  \bigg\langle\cF_y( S_{\ep}\tilde{U}_1^L)(\cdot, \hat{\eta}), \, g_{0k,\hat{a}}^{*}(\cdot, \hat{\eta}) \bigg\rangle\,   \, g_{0k,\hat{a}}({x}, \hat{\eta})\, \d \hat{\eta}\, .
\end{align*}
Note that the inner product in the above identities is taken in $L^2(\mR).$
Moreover, since $g_{0k, \hat{a}}$ and $g_{k,a}(\ep^{-1}\cdot)$ are smooth and exponentially decaying in $x,$ for any $\hat{\eta}\in [-\hat{\eta}_0, \hat{\eta}_0],$
it holds that 
\begin{align*}
& \| \big(1- \chi_2(\ep D_x)  \big) (g_{0k,\hat{a}}(\cdot, \hat{\eta}))\|_{L^2}\lesssim K^{-3}\|%(\cF_x (g_{0k,\hat{a}}))(\xi, \ep^2 \hat{\eta}))
 g_{0k,\hat{a}}(\cdot,  \hat{\eta})\|_{H^3}\lesssim K^{-3},  \\
&\| \big(1- \chi_2(\ep D_x)  \big) (g_{k,a}(\ep^{-1}\cdot, \ep^2 \hat{\eta}))\|_{L^2}\lesssim K^{-3}\|%\ep(\cF_x (g_{k,a}))(\ep\xi, \ep^2 \hat{\eta}))
g_{k,a}(\ep^{-1}\cdot, \ep^2 \hat{\eta}))\|_{H^3}\lesssim K^{-3}. 
\end{align*}
Consequently, since $S_{\ep}$ is a diffeomorphism on $L^2(\mR^2),$ it suffices to show that for any $\hat{\eta}\in[-\hat{\eta}_0, \hat{\eta}_0],$
\begin{align*}
    \bigg\|\chi_2(\ep D_x)\bigg(\ep^{-1}\big[%\chi_2( D_x)
    P^{-1}(D_x, \ep^2 \hat{\eta}) g_{k,a}\big]_1(\ep^{-1}\cdot, \ep^2 \hat{\eta})-g_{0k, \hat{a}}(\cdot, \hat{\eta})\bigg)\bigg\|_{L^2(\mR)}\lesssim (\ep+\hat{\eta}_0^2)\,,  \\
     \bigg\|\chi_2(\ep D_x)\bigg(\big[P^{*}(D_x, \ep^2 \hat{\eta}) g^{*}_{k,a}(\cdot, \hat{\eta})\big]_1(\ep^{-1}\cdot, \ep^2 \hat{\eta})-g_{0k, \hat{a}}^{*}\bigg)\bigg\|_{L^2(\mR)}\lesssim (\ep+\hat{\eta}_0^2)\,. 
\end{align*}
In view of the definition of $P^{-1}$ in \eqref{def-inverseP} as well as the fact $$\gamma=1+\cO(\ep^2), \quad  -\chi_2( D_x)\bI_{A\ep^2}(\eta)\,e^{i y\eta}\sqrt{-G_a[0]}\,e^{-i y\eta}=\ep(\p_x-a) +\cO_{B(L^2(\mR))}(K^3\ep^3)\,,$$
it amounts to verify that
\begin{align*}
    & \big\|(2\ep)^{-1}\big(g_{k,a}^1+%\sqrt{-\lambda^0( D_x, \ep^2 \hat{{\eta}})\,}
    (\p_x-a)\,g_{k,a}^2\big)(\ep^{-1}\cdot, \ep^2 \hat{\eta})-g_{0k, \hat{a}}(\cdot, \hat{\eta})\big\|_{L^2(\mR)}\lesssim (\ep^2+\hat{\eta}_0^2)\,,  \\
   &  \big\|\big(g^{*,1}_{k,a} -%\sqrt{-\lambda^0(D_x, \ep^2 \hat{\eta})\,}
   (\p_x-a)^{-1} \,g^{*,2}_{k,a}\big)(\ep^{-1}\cdot, \ep^2 \hat{\eta})-g_{0k, \hat{a}}^{*}(\cdot, \hat{\eta})\big\|_{L^2(\mR)}\lesssim (\ep^2+\hat{\eta}_0^2)\,.
\end{align*}
%where $\lambda^0(D)=(\mu_a \tanh \mu_a)(D).$ This 
%
Since they rely on direct but somewhat tedious algebraic calculations, we leave the details to Appendix \ref{approx-resonants-sec2}.
\subsection{Proof of the  estimate \eqref{prop-QuL}} 
Let us define  the operators
$$\lambda_{+,\,\ep}^0(D)= \lambda_{+}^0(\ep D_x, \ep^2 D_y), \quad R_{11}^{\ep}(x,D)=R_{11}\big(\ep^{-1}{x},\, \ep D_x,\, \ep^2 D_y\big), \quad \chi_{{\kp}}(D)=\chi_2(\ep D_x) \bI_{A}(D_y). $$
For the last one, note that we have used that  $ \bI_{A}(D_y) =  \bI_{A\ep^2}(\ep^2D_y)$  to simplify.
By applying  the scaling operator $S_{\ep}$ on both sides of 
$\eqref{eq-tUL}_1,$ 
we find that $u=\colon S_{\ep}\tilde{U}_1^L$ solves:
\beqs 
\big(\lambda- \lambda_{+,\ep}^0(D)- R_{11}^{\ep}(x,D)\big)\,\chi_{{\kp}}(D)\, u=f
%\lambda_{-}^0(\ep D_x, \ep^2 D_y)-R_{11}({x}/{\ep}, \ep D_x, \ep^2 D_y)
\eeqs
where $ f=S_{\ep}(\tilde{F}_1-R_{12}\, \tilde{ U}_2^L).$

%The main task in order to prove is the following 
Localizing in  the low frequency regimes, one expects that the linear operator $\lambda- \lambda_{+,\ep}^0(D)- R_{11}^{\ep}(x,D)$ can be well approximated by the linearized operator of the KP-II equation, this is justified in the following lemma.
\begin{lem}\label{lem-KPapprox}
    Let $\lambda=\ep^3 \Lambda, $ and $L_{\kp}^a=e^{\hat{a}x}L_{\kp}e^{-\hat{a}x},$ %the transformed 
    where $L_{\kp}$ is defined in \eqref{linear-kpII}.
    It holds that 
    \beq\label{id--KPapprox}
\big(\lambda- \lambda_{+,\,\ep}^0(D)- R_{11}^{\ep}(x,D)\big)\chi_{{\kp}}(D)=\f{\ep^3}{2} (2 \Lambda- L_{\kp}^{\hat{a}})\chi_{{\kp}}(D)+\cR
    \eeq
    where $\cR $ %\in B(L^2)$ 
    enjoys the property
    \beq\label{es-cR}
    \|\cR\|_{ B(L^2(\mR^2))}\lesssim K^{-3}\,\ep^{3}.
    \eeq
\end{lem}
Once this  lemma is proven, 
the estimate \eqref{prop-QuL} is the consequence of %this lemma and 
the invertibility of $ (\Lambda- L_{\kp}^{\hat{a}})\chi_{{\kp}}(D)$ on the space $\mathbb{Q}_{\kp}^a(\hat{\eta}_0)L^2$ which follows from Proposition 3.2,
\cite{Mizumachi-KP-nonlinear}: 
\begin{prop}[Proposition 3.2, \cite{Mizumachi-KP-nonlinear}]
     There is $\hat{\beta}_0$ such that for any $\Lambda$ such that 
$\Re \Lambda > -\hat{\beta}_0/2,$
\beq\label{revert-Lkp}
\|(2\Lambda- L_{\kp}^{\hat{a}})^{-1}\chi_{{\kp}}(D)\|_{B(\mathbb{Q}_{\kp}^{\hat{a}}(\hat{\eta}_0)L^2)}<+\infty\,.
\eeq
\end{prop}

Indeed, since $\mathbb{Q}_{\kp}^{\hat{a}} L_{\kp}^{\hat{a}}=L_{\kp}^{\hat{a}} \mathbb{Q}_{\kp}^{\hat{a}} $ and $\chi_{\kp}(D) u=u,$ we have 
that 
\beqs 
\ep^3 (\Lambda- L_{\kp}^{\hat{a}})\,\mathbb{Q}_{\kp}^{\hat{a}} u =\mathbb{Q}_{\kp}^{\hat{a}}\,(f-\cR).
\eeqs
It then follows from \eqref{revert-Lkp} and \eqref{es-cR} that 
\beqs 
\|\mathbb{Q}_{\kp}^a u \|_{L^2}\lesssim K^{-3}\, \|u\|_{L^2}+\ep^{-3}\|f\|_{L^2},
\eeqs
which, together with \eqref{estimate-RxD} and  the fact that $S_{\ep}$ is isometric on $L^2$, implies \eqref{prop-QuL}.
%\beqs \|\cQ \, \tilde{U}_1^L\|_{L^2}\lesssim K^{-3} \|\tilde{U}_1^L\|_{L^2}+ K\|\tilde{U}^L_2\|_{L^2}+   \ep^{-3}\,\|\tilde{F}_1\|_{L^2}.\eeqs
It now remains to prove  Lemma \ref{lem-KPapprox}. Let us set
\beq\label{defmuep}
\mu_{\ep}({\xi, \eta})=\ep\sqrt{(\xi+i\hat{a})^2+\ep^2\eta^2}\,.
\eeq
For any $(\xi,\eta)\in\Supp\chi_{\kp}(\xi,\eta)=\big\{(\xi,\eta)\big|\, |\xi|\leq 8K, |\eta|\leq A\leq 2K^2\big\}, $ it holds that
\begin{align*}
    \sqrt{-\mu_{\ep}^2(\xi, \eta)}=-i\ep(\xi+ i\hat{a}) \sqrt{1+\f{\ep^2\eta^2}{(\xi+ i\hat{a})^2}}=-i\ep(\xi+ i\hat{a}) \bigg( 1+\f{\ep^2 \eta^2}{2(\xi+i \hat{a})}+\cO(K^4\ep^4)\bigg).
\end{align*}
Consequently, by using the Taylor expansion $\tanh x= x-\f13 x^3+\cO(x^5),$ we have that
\beq \label{G0-lowlow}
\begin{aligned}
\sqrt{-\mu_{\ep} \tanh \mu_{\ep}}(\xi,\eta)&=\sqrt{-\mu_{\ep}^2\big(1-\f{ \mu_{\ep}^2}{3}+\cO(\mu_{\ep}^4)\big)}\\
    &= \sqrt{-\mu_{\ep}^2\,}\big(1-\f{\mu_{\ep}^2}{6} +\cO(K^4\ep^4)\big) \\
    & = -i\ep \bigg((\xi+ i\hat{a})+\f{\ep^2\eta^2}{2(\xi+ i\hat{a})} - \f{\ep^2(\xi+ i\hat{a})^3}{6}\bigg)+\cO(K^5\ep^5)\,.
\end{aligned}
\eeq
This, together with the expansion $\sqrt{\gamma\,}=\sqrt{1-\ep^2}=1-\f{\ep^2}{2}+\cO(\ep^4),$  leads to 
\begin{align*}
    \lambda_{+,\ep}^0(\xi, \eta)&= i\ep (\xi+ i\hat{a})+
\sqrt{-\gamma\,\mu_{\ep} \tanh \mu_{\ep}}\\
&= \f{\ep^3}{2} \bigg(i (\xi+i \hat{a})+\f{i(\xi+i\hat{a})^3}{3}+\f{\eta^2}{2i(\xi+ i\hat{a})}\bigg)+\cO(K^5\ep^5) \\
& =\f{\ep^3}{2}\Lambda_{\kp,0}^{\hat{a}}(\xi, \eta)+\cO(K^5\ep^5)\,.
\end{align*}
We thus proved that
\beq\label{KP-nonlinear}
 \lambda_{-,\ep}^0(D)\chi_{\kp}(D)=\f{\ep^3}{2}\Lambda_{\kp,0}^{\hat{a}}(D)+\cR_1, \quad \text{ with } \, \|\cR_1\|_{B(L^2)}\lesssim (K \ep)^5 . %+K^{-3}.
\eeq
Note that $\Lambda_{\kp,0}^{\hat{a}}(D)$ is the  operator  obtained by linearizing the  KP-II equation  about $0$ in the moving frame. 
In view of the expression $L_{\kp}^{\hat{a}}= \Lambda_{\kp,0}^{\hat{a}}(D)-3 \,\p_x (\Psi_{\kdv}\cdot),$
we see that in order to prove
\eqref{id--KPapprox}, it suffices to show that 
\beqs
 R_{11}^{{\ep}}\chi_{\kp}(D)=-\f{\ep^3}{2}\big(3 \,\p_x (\Psi_{\kdv}\cdot)\big)+\cR_2\,, \quad \text{ with } \, \|\cR_2\|_{B(L^2)}\lesssim \ep^3\big(\ep+K^5\ep^2
 +K^{-3}\big) .
\eeqs
By introducing
\begin{align*}
   ( v_{\ep}, \zeta_{\ep}, Z_{\ep} )(\cdot)=\ep^{-2}\,(v_c, \zeta_c, Z_c)(\ep^{-1}\cdot)\,, \quad d_{\ep}=1-\ep^2\, v_{\ep}, \, \quad \hat{G}_{\hat{a}}[\ep^2\zeta_{\ep}]=S_{\ep} G_a[\zeta_c] S_{\ep}^{-1}\, .
\end{align*}
and by using  the expression \eqref{defR}, it holds that 
\beq\label{expr-R11}
\begin{aligned}
     & R_{11}^{\ep}(x,D)=R_{11}\big(\ep^{-1}{x},\, \ep D_x,\, \ep^2 D_y\big)\\
     &=-\f{\ep^3}{2} \bigg( (\p_x-\hat{a}) (v_{\ep}\, \cdot)+\sqrt{-\hat{G}_{\hat{a}}[0]} \,v_{\ep}\, (\p_x-\hat{a}) \sqrt{-\hat{G}_{\hat{a}}[0]} ^{-1} \\
    &\qquad\qquad  +\ep^{-3}\sqrt{\gamma\,}\big(\hat{G}_{\hat{a}}[\ep^2\zeta_{\ep}]-\hat{G}_{\hat{a}}[0]\big)\sqrt{-\hat{G}_{\hat{a}}[0]} ^{-1} -\sqrt{\gamma\,}^{-1}\sqrt{-\hat{G}_{\hat{a}}[0]} \big(d_{\ep} \, \p_x Z_{\ep}\big)\bigg).
\end{aligned}
\eeq
Note that by definition, 
$\hat{G}_{\hat{a}}[0]=(\mu_{\ep}\tanh\mu_{\ep})(D),$
with $\mu_{\ep}$ defined in \eqref{defmuep}. %Therefore, we could conclude that 
We first claim that the last term in the above identity can  be considered as a remainder. Indeed, we  write 
\beqs 
\sqrt{-\hat{G}_{\hat{a}}[0]} \big(d_{\ep} \, \p_x Z_{\ep}\big) \chi_{\kp}(D)=\sqrt{-\hat{G}_{\hat{a}}[0]} \bigg(\chi_{\kp}(D)\big(d_{\ep} \, \p_x Z_{\ep}\big) - %\sqrt{-\hat{G}_{\hat{a}}[0]}
\big[\chi_{\kp}(D),\,(d_{\ep} \, \p_x Z_{\ep}\big)\big] \bigg),
\eeqs
we have, on  the one hand, in view of \eqref{G0-lowlow} that ,
\beqs 
\big\|\sqrt{-\hat{G}_{\hat{a}}[0]} \, \chi_{\kp}(D)\big\|_{B(L^2)}\lesssim K\ep\, ,
\eeqs
and on the other hand, by applying \eqref{es-commutator} with $s=1,$ that 
\beqs 
\big\|\sqrt{-\hat{G}_{\hat{a}}[0]} \,\big[\chi_{\kp}(D),\,(d_{\ep} \, \p_x Z_{\ep}\big)\big]\big\|_{B(L^2)}\lesssim C_1 \|\cF_{x\rightarrow \xi}\big(\p_x(d_{\ep} \, \p_x Z_{\ep}\big)\|_{L_{\xi}^1}\lesssim \ep,
\eeqs
where 
\beqs 
C_1=%\sup_{\xi,\, \xi'}
\bigg(\sup_{|\xi|\leq 16K,\, \xi'\in \mR}+ \sup_{|\xi|\geq 16K,\, |\xi'|\leq 8K}\bigg)
\bigg| \sqrt{-\mu_{\ep}\tanh \mu_{\ep}}\, \f{\chi_{\kp}(\xi)-\chi_{\kp}(\xi')}{\xi-\xi'}\bigg| \lesssim K\ep \cdot K^{-1}+ \ep\lesssim \ep .
\eeqs
It thus holds that
\beq\label{term-4}
\big\| \sqrt{-\hat{G}_{\hat{a}}[0]} \big(d_{\ep} \, \p_x Z_{\ep}\big) \chi_{\kp}(D)\big\|_{B(L^2)}\lesssim K \ep\,.
\eeq
Next, thanks again to \eqref{G0-lowlow}, we have 
\begin{align}\label{Ga0-lowfre-0}
    \sqrt{-\hat{G}_{\hat{a}}[0]}\,\chi_{\kp}(D)=-\ep\, \big(\p_x-\hat{a}\big)\big(1+\cO_{B(L^2)}(K^4\ep^2) \big) \chi_{\kp}(D),
\end{align}
which yields
\begin{align}\label{Ga0-lowfreq}
 - \ep  \big(\p_x-\hat{a}\big)\sqrt{-\hat{G}_{\hat{a}}[0]}^{-1} \chi_{\kp}(D)=\Id + \cO_{B(L^2)}(K^4\ep^2)\,.
\end{align}
%Let  $\tilde{\chi}_{\kp}$ be a smooth function supported on $\big\{(\xi, \eta)\big|\, |\xi|\leq 20 K,\, |\eta|\leq A \ep^2 \big\}$ and equals to $1$ 
Moreover, by denoting $m(\xi, \eta)=\ep^{-1} \sqrt{-\mu_{\ep}\tanh \mu_{\ep}}(\xi, \eta)-(i\xi-\hat{a}),$ we write
\begin{align*}
&\quad \cF_{x%\rightarrow\xi
}\bigg(\big(\ep^{-1}\sqrt{-\hat{G}_{\hat{a}}[0]}+(\p_x-\hat{a}) \big)  \, \big(v_{\ep}\, \chi_{\kp}(D) f \big)\bigg)(\xi) \\
&=\bigg(\int_{|\xi|\leq 16K} + \int_{|\xi|\geq 16K}\bigg) \, m(\xi,\,\eta) (\cF_x v_{\ep})(\xi-\xi') \chi_{\kp}(\xi')\cF_x(f)(\xi', y) \,\d \xi'=\colon I_1+I_2 \, .
\end{align*}
It follows from \eqref{G0-lowlow} and Young's inequality that 
\begin{align*}
   \| I_1 \|_{L_{\xi}^2L_y^2}\lesssim K^5\ep^2\, \|\cF_x v_{\ep}\|_{L^1} \|f\|_{L^2}.
\end{align*}
For the second term $I_2 ,$ let us first notice that for any $\xi\in \mR,\, |\eta|\leq 2K^2\ep^2 ,$ 
\begin{align*}
 \ep^{-1} |\sqrt{-\mu_{\ep}\tanh \mu_{\ep}}(\xi, \eta)|\lesssim \ep^{-1}|\mu_{\ep}|\lesssim |\xi+i\hat{a}|(1+\cO(K^4\ep^2))\leq 2|\xi+i\hat{a}|
\end{align*}
as long as $K^4\ep \leq 1$ and $\ep$ small enough. 
Therefore, by using the fact that $\chi_{\kp}(\xi', \eta)$ is supported on $|\xi'|\leq 8K,$ we conclude that
\beqs 
\|I_2\|_{L_{\xi}^2L_y^2}\lesssim  K\|f\|_{L^2} \int_{|\xi|\geq 8K} |\cF_x v_{\ep}'|(\xi)\, \d \xi  \leq K^{-\f{2\ell-1}{2}} \|v_{\ep}'\|_{H^{\ell+1}}\, \|f\|_{L^2} , \quad \forall\, \ell \geq 1.
\eeqs
 We thus have verified that (taking $\ell=3.5$ in the estimate for $I_2$),
\beqs 
\big(\ep^{-1}\sqrt{-\hat{G}_{\hat{a}}[0]}-(\p_x-\hat{a}) \big)  \, v_{\ep}\, \chi_{\kp}(D) =\cO_{B(L^2)}\big(K^5\ep^2+K^{-3}\big)\, ,
\eeqs
which, together with \eqref{Ga0-lowfreq} yields: 
\beq\label{term2}
\sqrt{-\hat{G}_{\hat{a}}[0]} \,v_{\ep}\, (\p_x-\hat{a}) \sqrt{-\hat{G}_{\hat{a}}[0]} ^{-1}= (\p_x-\hat{a}) (v_{\ep} \cdot)+\cO_{B(L^2)}\big(K^5\ep^2+K^{-3}\big).\, 
\eeq

Finally, by using the shape derivative of the Dirichlet-Neumann operator (see Section 4, \cite{Lannes}), we have by the Taylor expansion that:
\beq\label{Ga-talor1er-scaled}
\begin{aligned}
  %\sqrt{\gamma\,}\sqrt{-\hat{G}_{\hat{a}}[0]} ^{-1} 
  %( G_a[\zeta_c]-G_a[0])\vp
  &\ep^{-3}\big(\hat{G}_{\hat{a}}[\ep^2\zeta_{\ep}]-\hat{G}_{\hat{a}}[0]\big) =-(\p_x-\hat{a})\,\zeta_{\ep} \, [\ep(\p_x-\hat{a})]-\ep^3\zeta_{\ep} \,\p_y^2 \\
  &\qquad \qquad \qquad +\int_0^1  \big(-\hat{G}_{\hat{a}}[s\ep^2\zeta_{\ep}]+s\ep^4(\p_x-\hat{a})(\p_x\zeta_{\ep}\cdot)\big)(\zeta_{\ep} Z_{\ep}(s\ep^2\zeta_{\ep},\cdot))\,\d s\, 
\end{aligned}
\eeq
    where $Z_{\ep}(s\ep^2\zeta_{\ep},\cdot)=\f{\ep^{-1}\hat{G}_{\hat{a}}[s\ep^2\zeta_{\ep}]+s\ep^3\p_x\zeta_{{\ep}}(\p_x-\hat{a})}{1+|s\ep^2\p_x\zeta_{\ep}|^2}.$ 
The direct consequence of the above formula is, 
\beqs 
\| \big(\hat{G}_{\hat{a}}[\ep^2\zeta_{\ep}]-\hat{G}_{\hat{a}}[0]\big)\chi_{\kp}(D)\|_{B(L^2, H^1)}\lesssim \ep^2(1+(K\ep)^5)\lesssim \ep^2,
\eeqs
 which in turn, combined with \eqref{Ga0-lowfre-0} and the formula \eqref{Ga-talor1er-scaled},
   yields that  
\beqs 
 \ep^{-3}\big(\hat{G}_{\hat{a}}[\ep^2\zeta_{\ep}]-\hat{G}_{\hat{a}}[0]\big)\chi_{\kp}(D)=-(\p_x-\hat{a})\,\zeta_{\ep} \, [\ep(\p_x-\hat{a})]+\cO_{B(L^2)}(K^4\ep^3)\, .
\eeqs
Consequently, using \eqref{Ga0-lowfreq} and the expansion $\sqrt{\gamma\,}=\sqrt{1-\ep^2}=1+\cO(\ep^2),$ we obtain that:
\beqs 
\ep^{-3}\sqrt{\gamma\,}\big(\hat{G}_{\hat{a}}[\ep^2\zeta_{\ep}]-\hat{G}_{\hat{a}}[0]\big)\sqrt{-\hat{G}_{\hat{a}}[0]} ^{-1}= (\p_x-\hat{a}) (v_{\ep} \cdot)+\cO_{B(L^2)}\big(K^4\ep^3\big).\,
\eeqs
Plugging this estimate and \eqref{term-4} \eqref{term2} into \eqref{expr-R11}, then using the fact $\|v_{\ep}-\Psi_{\kdv}\|_{H^1}=\cO(\ep)$
 (see Theorem 1) and the assumption 
$K^4\ep\leq 1,$ we find \eqref{KP-nonlinear}.

\section{Spectral stability in the unweighted space}
In this section, we prove Theorem \ref{thm-spectral-unweighted}, %which states
or precisely, that the spectra of the %linearized 
operators $L(\eta)$ defined \eqref{def-Leta} in the unweighted space 
$Y_0(\eta)=\colon L^2(\mR)\times H_{0,*,\eta}^{1/2}$
are contained in the imaginary axis.  Thanks to \cite{Pego-Sun}, we just need to study the case $\eta \neq 0$.
Note that for $\eta \neq 0$, we have  $Y_0(\eta)= L^2(\mR) \times H^{1 \over 2}(\mathbb{R}).$

%At first, by the reversibility of the water waves system, 
%we already know that the spectrum of $L(\eta)$ is symmetric with respect to the imaginary axis, so it suffices to show that there is no spectrum 
% of positive real part.
 
%  This will be done by contradiction. Suppose that $\lambda$ is in the  spectrum  of $L(\eta)$ with $\Re \lambda>0,$ we will first show that $\lambda$ must be an eigenvalue. Let $U$ be the corresponding eigenfunction, then we will  show that $U$ belongs indeed to the weighted space $Y_a$ for $0<a<\f{\sqrt{3}}{4}\ep$ and  $\ep$
%sufficiently small. Finally, we will
%prove that for each 
%$\eta\in \mR,$ the operator $L(\eta)$ has no eigenvalues with positive real part in the weighted space $Y_a(\eta).$  The last two facts lead to the contradiction.
%This is then in contradiction  with the fact that $\lambda$ lies in the resolvent sets of $L(\eta)$ in the weighted space $Y_a(\eta)$. 

We first show that the  essential  spectrum of $L(\eta)$ (seen as an unbounded operator on   $Y_0(\eta)$) is on the imaginary axis.
 We write $L(\eta)=M(\eta)+R(\eta),$ where  
\beqs
 M(\eta)
 =\left( \begin{array}{cc}
  \sqrt{d_c} \,\p_x\, \sqrt{d_c}  &  G_{\eta}[0]   \\[5pt]
 -  \gamma& \sqrt{G_{\eta}[0]}^{-1} \sqrt{d_c} \p_x  \sqrt{d_c} \sqrt{G_{\eta}[0]} \end{array}\right), \,\, \qquad  R(\eta)
 =\left( \begin{array}{cc}
 R_{11}  &  R_{12}   \\[5pt]
 R_{21} & R_{22}
\end{array}\right) 
\eeqs
where 
$$R_{11}= \sqrt{d_c} [\sqrt{d_c}, \p_x], \quad R_{12}= G_{\eta}[\zeta_{c}] -G_{\eta}[0],\qquad  R_{21}= d_c \p_x Z_c\,,  $$ 
$$R_{22}=\sqrt{G_{\eta}[0]}^{-1}\bigg[\sqrt{G_{\eta}[0]}, d_c\bigg]\p_x+ \sqrt{G_{\eta}[0]}^{-1}\sqrt{d_c} \bigg[ \sqrt{d_c}, \p_x \bigg]\sqrt{G_{\eta}[0]}\,.$$ 
We shall check  that $\lambda-M(\eta)$ is invertible and then check that $R(\eta)$ is a relatively compact perturbation.

Since $\sqrt{G_{\eta}[0]}: H^{1 \over 2}_{0, *, \eta}\rightarrow L^2(\mathbb{R})$ is an isomorphism,  the invertibility of the operator $ (\lambda-M(\eta))$ on  $Y_0(\eta)$ is equivalent to the invertibility 
of the operator $(\lambda - M_{1}(\eta))$ on $L^2 \times L^2$ where
$$ M_{1}(\eta)
 =\left( \begin{array}{cc}
  \sqrt{d_c} \,\p_x\, \sqrt{d_c}  &  (\gamma G_{\eta}[0])^{1 \over 2}   \\[5pt]
 -   (\gamma G_{\eta}[0])^{1 \over 2} & \sqrt{d_c} \p_x  \sqrt{d_c}  \end{array}\right).$$
 Next, we also observe that
 $$ M_{1}(\eta)= i P^{-1}D(\eta)P, \quad P= \left(\begin{array}{cc} 1 & i \\ 1 & -i \end{array} \right), \quad D(\eta)= \mbox{diag }( S_{-}(\eta), S_{+}(\eta))$$
 where $$ S_{\pm}(\eta)=   \sqrt{d_c} D_x\, \sqrt{d_c} \pm (\gamma G_{\eta}[0])^{1 \over 2}, \quad D_{x}= { 1 \over i} \partial_{x}.$$
 Consequently, it suffices to prove that $\lambda - i S_{\pm} (\eta)$ are invertible on $L^2$ for $\lambda \notin i \mathbb{R}$.
  Since  $S_{\pm}$ are clearly  self-adjoint operators on $L^2$ with domain $H^1$, their spectrum is real 
  and thus the spectrum of  $i S_{\pm} (\eta)$ is included in the imaginary axis.
% obtained that the spectum of $M(\eta)$ 
%Therefore, it holds that $\|(\lambda-M(\eta))^{-1}\|_{B(Y_0(\eta))}\lesssim (\Re \lambda)^{-1}.$
%It is then enough to show that %$\lambda-M(\eta)$
This show that $\lambda - M(\eta)$ is invertible on $Y_0(\eta)$ for $\lambda \notin i\mathbb{R}$.

Next, we shall prove that 
$R(\eta)$ is a relatively compact perturbation of $\lambda-M(\eta)$ on $Y_0(\eta)$  that is to say that  
$(\lambda-M(\eta))^{-1}R(\eta)$ is a compact operator on $Y_0(\eta)$ for $\lambda \notin i \mathbb{R}$
to get that the essential spectrum of $L(\eta)$ is on the imaginary axis.
  %so that they have the same essential spectrum, which implies that if $\lambda$ is a spectra, it must be an eigenvalue. 
First, since the domain of $\lambda-M(\eta)$ is $\langle D_x \rangle^{-1} Y_a(\eta),$ we know that  $(\lambda-M(\eta))^{-1}$ maps $Y_0(\eta)$ to $\langle D_x \rangle^{-1} Y_0(\eta).$ Therefore, we just need to show that
 %\eqref{def-Leta0-1}, \eqref{inverselabda-Leta0} to write down the explicit form of each entries $B_{ij}$  and show that $B_{11}, B_{22}, \sqrt{G_{\eta}[0]}B_{21}, B_{12}\sqrt{G_{\eta}[0]}^{-1}$ 
 $$\langle D_x \rangle^{-1} R_{11}, \quad  \langle D_x \rangle^{-1}\sqrt{G_{\eta}[0]}R_{21}, \quad \langle D_x \rangle^{-1}R_{12}\sqrt{G_{\eta}[0]}^{-1}, \quad \langle D_x \rangle^{-1} \sqrt{G_{\eta}[0]}R_{22}\sqrt{G_{\eta}[0]}^{-1}$$
 are compact operators on $L^2(\mR).$  Indeed, we observe that  each of the above %entry of $L^1(\eta)$ 
 operators contains the background waves $v_c,\p_x Z_c, \zeta_c$ which have exponential decay and thus belong to the Schwarz class.    By using the expansion \eqref{expansion-Gaeta} for $a=0$ as well as the fact that the multiplication by  a Schwarz function is compact from $H^t(\mR)$ to $H^s(\mR)$ for $t>s$ we get the desired property. 

It remains to study the presence of eigenvalues for $\lambda \notin i \mathbb{R}$.
By the reversibility of the water waves system, 
we already know that the spectrum of $L(\eta)$ is symmetric with respect to the imaginary axis, so it suffices to show that there are no eigenvalues
 of positive real part.
 We proceed by contradiction. Let us assume that for some $\eta_{u} \neq 0$, there exists a nontrivial $U_{u} \in H^1 \times H^{3 \over 2}$ and $\sigma_{u}$, $\Re\, \sigma_{u}>0$ such that 
 $$ \sigma_{u} U_{u}= L(\eta_{u}) U_{u}.$$
 Since $\sigma_{u}$ has to be an isolated eigenvalue, we can define the spectral projection associated to $\sigma_{u}$ by
  $$ \pi (\eta_u) = { 1 \over 2i \pi} \int_{\gamma} (\lambda - L(\eta_{u}))^{-1 } d \lambda\,,$$ 
  where $\gamma$ is a small circle enclosing $\sigma_{u}$ and no other eigenvalue.
  Since the dependence of $L(\eta)$ in $\eta$ is analytic, we can use the analytic perturbation theory in \cite{Book-Kato}.
  We can  assume that $\eta_{u}$ is not an exceptional point (otherwise we change it for a nearby non exceptional point), we get that for $\eta \in (\eta_{u}- \delta_{0}, \eta_{u}+\delta_{0})$
   for some small  $\delta_{0}>0$, there exists an analytic  curve $\sigma(\eta)$ of positive real part eigenvalues of $L(\eta)$  which are of constant algebraic multiplicity $m$, 
   such that $\sigma(\eta_{u}) = \sigma_{0}$.
   Moreover, the eigenprojector
   $\pi(\eta)$ is also analytic.  The (possibly trivial) eigennilpotent which is  $N(\eta)=( \lambda - L(\eta)) \pi(\eta)$ is thus also analytic.

As in 
\cite{Pego-Sun}, we can get  that any eigenfunction $U(x)\in Y_0(\eta)$ of $L(\eta)$ associated to a positive real part  eigenvalue $\lambda$ belongs also to the weighted space $Y_a(\eta).$  
 Let us rewrite $L(\eta)U=\lambda U$ as 
 $$(\lambda-L^0(\eta))U=L^1(\eta)U\,,$$
 where  
\beq\label{def-Leta0-1}
 {L}^0(\eta)
 =\left( \begin{array}{cc}
   \p_x  &  G_{\eta}[0]   \\[5pt]
 -  \gamma& \p_x
\end{array}\right), \,\, \qquad  L^1(\eta)
 =\left( \begin{array}{cc}
   -\p_x(v_c\cdot)  &  G_{\eta}[\zeta_{c}] -G_{\eta}[0]  \\[5pt]
   d_c \p_x Z_c & -v_c\p_x
\end{array}\right).
\eeq
The eigenvalues of the symbol of $L^0(\eta)$ are given by
$$\lambda_{\pm,0}^0(\xi,\eta)=i\bigg(\xi\mp \sqrt{\gamma \mu_0\tanh\mu_0}(\xi,\eta)\bigg), \quad \big(\mu_0=\sqrt{\xi^2+\eta^2}\big),$$
which are purely imaginary. Therefore, 
for any $\lambda$ with $\Re \lambda>0,$ we know that the $\lambda-L^0(\eta)$ is invertible on $Y_0(\eta)$ 
and the symbol of its inverse 
is given by 
\beq\label{inverselabda-Leta0}
\begin{aligned}
&\qquad\big(\lambda - L^0(\xi,\eta)\big)^{-1}\\
&= \f{1}{(\lambda-\lambda^0_{+,0})(\lambda-\lambda^0_{-,0})}\left( \begin{array}{cc} \f{(\lambda-\lambda^0_{+,0})+(\lambda-\lambda^0_{-,0})}{2}  & \f{(\lambda-\lambda^0_{-,0})-(\lambda-\lambda^0_{+,0})}{2}  \sqrt{-\mu_0\tanh \mu_0/\gamma}   \\[5pt] -  \gamma&  \f{(\lambda-\lambda^0_{+,0})+(\lambda-\lambda^0_{-,0})}{2} \end{array}\right). 
%&=\colon \left(  \begin{array}{cc}A_{11}, & A_{12}\\A_{21} & A_{22}\end{array}\right). 
\end{aligned}
\eeq
 Moreover, from  similar computations as in the proof of \eqref{id-la0-inverse}, we see that $\lambda-L^0(\eta)$ is invertible also in the weighted space $Y_a(\eta)$ and it maps $Y_a(\eta)$ to $\langle D_x \rangle^{-1} Y_a(\eta).$ 
 By using the exponential localization of $v_c,\p_x Z_c, \zeta_c$ as well as the expansion  \eqref{expansion-Gaeta} for $a=0,$ one can verify that $L^1(\eta)U\in \langle D_x \rangle Y_a(\eta),$ as long as $0<a\leq \f{\sqrt{3}}{4}\ep$ and $\ep$ is small enough. These two facts yield that $U\in Y_a(\eta).$ 
 
 Next by iterating the argument, we get that for $\Re \, \lambda>0$, any element of the generalized kernel has to be in the weighted
 space.
 
 By using this observation we thus get that  for 
 $\eta \in (\eta_{0}- \delta, \eta_{0}+\delta)$ the image of 
   $\pi(\eta)$ is included in the weighted space.  We also deduce that all the operators seen  as valued into the weighted space
   are continuous. 
   
   We then define for every $\delta >0$ sufficiently small the wave packet 
$$V_{\delta}(t,x,y)=\int_{I_\delta } e^{iy\eta\,} e^{L(\eta) t}\,\pi(\eta) U_{u}\,\d \eta\,, \quad I_{\delta}= (\eta_{u} - \delta^{m+1} , \eta_{u}+ \delta^{m+1}).$$
which solves the evolution problem $\pt V=LV.$ 
We recall that we have defined  $m$ as the constant algebraic multiplicity of the eigenvalues $\sigma(\eta)$.
Moreover, we notice that from the above argument $V_{\delta}(t)$ belongs to the  weighted space
 and actually to the image of  $\mathbb{Q}(\ep^2\hat{\eta}_0)X_a$.
 Indeed,  for each $\eta \in I_{\delta}$, even if $\eta \in (- \eta_{0}, \eta_{0})$, we have that  $\mathcal{F}_{y} V_{\delta}(t, \cdot, \eta)$ 
 is orthogonal to the subspace generated by  $(g_{k}^*(\cdot, \eta)_{k=1, \, 2},$ which is the subspace generated the  %which are
 eigenvectors $U^*(\pm \eta)$  of $L(\eta)^*$ associated to different eigenvalues.
  
  We then observe that  we can equivalently write
  $$V_{\delta}(t,x,y)=\int_{I_\delta } e^{iy\eta\,} e^{\sigma(\eta) t}  e^{N(\eta) t}\,\pi(\eta) U_{u}\,\d \eta\,$$
  and that from the Bessel identity, we have
  $$ \|V_{\delta}(0)\|_{X_{a}}^2 =\|e^{ax}V(t)\|_{X}^2= C\int_{I_\delta} \|\pi(\eta) U_{u}\|_{Y_a(\eta)}^2\, \d \eta>0$$
  for $\delta$ sufficiently small since $\pi(\eta_{u}) U_{u}= U_{u}$.
   Moreover, we also obtain that 
  $$   \|V_{\delta}( \delta^{- 1})\|_{X_{a}}^2 = C\int_{I_\delta}
  e^{ 2 \Re \sigma(\eta)\delta^{- 1} }
   \| e^{\delta^{- 1}N(\eta)}\pi(\eta) U_{u}\|_{Y_a(\eta)}^2\, \d \eta.$$
   Note that, for $\delta$ sufficiently small, we have that
   $\Re  \sigma(\eta) \geq \Re \sigma_{u}/2$ and that for $\eta \in I_{\delta}$
   $$  \|N(\eta)  \pi (\eta) U_{u}\|_{Y_{a}(\eta)} \lesssim | \eta - \eta_{u}| \|U_{u}\|_{Y_{a}(\eta)} \lesssim \delta^{m+1}\|U_{u}\|_{Y_{a}(\eta)}$$
   and hence since $N(\eta)^m=0$,  that 
   $$\| e^{ \delta^{-1} N(\eta)} U_{u} - U_{u}\|_{Y_{a}(\eta)} \lesssim \delta^{m+1} { 1 \over \delta^{m}} \lesssim \delta  \| U_{u}\|_{Y_{a}(\eta)}.$$
 This yields that for every $\delta$ sufficiently small, we have
$$  \|V_{\delta}( \delta^{- 1})\|_{X_{a}}^2  \geq C  e^{ \Re \sigma_{u} \over \delta} \|V_{\delta}(0)\|_{X_{a}}^2$$
for some $C>0$ independent of $\delta$.
%
%
%  since $I\cap [-\hat{\eta}_0\ep^2,\hat{\eta}_0\ep^2]=\varnothing.$
%Denote $c^{*}=\min\big\{\Re \lambda(\eta)|\,\eta\in I_0\big\}>0,$
%it then follows from the Parseval identity that for any $t\geq 0$
%\beqs 
%\|V(t)\|_{X_a}=\|e^{ax}V(t)\|_{X}= C\int_{I_1}e^{\lambda(\eta)t} \|e^{a\cdot}U(\cdot, \eta)\|_{Y_0(\eta)}\, \d \eta \gtrsim e^{c^{*}t}\|e^{ax}V(0)\|_{X}=e^{c^{*}t}\|V(0)\|_{X_a}.
%\eeqs
This is in contradiction  with the semigroup estimate \eqref{semigroup} for $\delta$ sufficiently small
(note that $\ep$ is fixed for this argument).
 This  ends the proof.

\appendix 

\section{A technical lemma using pseudodifferential calculus }
We first recall a lemma on  the invertibility  of   pseudodifferential operators. %For any symbol,
For any integer $m,$ we denote by  $S^m$ the symbol class whose elements $b(x,\xi)$ satisfy  the following: for any $j,k\in\mathbb{N},$ there exists $C_{j,k}>0,$ such that
\beqs  
|\p_x^{j}\p_{\xi}^{k}b(x,\xi)|\leq C_{j,k}(1+\xi^2)^{\f{m-k}{2}},\, \quad \forall \,(x,\xi)\in \mR^2.%\big\}
\eeqs
Given a symbol $b=b(x,\xi)\in S^m,$ we use the standard notation $\op(b)$ 
to denote the operator associated to $a$ which is defined by, for any $f$ belongs to the Schwarz class,
\beq 
\op(b)f=\f{1}{2\pi}\int b(x,\xi)\, e^{i\xi(x-x')} f(x')\, \d x' \d \xi\, .
\eeq
\begin{lem}\label{lem-inverse}
 Let $a=a(x,\xi,\eta):\mR^3\rightarrow \mathbb{C}$ be 
    a $S^1$ symbol 
    satisfying: 
 $$ |a|\geq \gamma>0,  \quad \forall \,(x,\xi,\eta)\in \mR^3, \quad %\p_x a(\cdot,\eta) \in S^0, \,\, %\|\p_x a(\cdot,\eta)\|_{S^0}\lesssim \delta, 
 \Sup_{j+k\leq M} \sup_{(x,\xi)\in \mR^2} (1+\xi^2)^{\f{k}{2}} \big|\p_x^{j}\p_{\xi}^{k} (\p_x a)(x,\xi,\eta) \big|\lesssim \delta,
 \quad \forall\, \eta\in\mR,$$
 where $M\geq 2$ is a sufficiently large integer.
 Then there exists $\delta_0>0$ small enough 
 such that: for any $\delta\in (0, \delta_0],$ any $\eta\in \mR,$
 \begin{align}
     \|\op(a(\cdot,\eta))\,u\|_{L^2}\geq \f{\gamma}{2} \|u\|_{L^2}.
 \end{align}
\end{lem}
\begin{proof}
One could refer for example  to Theorem 4.29 in \cite{Zworski-book} for the proof. For the convenience of the readers, we  recall the proof here. Define $b=\f{1}{a},$ then by applying the pseudodifferential calculus for the composition of operators,  
(Theorem 4.12, \cite{Zworski-book}), 
one has
\beqs 
\op\big(b(\cdot,\eta)\big)\,\circ\, \op\big(a(\cdot,\eta)\big) =\Id +r_{\eta} \,, 
\eeqs
  where $r_{\eta} \in S^0$ and $\|r_{\eta}\|_{B(L^2)}\lesssim \delta.$  Therefore, as long as $|\delta|\leq \delta_0$ is small enough, it holds that $\Id+r_{\eta}$ is invertible and 
  \beqs
  \|(\Id+r_{\eta})^{-1}\|_{B(L^2)}\leq 2 \,.
  \eeqs
It then follows that 
\begin{align*}
\|u\|_{L^2}&\leq\|(\Id+r_{\eta})^{-1} \op\big(b(\cdot,\eta)\big)\,\circ\, \op\big(a(\cdot,\eta)\big) \|_{L^2}\\
&\leq 2 \|\op\big(b(\cdot,\eta)\big)\,\circ\, \op\big(a(\cdot,\eta)\big) \|_{L^2}\leq \f{2}{\gamma}  \|\op\big(a(\cdot,\eta)\big)\,u\|_{L^2}\, .
\end{align*}
\end{proof}
In the following lemma, we will prove that the function $g_{\eta}$ defined in \eqref{def-g} belongs to the symbol class $S^0,$ for any $|\eta|\geq 2.$
\begin{lem}\label{lem-g}
    Assume that $\ep$ is small enough, then for any $|\eta|\geq 2,$ $g_{\eta}(x,\xi), \f{1}{g_{\eta}}(x,\xi)$ are both smooth in $(x,\xi)\in\mR^2$ and belong to the symbol class $S^0.$ Moreover, the corresponding seminorms are bounded uniformly for $|\eta|\geq 2.$
\end{lem}
\begin{proof}
    For the convenience of the reader, we rewrite down here the expression of $g_{\eta}:$
    $$g_{\eta}(x,\xi)=\sqrt{\f{w_c(x) (\mu_a\tanh \mu_a)(\xi,\eta)}{\lambda^1_{\zeta_c}\tanh \lambda^1_{\zeta_c}(x,\xi,\eta)}}\,,$$ 
    where $\mu_a(\xi,\eta)=\sqrt{(\xi+ia)^2+\eta^2}, \, \lambda
^1_{\zeta_c}(x, \xi, \eta)=\sqrt{(\xi+ia)^2+\eta^2 (1+(\p_x\zeta_c)^2)}\,.$ 
Let us first remark that since $|\eta|\geq 2, \,0< a\leq \f12$ for $\ep$ small enough, it holds that $\Re \big((\xi+ia)^2+\eta^2\big)=\xi^2+\eta^2-a^2\geq \xi^2+3$ so that $\mu_a$ is smooth in $\xi.$ The similar fact holds for $\lambda_{\zeta_c}^1.$

We will prove that the function inside  the square root in the definition of $g_{\eta}$ is of form $1+f(x,\xi),$ where $f$ is a smooth function in $\xi$ with the property $|f|=\cO(\ep^2),$ this will lead to both the smoothness and the lower and upper boundedness of  $g_{\eta}\,.$ On the one hand, we have by definition  $w_c=\gamma-d_c\p_x Z_c=1+\cO(\ep^2).$ On the other hand, defining   
$\omega(s)=\sqrt{(\xi+ia)^2+\eta^2 (1+s^2(\p_x\zeta_c)^2)},$ so that $\omega(0)=\mu_a,\, \omega(1)=\lambda_{\zeta_c},$
 it holds after some calculations that 
\beqs 
\f{\omega(0)\tanh \omega(0)}{\omega(1)\tanh \omega(1)}=1-\f{\int_0^1 \p_s \big(\omega(s)\tanh \omega(s) \big) \d s}{\omega(1)\tanh \omega(1)}=1+\cO(\ep^4).
\eeqs
Indeed, the above estimate follows from the identity $$\p_s \big(\omega(s)\tanh \omega(s) \big)=\p_s\omega(s)\bigg(\tanh \omega(s)+\f{\omega(s)}{\cosh^2 \omega(s)} \bigg)$$
and the estimate 
\begin{align*}
   \bigg|\f{\p_s \omega(s)}{\omega(1)\tanh \omega(1)}\bigg|=\bigg|\f{\eta^2 s }{\omega(s) \,\omega(1)\tanh \omega(1)}\bigg| (\p_x\zeta_c)^2\lesssim(\p_x\zeta_c)^2\lesssim \ep^4, \quad \forall\, s\in[0,1].
\end{align*}
Note that there exists a constant $C>1,$ such that as long as $\ep$ is small enough, it holds that for any $|\eta|\geq 2,\, s\in[0,1],$  
\beq \label{fact-omega}
\f{1}{C} \sqrt{\xi^2+\eta^2}\leq |\omega(s)|\leq C \sqrt{\xi^2+\eta^2}.
\eeq

It now remains to show that for any $(j, k)\in \mathbb{Z}^2,$ 
there exist finite numbers $C_{j,k}^1, C_{j,k}^2$ which are independent of $\eta$ when $|\eta|\geq 2,$ such that:
\beqs 
|\p_x^{j}\p_{\xi}^{k}g(x,\xi)|\leq C_{j,k}^1(1+\xi^2)^{-\f{k}{2}},\, \quad |\p_x^{j}\p_{\xi}^{k}g^{-1}(x,\xi)|\leq C_{j,k}^2(1+\xi^2)^{-\f{k}{2}},\quad \forall \,(x,\xi)\in \mR^2.
\eeqs
These are  direct consequences of the smoothness of $w_c$ in the variable $x$ and  \eqref{fact-omega}.
\end{proof}

\section{Some algebraic properties for the
spectrum of the linearized operator }% with constant coefficient }
In this section, we give the proof of Lemma \ref{lem-sym-Ga0},
from which we can derive some algebraic properties for the 
spectrum of the linearized operator with constant coefficients, see Proposition \ref{lem-lampm1}.

\begin{proof}[Proof of Lemma \ref{lem-sym-Ga0}]
Thanks to the following formula which is shown in  Lemma 3, \cite{Pego-Sun}:
\begin{align}\label{remga0}
   0\leq  \Re \sqrt{-\mu_a \tanh \mu_a} (\xi,\eta) \leq \f{|\Im \mu_a|}{\cos (\Im \mu_a) } \sqrt{\f{\tanh \Re \mu_a}{\Re \mu_a}},
\end{align}
it is enough  to study the imaginary and real parts of the symbol $\mu_a.$

We first show that $0\leq |\Im \mu_a|\leq a$ and that  it is a decreasing function in terms of $s=\eta^2.$ Indeed, by differentiating the identities: $(\Re \mu)^2-(\Im \mu)^2=\xi^2+\eta^2-a^2, \, \Re \mu \,\Im \mu= a \xi,$ we find that:
\beqs
\p_s (\Im \mu_a)=-\f{\Im \mu_a}{2 |\mu_a|^2} \quad \text{ or equivalently } \p_s (|\Im \mu_a|)=-\f{|\Im \mu_a|}{2 |\mu_a|^2}.
 %|\xi+ia|^2+\eta^2 }%((\Re \mu)^2+(\Im \mu)^2)}
\eeqs
Moreover, when $\eta^2=0, \, |\Im \mu_a|=a. $
As a result, 
\beqs
|\Im \mu_a|(\eta^2) \leq a  \exp\bigg(-\int_0^{\eta^2} \f{1}{a^2+\xi^2+s}\,\d s\bigg)\leq a  \exp\bigg( -\f{\eta^2}{a^2+\xi^2+\eta^2} \bigg).
\eeqs
Therefore, when $|\eta|\geq \kappa |\xi+ia|,$ it holds that, %or $$|\eta|\geq K |\xi+ia|,$$
\beqs 
|\,\Im \mu_a| \leq a  \exp\bigg( -\f{\kappa^2}{1+\kappa^2} \bigg)\leq a \big(1-\kappa^2/e\big).
\eeqs
 Note that we have used the inequality  $\exp(-\f{x}{1+x})\leq 1-\f{x}{e}$ for any $x>0.$ 
 
Let $\kappa=\delta$ or $K \ep$ and notice that $\cos(\Im \mu_a)\geq \cos (a)$ and
$\sqrt{\f{\tanh \Re \mu_a}{\Re \mu_a}}\leq 1,$ we get \eqref{essym-Ga0}
 for the regions $R_{\eta,1}^{I}$ and $R_{\eta,2}^{I}$ from \eqref{remga0}.

Next, since  $\sqrt{\f{\tanh x}{x}}$ is decreasing for $x>0,$ we have for any $\Re \mu_a \geq \kappa,$ $\sqrt{\f{\tanh \Re \mu_a}{\Re \mu_a}}\leq \sqrt{\f{\tanh \kappa}{\Re \kappa}}.$
On the other hand, 
when $\Re \mu_a^2=\xi^2+\eta^2-a^2\geq 0,$ and so that 
$\Re \mu_a \geq  \sqrt{\Re \mu_a^2}\geq
(K-\hat{a})\ep^2\geq (K-1)\ep,$  { if } $|\xi|\geq K\ep$ and $\Re \mu_a \geq \sqrt{\delta^2-a^2}\geq \delta^2/2 $ if $|\xi|\geq \delta.$ We thus proved \eqref{essym-Ga0} for the regions $R_{\xi}^H $ and  $R_{\xi}^I$ by using the fact 
$\sqrt{\f{\tanh \kappa}{\Re \kappa}}\leq 1-C\kappa$ for $\kappa$ small and the estimate $|\Im \mu_a|\leq a.$

Finally, recall that $\lambda_{\pm}^0(\xi,\eta)= i(\xi+ia)\pm \sqrt{1-\ep^2} \sqrt{-\mu_a \tanh \mu_a}(\xi,\eta),$ for any 
  $\lambda \in \Omega_{\beta, \ep},$ it follows from \eqref{remga0}
  that for $0<a\leq \f{\sqrt{3}}{4}\ep$ and $\ep$ small enough,
 \beq \label{Relambdapm}
\Re \lambda_{\pm}^0\leq a\bigg(\f{\sqrt{1-\ep^2}}{\cos a}-1\bigg)\leq a \bigg(\f{1- \ep^2/2}{1-3\ep^2/32}-1\bigg)\leq -a \ep^2/4.
    \eeq
\end{proof}
\section{Study of the transformed Dirichlet-Neumann operator}\label{appendix-DN}
In this section, we study carefully the transformed Dirichlet-Neumann opearator $G_a[\zeta_c]$  defined in \eqref{def-DN-T}. 

The first lemma concerns the solvability of the  problem \eqref{ellipticpb}. Before stating the result, we first introduce a change of variable which is a diffeomorphism between the domain $\Omega$ delimited by the surface defined by the solitary wave and the fixed strip $\cS=\mR^2\times [-1,0].$ As it is known that the definition of the  Dirichlet-Neumann operator does not depend on this diffeomorphism (Lemma 3.1, \cite{Lannes}), we will just choose one that is most convenient to us. 
Let $\vp: \Omega\rightarrow \cS$ be defined by $\vp(x,y,z)=(x, y, \f{z-\zeta_c}{1+\zeta_c}).$
Then
$\Psi_a=\colon \Phi_a \circ \vp^{-1}$ %(x,y,%\tilde{z})=\Phi_a\big(x,y, {z}(1+\zeta_c)+\zeta_c\big),$ then $\Psi_a$ 
solves the following  problem in $\cS:$ 
\beq\label{elliptic-newcor}
-(\Delta_a^g+\p_y^2)\Psi_a=0, \qquad \Psi_a|_{{z}=0}=f, \qquad \p_{{z}} \Psi_a|_{{z}=-1}=0,
\eeq
where $\Delta_a^g=\big(\det (g)\big)^{\f12}\div_a \big(\big(\det (g)\big)^{-\f12} g \na_a\big)$ and $$g(x,z)=\left(\begin{array}{cc}
   1  &  -\f{(z+1)\p_x\zeta_c}{1+\zeta_c}\\[4pt]
 -\f{(z+1)\p_x\zeta_c}{1+\zeta_c}  &\f{1+(z+1)^2(\p_x\zeta_c)^2}{(1+\zeta_c)^2}
\end{array}\right).$$

\begin{lem}\label{lem-elliptic}
 For any $f\in H_*^{1/2}(\mR^2),$ let $E[0]f$ be defined in \eqref{def-e0f} and 
 \beqs 
H_{0,up}^1(\cS)=\big\{u\in H^1(\cS)\big|\, u|_{z=0}=0 \big\}.
 \eeqs
 There is a unique solution $\Psi_a$ to the equation \eqref{elliptic-newcor} in the space  $H_{0,up}^1(\cS)+
    E[0](f).$
 Moreover, it holds that
\beqs
\|\Phi_a(f)-E[0](f) \|_{H^1(\cS)}\lesssim \ep^2\,\|f\|_{H_*^{1/2}(\mR^2)}.
\eeqs

\begin{proof}
On the one hand, when the surface is flat, ie, $\zeta_c=0,$  the Fourier transform together with the Poincaré inequality yield that  $E[0](f)$ is the unique solution to the equation \eqref{elliptic-newcor}. Moreover, it follows from the explicit formulae that:
\beq\label{factflat}
\|(\p_x,\p_z)E[0]f\|_{L^2(\cS)}\lesssim \|f\|_{H_*^{1/2}(\mR^2)}\, .\quad 
\eeq
%where the strip $\cS=\mR^2\times [-1,0].$ For instance, %by noticing that $\sqrt{|\mu_a(\xi,\eta)|}\lesssim a^{-1}(a+\f{|(\xi,\eta)|}{\langle\xi,\eta\rangle^{1/2}})$ 
For instance, by direct computation,
we can  control $\p_z E[0]f$ as:
\begin{align*}
\|\p_z E[0]f\|_{L^2(\cS)}& \lesssim \bigg(\int_{\mR^2}  \f{\mu_a^2}{{\cosh^2 \mu_a}} \int_{-1}^0 \sinh ^2(\mu_a(z+1)) f \, \d z \d x\d y\bigg)^{\f12}  \\ 
&\lesssim \|\sqrt{\mu_a \tanh\mu_a} \,f\|_{L^2{(\mR^2)}}+ \big\|\f{\mu_a}{\cosh \mu_a}f\big\|_{L^2{(\mR^2)}}\lesssim  \|f\|_{H_*^{1/2}(\mR^2)}.
\end{align*}
On the other hand, when $\zeta_c\neq 0,$ we define $\tilde{\Psi}_a={\Psi}_a-E[0]f,$ which solves 
\beqs 
-\big((\p_x-a)^2+\p_z^2+\p_y^2\big)\tilde{\Psi}_a=(\Delta_a^g-(\p_x-a)^2-\p_z^2)\Psi_a, \qquad \Psi_a|_{{z}=0}=0, \qquad \p_{{z}} \Psi_a|_{{z}=-1}=0.
\eeqs
It follows from the standard energy estimate that (we denote $L^2(\cS)$ by $L^2$ for short) 
\beq\label{EE-ep}
\|\na \tilde{\Psi}_a \|_{L^2}^2-a^2 \|\tilde{\Psi}_a \|_{L^2}^2 \lesssim \Lambda( \|\zeta_c\|_{W^{2,\infty}}) \|\na \tilde{\Psi}_a \|_{L^2} 
\big(\|\na \tilde{\Psi}_a \|_{L^2} +\|\na_{x,z} E[0]f \|_{L^2} \big)\, ,
\eeq
where $\Lambda$ is a polynomial which vanishes at the origin. On the one hand, as $\tilde{\Psi}_a $ vanishes on $\{z=0\},$ we have by 
%and use 
the Poincaré inequality $\|\tilde{\Psi}_a\|_{L^2(\cS)}\leq \|\p_z \tilde{\Psi}_a\|_{L^2(\cS)},$ thus the left hand side of \eqref{EE-ep} is larger than $\f14 \|f\|_{H^1}$ as long as $a<\f12.$ We then use the estimate \eqref{factflat}, the fact $\|\zeta_c\|_{L^{\infty}}\lesssim \ep^2,$ and the Young's inequality to find that 
$$\|{\Psi}_a-E[0]f\|_{H^1(\Omega)}\lesssim\|\tilde{\Psi}_a\|_{H^1(\cS)}\lesssim \ep^2 \,\|f\|_{H_*^{1/2}(\mR^2)}.$$
\end{proof}
\end{lem}
In the next lemma, we show that the Dirichlet-Neumann operator $G_a[\zeta_c]$ can be well approximated by $\op(\tilde{\lambda}_c^1),$ the operator associated to the  symbol 
$$\tilde{\lambda}_{\zeta_c}^1=({\lambda}_{\zeta_c}^1\tanh {\lambda}_{\zeta_c}^1) (x, \xi, \eta),\qquad  \qquad \big(
 {\lambda}_{\zeta_c}^1=\sqrt{(\xi+ia)^2+\eta^2 (1+(\p_x\zeta_c)^2)}\big)\, $$
in the sense that: 
\beqs 
\|\tilde{\mathbb{I}}_2(D_y) \big(G_{a,\eta}[\zeta_c]-\op\big(%{\lambda}_{\zeta_c}^1\tanh \lambda_{\zeta_c}^1
\tilde{\lambda}_{\zeta_c}^1\big)\big)\|_{B(L^2(\mR^2))}\lesssim \ep^2.
\eeqs
We recall that the definitions of the localization operators in frequency are given in the beginning of Section \ref{paragraph3}.
\begin{lem}\label{lem-error}
  For any $\eta,$ such that $|\eta|\geq 2,$ there is a constant $C>0$ independent of $\ep,$ such that
\begin{align}
    \|G_{a,\eta}[\zeta_c]-\op\big(\tilde{{\lambda}}_{\zeta_c}^1(\cdot,\cdot, \eta)\big)\|_{B(L^2(\mR))}\leq C \ep^2.
    \label{err-high}
\end{align}
\end{lem}
\begin{proof}[Proof of Lemma \ref{lem-error} ] 
The proof relies on the construction of a parametrix for the elliptic problem in $\tilde{\cS}=\mR\times[-1,0]:$
\beqs 
(-\Delta_a^g+\eta^2)\Psi_{a,\eta}=0, \qquad \Psi_{a,\eta}|_{{z}=0}=f, \qquad \p_{{z}} \Psi_{a,\eta}|_{{z}=-1}=0.
\eeqs
Define the functions $$%\alpha=
\alpha(x,z)=\f{1+(z+1)^2(\p_x\zeta_c)^2}{(1+\zeta_c)^2}, \,  \qquad \kappa(x,z)%=\kappa
=\f{\p_x\zeta_c (z+1)}{(1+\zeta_c)},$$ and %define 
the operators 
\begin{align*}
\nu(z)=\nu(z)(x, D_x, \eta)=\op\big(\sqrt{(\xi+ia)^2+\eta^2 (1+(z+1)^2(\p_x\zeta_c)^2)}\big)\,, \\
\sigma_{\pm}(z)=\sigma_{\pm}(z) (x, D_x, \eta)=\alpha^{-1}\big(\kappa \p_x \pm (1+\zeta_c)^{-1} \nu(z)(x, D_x, \eta)\big)\, .
\end{align*}
Note that by the definition, 
\begin{align*}
    \nu(0)=\op({\lambda}_{\zeta_c}^1), \quad \nu(-1)=\op(\sqrt{(\xi+ia)^2+\eta^2 })=:\op(\mu_a)\,.
\end{align*}
By using  pseudo-differential calculus, we find that 
\beqs 
\Delta_{g,\eta}^{{app}}=\colon -\alpha^{-1}(\p_z-\sigma_{-}(z))(\p_z-\sigma_{+}(z))
\eeqs
approximates the elliptic operator $-\Delta_g+\eta^2,$ in the sense that for any $s\geq 0,$
\beq\label{property1}
\begin{aligned}
  \|\big((-\Delta_g+\eta^2) - \Delta_{g,\eta}^{{app}}\big)\|_{H^s(\tilde{\cS})}&\lesssim 
  \|\p_x \zeta_c\|_{W^{s+2,\infty}}\big(\|f\|_{H^{s+1}(\tilde{\cS})}+\eta \|f\|_{H^s(\tilde{\cS})}\big)\\
 &\lesssim  \ep^3 \big(\|f\|_{H^{s+1}(\tilde{\cS})}+\eta \|f\|_{H^s(\tilde{\cS})}\big).
\end{aligned}
\eeq
  Consider the elliptic problem 
\beqs 
\Delta_{g,\eta}^{{app}} \Phi =0 \,\, \text{ in } \tilde{\cS} , \qquad \Phi|_{z=0}=f, \qquad \p_z\Phi|_{z=-1}=0.
\eeqs
The solution map $$   \begin{aligned}
\Phi_{a,\eta}^{app}:  \quad & L^2(\mR)\rightarrow H^{\f12}(\tilde{\cS})\\
&\qquad 
  f \mapsto \Phi_{a,\eta}^{app}(f) \end{aligned} $$ 
is given by 
\beq\label{def-phiapp}
\Phi_{a,\eta}^{app}= \bigg(\Gamma_{+}(0,z)+  \int_z^0 \Gamma_{+}(z',z)\Gamma_{-}(z',-1)\, \sigma_{+}(-1) \,A^{-1}\,\Gamma_{+}(0,-1)\bigg)\, \d z' , %\bigg(\exp\big(-\int_z^0 \sigma_{+}(z') \,\d z'\big) + \int_z^0 \exp\big(-\int_{z}^{z'}\,\sigma_{+}(z'')\, \d z''+\int_{-1}^{z'}\sigma_{-}(z'') \d z''\big) \,\d z' A(x, D, \eta)\exp\big(-\int_{-1}^0 \sigma_{+}(z') \,\d z'\big)\bigg) 
\eeq
where 
\begin{align*}
 &\Gamma_{\pm}(z', z)= \Gamma_{\pm}(z', z)(x, D_x, \eta)= \exp\bigg( \mp \int_{z}^{z'} \sigma_{\pm} (z'')(x, D_x, \eta)\, \d z''\bigg),\\
& A=A(x, D_x, \eta)= %\bigg( 
 \Id-\int_{-1}^0  \Gamma_{+}(z',-1)\Gamma_{-}(z',-1)  \,\sigma_{+}(-1) \,\d z'. %\bigg)^{-1}.
\end{align*}
Let us notice that the symbols of $\Gamma_{\pm}(z', z)$ satisfy 
\beq\label{es-Gammapm}
\big|\Gamma_{\pm}(z', z)(x, \xi, \eta)\big|\lesssim \exp\bigg( - \int_{z}^{z'} \Re \sigma_{\pm} (z'')(x, \xi, \eta)\, \d z''\bigg)\lesssim \exp\bigg( - \tau(z'-z) \Re \mu_a(x, \xi, \eta) %\f{(z'-z)(1+\zeta_c)}{1+(\p_x\zeta_c)^2}\Re \mu_a 
\bigg),
\eeq
where we denote $\tau= \tau (x) %\xi, \eta)
=\f{1+\zeta_c}{1+(\p_x\zeta_c)^2}.$
Consequently, we get that
\begin{align*}
&\bigg|\int_{-1}^0  \Gamma_{+}(z',-1)\Gamma_{-}(z',-1) \, \sigma_{+}(-1)  (x, \xi,\eta)\,\d z'\bigg|\\
&\leq \int_{-1}^0 \exp\big(-2\tau(z'+1)\Re \mu_a%\f{2(z'+1)(1+\zeta_c)}{1+(\p_x\zeta_c)^2}
\big)\, 
    (1+\zeta_c) |\mu_a| \,\d z'\\
    &\leq \f12 \big(1-%\exp\big(%{-\f{2(1+\zeta_c)}{1+(\p_x\zeta_c)^2}\Re \mu_a}
    e^{-2\tau \Re \mu_a}\big) %\f{}%{2}
    \big(1+(\p_x\zeta_c)^2\big) \f{|\mu_a|}{\Re \mu_a}\leq \f23
\end{align*}
for any $\eta$ such that $|\eta|\geq 2,$ provided that $\ep$ is small enough. The last inequality comes  from the fact that $\Re \mu_a \geq \sqrt{\xi^2+\eta^2-a^2}, \,\, |\Im \mu_a|\leq a=\hat{a}\ep.$ Therefore, thanks to Lemma \ref{lem-inverse}, the 
operator $A$ is invertible in $B(L^2(\mR))$ and thus $A^{-1}$ is well-defined. 
%the inverse in the definition $A$ is well defined 
Moreover, it follows from again \eqref{es-Gammapm} that: 
\begin{align*}
  \big|  \int_z^0 \Gamma_{+}(z',z)\Gamma_{-}(z',-1)\, \d z'\big|\lesssim e^{-(1-z)\tau} \int_z^0 e^{-2z'\tau} \d \tau\leq -\f{e^{-(1-z)\tau}-e^{-(1+z)\tau}}{2\tau}.
\end{align*}
%One solution matrix associated to $\Delta_{g,\eta}^{{app}}$ could be 
%\beqs \Phi_{a,\eta}^{app}=\exp\big(-\int_z^0 \sigma_{+}(x,z', D_x, \eta)\,\d z'\,\big).\eeqs
%Note that since $\Re \sigma_{+}(x,z,\xi,\eta)\gtrsim \sqrt{\xi^2+\eta^2/2} $
Therefore, for any  $\ep$ small enough, $\Phi_{a,\eta}^{app}$ enjoys some regularizing 
effects
\beq \label{property2}
\|\Phi_{a,\eta}^{app} f\|_{H^{s+\f12}(\tilde{\cS})}\lesssim \|f\|_{{H^{s}({\mR})}}, \qquad \forall \,|\eta|\geq 2, \quad s\in \mR\, .
\eeq
By the definition \eqref{def-phiapp}, it holds that
\beqs 
\p_z \Phi_{a,\eta}^{app} (f)|_{z=0}=\sigma_{+}(0)f-\Gamma_{-}(0,-1) \sigma_{+}(-1)  A^{-1} \Gamma_{+}(0,-1)f.
\eeqs
In view of the classical commutator and composition estimates for  pseudo-differential operators, we have that
\begin{align*}
A^{-1}= \op\big(\f{2}{1+e^{-2\tau\mu_a}}\big)+\cR\, ,\qquad  \Gamma_{-}(0,-1)\Gamma_{+}(0,-1)=\op(e^{-2\tau\mu_a})+\cR
\end{align*}
and thus
\beqs 
  \Gamma_{-}(0,-1) \sigma_{+}(-1)  A^{-1} \Gamma_{+}(0,-1)=\op\big( \f{2  \mu_a}{e^{2\mu_a}+1}\big)+\cR\,, %=\op\big( \f{2 \tau \mu_a}{e^{2\tau\mu_a}+1}\big)+\cR, %= \Gamma_{-}(0,-1)\Gamma_{+}(0,-1) \sigma_{+}(-1)  A^{-1} 
\eeqs
where we denote by  $\cR$ a generic bounded operator on $L^2(\mR)$ with  operator norm of order $\cO(\ep^2),$ uniformly in $\eta$ such that $|\eta|\geq 2.$  Note that we have used that the commutators of the above involved operators contain always some derivatives of $\zeta_c$ and that $\|(\ep^{-2}\zeta_c, \ep^{-3}\p_x\zeta_c))\|_{W^{1,\infty}}\lesssim 1.$
%By using the fact that $\tau(x)=1+\cO(\ep^2),$and 
 Consequently, by the identity $\lambda_{\zeta_c}^1-\mu_a=\f{(\p_x\zeta_c)^2\eta^2}{\lambda_{\zeta_c}^1+\mu_a}$
% by using again the fact that 
and the fact $\Re \mu_a \geq \sqrt{\xi^2+\eta^2-a^2},$ we conclude that
\begin{align*}
\p_z \Phi_{a,\eta}^{app} (f)|_{z=0}&=\tau(x)\bigg(\p_x\zeta_c\p_x+\op\big(\lambda_{\zeta_c}^1\big(1-\f{2}{e^{2\mu_a}%\lambda_{\zeta_c}^1}
+1}\big)\big) \bigg) f+\cR f \\
  %\op\big( \f{2 \tau \mu_a}{e^{2\tau\mu_a}+1}\big)=  \op\big( \f{2  \lambda_{\zeta_c}^1}{e^{2 \lambda_{\zeta_c}^1}+1}\big)+\cR, 
&=\tau(x)\bigg(\p_x\zeta_c\p_x+\op\big(\lambda_{\zeta_c}^1\tanh \lambda_{\zeta_c}^1\big) \bigg)f+\cR f,
\end{align*}
or equivalently, 
\beqs
\op\big(\lambda_{\zeta_c}^1\tanh \lambda_{\zeta_c}^1\big) f=\bigg(-\p_x \zeta_c \p_x+\f{1+(\p_x\zeta_c)^2}{1+\zeta_c}\p_z\bigg)\Phi_{a,\eta}^{app}(f)\big|_{z=0}+\cR f. 
\eeqs
%It is important to notice that on one hand, 
%On the other hand,
%It follows from 
Recall that $ G_{a,\eta}[\zeta_c] $ is defined as
\begin{align*}
   G_{a,\eta}[\zeta_c] f =\bigg(-\p_x \zeta_c \p_x+\f{1+(\p_x\zeta_c)^2}{1+\zeta_c}\p_z\bigg) \Psi_{a,\eta}(f)|_{z=0} \, ,
\end{align*}
where $\Psi_{a,\eta}(f)$ solves the problem \eqref{elliptic-newcor}. Therefore,
to show \eqref{err-high}, it  suffices to prove that for any %$\eta$ such that 
$|\eta|\geq 2,$ there is a constant $C,$ such that
\beqs 
\|\na_{x,z} (\Psi_{a,\eta}-\Phi_{a,\eta}^{app})(f) (x,0)\|_{L^2(\mR)}%_{H^{\f12}(\tilde{\cS})}
\leq C\ep^2 \|f\|_{L^2(\mR)},
\eeqs
which is based on for $-\Delta_{g}+\eta^2$ which are similar to the one  performed in the previous lemma and the estimates \eqref{property1}, \eqref{property2}.
%By applying  the result in \cite{Alazard-Metivier}, we have that: 
%\beqs \lambda_{\zeta_c}(x,\xi,\eta)=\lambda_{\zeta_c}^1 (x,\xi,\eta)+R(x,\xi,\eta)\eeqs
%where $R\in S^0$ and for any $|\eta|\geq 2,$ {\color{red} to verify
%\beq\label{es-firstGN}
%\big|\p_x^m \p_{\xi}^{\ell}R\big|(\eta)\lesssim_{m,\ell} \sum_{k\leq m} \sup_{x} |\p_x^k \zeta_c|\lesssim \ep^2, \quad \forall (m,\ell) \in \mathbb{Z}^2 .\eeq}
%Moreover, for $|\eta|\geq 2$ and  $\ep$ small enough 
%$$|(\lambda_{\zeta_c}^1)^2|(x,\xi,\eta)\geq \Re (\lambda_{\zeta_c}^1)^2(x,\xi,\eta)\geq \eta^2(1+(\p_x\zeta_c)^2)-a^2\geq \eta^2/2\geq 2. $$ Consequently, \beqs |\sqrt{\lambda_{{\zeta_c}}}+\sqrt{\lambda_{\zeta_c}^1}|\geq 1.\eeqsThis, combined with \eqref{es-firstGN} leads to:\beqs \big|\p_x^m \p_{\xi}^{\ell} \big(\sqrt{\lambda_{{\zeta_c}}}-\sqrt{\lambda_{\zeta_c}^1}\big)\big|\lesssim_{m,\ell, M} \lesssim \ep^2,  \quad \forall \, (m,\ell) \in \mathbb{Z}^2.\eeqsThe estimate \eqref{err-high} then follows. 
\end{proof}

The next lemma shows that when localizing on the bounded  transverse frequencies, the Dirichlet-Neumann operator 
can be approximated by  %could be well approximated by 
$G_a[0]:$
\begin{lem}\label{lem-Ga-Ga0-g}
 It holds that:
    \beqs 
\|\bI_2(D_y) (G_a[\zeta_c]-G_a[0])\|_{B(H_{*}^{\f12}(\mR^2), L^2(\mR^2))}\lesssim \ep^2.
    \eeqs
\end{lem} 
In order to prove this lemma, we need to use the following result which states that when focusing on the bounded  transverse frequencies, the Dirichlet-Neumann operator $G_a[\zeta_c]$ can  be written as a main order $\mu_a(D)$ plus a remainder which sends a $H^{-1}$ function to $L^2.$
 %the main order whose symbol is the principle 
%concerning the expansions of the Dirichlet-Neumann operator 

\begin{lem}\label{lem-Gaeta}
    Denote $G_{a,\eta}[\cdot]=\colon e^{-iy \eta}\,G_a[\cdot]\, e^{iy \eta}$ and $\mu_{a,\eta}(D_x)
    =\sqrt{(D_x+ia)^2+\eta^2}.$ Let $\chi: \mR\rightarrow \mR$ be a smooth function  with compact support and equals to one on $[-10,10]$ and denote $\tilde{\chi}=1-\chi.$ 
    Then we have that: 
\begin{align*}
        G_{a,\eta}[\cdot]=\mu_a(D_x,\eta)\tilde{\chi}(D_x
        )+R_{\eta}
\end{align*}
    where $R_{\eta}\in B(H^{-1},L^2)$ uniformly for $\eta$ such that $|\eta|\leq 2.$
\end{lem}
Note that in the above Lemma we just get an approximation in terms of order of the operator, the remainder is not
necessarily small.
\begin{proof}
    Since  $\zeta_c\in C^{\infty}(\mR),$ %we know from \cite{Alazard-Metivier} that   
    the operator $G_{a,\eta}[\zeta_c]$ can be expanded under the form
    \beqs 
 G_{a,\eta}[\zeta_c]=\op (\tilde{\chi}(\xi )\lambda_{\zeta_c}^1(x,\xi,\eta))+\tilde{R}_{\eta}\,, 
\eeqs
where $\lambda%_{\zeta_c}
^1_{\zeta_c}(x, \xi, \eta)=\sqrt{(\xi+ia)^2+\eta^2 (1+(\p_x\zeta_c)^2)}$ and $\tilde{R}_{\eta} $ is a bounded operator on $L^2_x$ for any $\eta \in \mR.$ 
This follows again from a study of the modified equation which is very similar to the one that is done when 
$a=0$ (see \cite{Lannes} for example).
For any $|\eta|\leq 2,$  straightforward computations show that:
\beqs 
|\lambda_{\zeta_c}^1-\mu_{a,\eta}|(\xi)=\bigg|\f{\eta^2(\p_x\zeta_c)^2}{\lambda_{\zeta_c}^1+\mu_a}\bigg|\lesssim \f{|\eta \p_x\zeta_c|}{\sqrt{1+\xi^2}},
\eeqs
which enables us to write
\beqs 
 G_{a,\eta}[\zeta_c]=%$\op(\lambda_{\zeta_c}^1)
 \mu_a(D_x, \eta)\tilde{\chi}(D_x )+{R}_{\eta}\,, 
\eeqs
where ${R}_{\eta}-\tilde{R}_{\eta}=\op\big(\tilde{\chi}
        (\lambda_{\zeta_c}^1-\mu_a)\big)\in B(H_x^{-1},L_x^2).$  We thus only need  to show that $\tilde{R}_{\eta}\in  B(H_x^{-1},L_x^2)$ for any $\eta$ such that $|\eta|\leq2 .$ 
Let us write 
$$\tilde{R}_{\eta}=\chi(D_x) \circ \tilde{R}_{\eta}+\tilde{\chi}(D_x)\circ \tilde{R}_{\eta} \circ \chi(D_x)+ \tilde{\chi}(D_x)\circ \tilde{R}_{\eta} \circ \tilde{\chi}(D_x).$$ It is clear that the first two terms in the right hand side belong to  $B(H_x^{-1},L_x^2).$ For the last term, we 
follow  the computation
in \cite{Alazard-Metivier} (see Section 4.3-4.4) to get that
%we need to know the main term of $\tilde{R},$ 
\beqs 
\tilde{\chi}(D_x)\circ \tilde{R}_{\eta} =\op(\tilde{\lambda}_{\zeta_c}^0)+ \tilde{R}_{\eta,-1}
\eeqs
with $\tilde{\lambda}_{\zeta_c}^0=\tilde{\chi}^2(\xi)\f{b\p_x^2\zeta_c A_1-i\p_{\xi}a_1\cdot \p_x A_1}{2b\lambda_{\zeta_c}^1},$ 
where 
\beqs 
b=\f{1}{1+(\p_x\zeta_c)^2}, \quad A_1=b\big(i\xi\p_x\zeta_c+\lambda_{\zeta_c}^1\big), \quad a_1=b(i\xi\p_x\zeta_c-\lambda_{\zeta_c}^1).
\eeqs
%It follows from the direct computation that:
We compute 
\begin{align*}
&\qquad\qquad \qquad
\big|(\lambda_{\zeta_c}^1)^{-1} \big|\approx \langle \xi \rangle^{-1},\\[3pt]
\big|b\p_x^2\zeta_c A_1-i\p_{\xi}a_1\cdot \p_x A_1\big|&= \bigg|\f{-(i\p_{\xi} a_1)b\eta^2\p_x(\p_x\zeta_c)^2+2b^2\p_x^2\zeta_c\big(2b  (\p_x\zeta_c)^2-1\big)\big(\xi^2-(\lambda_{\zeta_c}^1)^2\big)} 
{2\lambda_{\zeta_c}^1}\bigg|
\leq C 
\end{align*}
on the support of $\tilde{\chi}(\xi),$ uniformly for $|\eta|\leq 2.$ 
As a result, we see that $\op\big(\tilde{\lambda}_{\zeta_c}^0\big)\in B(H_x^{-1},L_x^2).$
The proof is now complete.
\end{proof}
\begin{rmk}
    We see in the proof that %it holds 
   $\tilde{R}_{\eta}$ actually  belongs to  $ B(H_x^{-2},L_x^2).$
\end{rmk}
\begin{proof}[Proof of Lemma \ref{lem-Ga-Ga0-g}]
%Denote $G_{a,\eta}[\cdot]=\colon e^{-iy \eta}G_a[\cdot] e^{iy \eta},$ 
It suffices to show that for any $\eta \in [-2, 2],$ %it holds that: 
   \beqs 
   \big\|G_{a,\eta}[\zeta_c]-G_{a,\eta}[0]\big\|_{B(H_{*,\eta}^{\f12}, L_x^2)}\lesssim \ep^2.   
   \eeqs
By the Taylor expansion of $G_{a,\eta}[\zeta_c]$ in terms of $\zeta_c,$ we have that: 
\beq\label{expansion-Gaeta}
\begin{aligned}
 & \qquad ( G_{a,\eta}[\zeta_c]-G_{a,\eta}[0])\vp\\
 & =\int_0^1  \big(-G_{a,\eta}[s\zeta_c]+(\p_x-a)(s\p_x\zeta_c\cdot)\big)(\zeta_c Z_{a,\eta}(s\zeta_c,\vp))-(\p_x-a)(\zeta_c(\p_x-a)\vp)+\eta^2\zeta_c\vp \, \d s, \quad 
\end{aligned}
\eeq
    where $Z_a(s\zeta_c,\vp)=\f{G_{a,\eta}[s\zeta_c]\varphi+s\p_x\zeta_{c}(\p_x-a)\vp}{1+|s\p_x\zeta_{c}|^2}.$ Since for all $s\in[0,1],$ $G_{a,\eta}[s \zeta_c]\in B(H^1_x, L_x^2),$ we get directly from $\|\zeta_c\|_{W_x^{2,\infty}}\lesssim \ep^2$
  that: 
  \beqs 
\|\chi(D_x) ( G_{a,\eta}[\zeta_c]-G_{a,\eta}[0]) \vp \|_{L_x^2}+\| ( G_{a,\eta}[\zeta_c]-G_{a,\eta}[0]) \chi(D_x)\vp \|_{H_x^1}\lesssim \ep^2 \|\vp\|_{L_x^2}.
  \eeqs
This in turn, by using \eqref{expansion-Gaeta} again, gives that
\beqs 
\| ( G_{a,\eta}[\zeta_c]-G_{a,\eta}[0]) \chi(D_x)\vp \|_{L_x^2}\lesssim \ep^2 \|\big(G_{a,\eta}[\zeta_c],(\p_x-a),\eta\big)(\chi(D_x)\vp)\|_{L_x^2}\lesssim \ep^2\|\vp\|_{H_{*,\eta}^{\f12}}.
\eeqs
  It now remains for us to show that 
\beqs 
\|\tilde{\chi}(D_x) \big( G_{a,\eta}[\zeta_c]-G_{a,\eta}[0]\big) \tilde{\chi}(D_x)\|_{B(L_x^2)}\lesssim \ep^2.
\eeqs
To achieve such a goal, we expand each term in the right side of \eqref{expansion-Gaeta} and identify cancellations for terms involving two derivatives and one derivative. 
 For brevity, let's denote:
$$b_s=\f{1}{1+(s\p_x\zeta_c)^2}\, , \quad B_s=s\,b_s\,\zeta_c\,\p_x\zeta_c\, .$$
Thanks to Lemma \ref{lem-Gaeta}, we can  conclude that: 
\begin{align*}
    %-\tilde{\xhi}(D_x)
   &\qquad  G_{a,\eta}[s\zeta_c](\zeta_c Z_{a,\eta})
    =\mu_{a,\eta}(D_x)( b_s \zeta_c\mu_{a,\eta}(D_x)+B_s \p_x )\\
    &=b_s \zeta_c \mu_{a,\eta}^2(D_x)+B_s\mu_{a,\eta}(D_x)(\p_x-a)+[\mu_{a,\eta}(D_x), B_s](\p_x-a)+[\mu_{a,\eta}(D_x), b_s \zeta_c]\mu_{a,\eta}(D_x)+\cR\\
    & =\colon \cT_{11}+ \cT_{12}+ \cT_{13}+ \cT_{14}+\cR\, , 
    \\[4pt]
     &\,\,\qquad(\p_x-a)(s\p_x\zeta_c\cdot)\big)(\zeta_c Z_{a,\eta})\\
    &=s B_s\p_x\zeta_c (\p_x-a)^2 +B_s(\p_x-a)\mu_{a,\eta}(D_x)+(\p_xB_s) \mu_{a,\eta}(D_x)+\p_x(s B_s\p_x\zeta_c)(\p_x-a)+\cR\\
    &=\colon \cT_{21}+ \cT_{22}+ \cT_{23}+ \cT_{24}+\cR\, .
\end{align*}
Hereafter we denote by  $\cR$  some operator such that 
$\tilde{\chi}(D_x)\cR \tilde{\chi}(D_x)\in B(L_x^2)$ with operator norm proportional to $\ep^2$ and independent of $\eta$ for  $|\eta|\leq 2.$ 
We write also
\begin{align*}
 (\p_x-a)(\zeta_c(\p_x-a)) =\zeta_c(\p_x-a)^2+\p_x  \zeta_c (\p_x-a)^2=\colon \cT_{31}+\cT_{34}.
\end{align*}
By direct computations, we find that
\begin{align*}
    \cT_{11}-\cT_{21}+\cT_{31}=\cR, \quad \cT_{12}-\cT_{22}=0
\end{align*}
and thus the second order vanishes. 
Next, by using pseudodifferential calculus and in particular the classical fact that the principal symbol of the commutator of two pseudo differential operators is  the Poisson bracket of their symbols,  we get that
\beqs 
%\tilde{\chi}(D_x)%\tilde{\chi}(D_x)=
[\mu_{a,\eta}(D_x), B_s]=\f{1}{i}\op\big(\p_x B_s \p_{\xi}\mu_{a,\eta}\big)=-\op \bigg(\p_x B_s \, \f{i\xi-a}{\mu_{a,\eta}}\bigg)+\cR\, . %- \f{\p_x-a}{\mu_{a,\eta}}+\cR
\eeqs
Consequently, we have that 
\beqs
\cT_{13}-\cT_{23}=\p_x B_s\,\op \bigg(\f{(\xi+ia)^2-\mu_{a,\eta}^2}{\mu_{a,\eta}}\bigg)+\cR=\cR\, .
\eeqs
Similarly, we compute that
\begin{align*}
    \cT_{14}-\cT_{24}+\cT_{34}=\p_x(b_s \,\zeta_c)\big[(\p_x-a)+\f{1}{i}\op\big(\mu_{a,\eta}\,\p_{\xi}\mu_{a,\eta}\big)\big]+\cR =\cR\, .
\end{align*}
We see henceforth that the first order vanishes and  the proof is thus finished.
\end{proof}

\begin{lem}\label{lem-G+Delta}
Let $\pi_r(\xi,\eta)$ be the characteristic function of the set $\cS_{K,\delta}$ defined in \eqref{def-singset}, it holds that: 
\beqs 
\|%\rho_{_{K,\delta}}
\pi_s({D})  \big(G_a[\zeta_c]+\Delta_a\big)\|_{B(L^2)}\lesssim (K\ep)^3.
\eeqs
\end{lem}
\begin{proof}
   We write $G_a[\zeta_c]+\Delta_a=G_a[\zeta_c]-G_a[0]+G_a[0]+\Delta_a.$
   Let $\mu_a=\sqrt{(\xi+ia)^2+\eta^2}.$ On the support of $\pi_s(\xi,\eta),$ it holds that
$|\mu_a|\leq 5K\ep.$ By using a  Taylor expansion, we can express
   $$m(G_a[0]+\Delta_a)=\mu_a (\tanh \mu_a-\mu_a)=\f{1}{3}\mu_a^3(1+\cO(\mu_a))\,.$$
 This leads to the conclusion that:
   \beqs 
\|  \pi_s({D})  (G_a[0]+\Delta_a)\|_{B(L^2)}\lesssim (K\ep)^3.
   \eeqs
Moreover,  by the
 Taylor formula for $G_a[\zeta_c]$ in terms of $\zeta_c:$  %{\color{red}check
\begin{align}\label{Ga-talor1er}
  ( G_a[\zeta_c]-G_a[0])\vp=\int_0^1  \big(-G_a[s\zeta_c]+s(\p_x-a)(\p_x\zeta_c\cdot)\big)(\zeta_c Z_a(s\zeta_c,\vp))-\div_a(\zeta_c\na_a\vp) \, \d s, \quad 
\end{align}
    where $Z_a(s\zeta_c,\vp)=\f{G_a[s\zeta_c]\varphi+s\p_x\zeta_{c}(\p_x-a)\vp}{1+|s\p_x\zeta_{c}|^2}.$
    Recalling $\|(\ep^{-2}\zeta_c, \ep^{-3}\p_x \zeta_c)\|_{L_x^{\infty}}\lesssim 1,$ we find that:
\beqs 
\|\pi_{s}({D}) \big(G_a[\zeta_c]-G_a[0]\big)\|_{B(L^2)}\lesssim K^2\ep^4.
\eeqs
\end{proof}

\begin{lem}\label{lem-remaider-sec}
Let $\pi_r(\xi,\eta)$ be the characteristic function of the set $\cS^c_{K,\delta}$ defined in \eqref{def-reguset}, the following property holds:
    \begin{align}\label{remainder-sec}
       \| \pi_r(D)\big(G_a[\zeta_c]-G_a[0]+Q_a(D) \div_a(\zeta_c\na_a Q_a(D))\big) \pi_r(D)\|_{B({L^2})}\lesssim \ep^3.
    \end{align}
\end{lem}
\begin{proof}
    On the one hand, it holds that: 
    \begin{align*}
        Q_a(D) \div_a(\zeta_c\na_a Q_a(D))=G_a[0] \zeta_c G_a[0]+\div_a(\zeta_c \na_a)+\cR\, ,
    \end{align*}
    where $\|\cR\|_{B(L^2)}\lesssim \ep^3.$ On the other hand,  using the Taylor formula \eqref{Ga-talor1er} together with Lemma \ref{lem-Ga-Ga0-g}, one readily gets that:
   \begin{align*}
       \| \pi_r(D)\big(G_a[\zeta_c]-G_a[0]+G_a[0] \,\zeta_c \,G_a[0]+\div_a(\zeta_c \na_a)\big) \pi_r(D)\|_{B({L^2})}\lesssim \ep^4.
    \end{align*}
    The estimate \eqref{remainder-sec} then follows.
\end{proof}
\begin{lem}\label{lem-Ga-ll}
    Let $ \chi_{\ell}: \mR^2\rightarrow \mR $ be a function supported on 
 $\big\{(\xi,\eta)\big|\,|\xi|\leq 4K\ep, |\eta|\leq A\ep^2\big\},$ it then holds that:
 \begin{align*}
      \| \big(G_a[\zeta_c]-G_a[0]\big) \chi_{\ell}(D)\|_{B(H_{*}^{\f12}, L^2)}\lesssim (K\ep^3+A^2\ep^5)\,.
 \end{align*}
\end{lem}
\begin{proof}
    It follows directly from \eqref{Ga-talor1er}.
\end{proof}
\begin{lem}\label{lem-clatla}
   Let the operator $\cL_a^2$ %\tilde{L}_a$
   be defined in  \eqref{defcL0-1} and $\bI_{\delta}(\cdot): \mR\rightarrow \mR$ be the characteristic function on $[-\delta, \delta].$ 
    It holds that  
\beq\label{es-cLa2}
\|\cL_a^2 \, \bI_{\delta}(D_x) \|_{B(X)}+\|\bI_{\delta}(D_x)\, \cL_a^2\|_{B(X)} \lesssim \ep^2.
\eeq
\end{lem}
\begin{proof}%[Proof of \eqref{es-cLa2}]
%Let us give the proof of \eqref{es-cLa2}. That 
%It is the consequence of the following estimates: 
We can control the operator norm of each entry of $\cL_a^2$ in the following manner: 
 \begin{align*}
&\|\bI_{\delta}(\p_x-a)(v_c )\|_{B(L^2)}+\|(\p_x-a)(v_c \bI_{\delta} )\|_{B(L^2)} \lesssim \|v_c\|_{W^{1,\infty}_x}\lesssim \ep^2, \\ 
&\|\bI_{\delta}(d_c \p_x Z_c )\|_{ B(L^2, H_{*}^{\f12})}+ \|(d_c \p_x Z_c ) \bI_{\delta} \|_{ B(L^2, H_{*}^{\f12})}\lesssim \|d_c \p_x Z_c  \|_{W^{1,\infty}_x}\lesssim \ep^3, \\
&\|v_c (\p_x-a) \bI_{\delta}\|_{B( H_{*}^{\f12})}\lesssim \|v_c\|_{W^{1,\infty}_x}\| (\p_x-a) \bI_{\delta}\|_{B( H_{*}^{\f12})}\lesssim \ep^2, \\
&\|\bI_{\delta}\, v_c\, \p_x \tilde{\bI}_{\delta}\|_{B( H_{*}^{\f12})}\lesssim \|\bI_{\delta}\,\p_x (v_c \tilde{\bI}_{\delta})\|_{B( H_{*}^{\f12})}+\|\bI_{\delta}\,\p_x v_c \tilde{\bI}_{\delta})\|_{B( H_{*}^{\f12})}\lesssim \|v_c\|_{W^{1,\infty}_x}\lesssim \ep^2.
 \end{align*}
\end{proof}

In the following lemma, we study the remainder of the linearized operator $L_a$ after subtracting the constant part, denoted by  the operator $L_a^1$  defined in \eqref{def-La0-1}. We  present various estimates useful for establishing the resolvent bounds. Notably, %we underscore the significance of the estimate \eqref{La1-commu}, 
  we emphasize that
the estimate \eqref{La1-commu}, although quite natural and not hard to prove, is particularly useful  in analyzing the low transverse frequency regime detailed in Section \ref{sec-lowtransverse}. The essence of this estimate lies in the appearance of an additional factor of $K^{-1}$ in front of  the low-frequency component.
 %The point of this estimate is that one could get an extra $K^{-1}$ in front of the low frequency part. 
\begin{lem}\label{lem-La1}
Let $L_a^1$ be defined in \eqref{def-La0-1}. Let $\pi_r: \mR^2\rightarrow \mR, $ be a function supported on 
%be a characteristic function on 
%be such that $\Supp \pi_r \subset
$\big\{(\xi,\eta)\big|\,|\xi|\leq \delta, |\eta|\leq 2\big\}$ 
and $\chi, \chi_1: \mR\rightarrow \mR$ be two smooth functions defined in \eqref{def-smoothcutoff}
%two smooth function 
%supported on $\big\{\xi\big|\,|\xi|\leq 2K\ep \big\}$
% $\big\{\xi\big|\,|\xi|\leq 4K\ep\big\}$ respectively, and satisfying 
then 
 \beq \label{es-tla}
\|\pi_r(D){L}_a^1\|_{B(X)}\lesssim \ep^2\, ,
\eeq
    \beq\label{es-R}
  %  \|L_a^1( D)\chi_{\ell}(D)\|_{B(X)}\lesssim (K\ep^3+A^2\ep^5)\, .
\|L_a^1(x, D)\,\chi_1(D_x)\,\bI_{A\ep^2}(D_y)\|_{B(X)}\lesssim (K\ep^3+A^2\ep^5)\, .
\eeq
Moreover, for any $U\in X,$ it holds that for any $K^4\ep\leq 1,$
\beq \label{La1-commu}
\|[\chi_1(D_x),\,L_a^1(x, D)]\,\bI_{A\ep^2}(D_y) U\|_{X}\lesssim \ep^3\big( K^{-1}\|\chi(D_x)\, U\|_{X} +\|(1-\chi_1)(D_x) U\|_{X} \big) .%\chi(D_x)\|_{B(X)}\lesssim K^{-1}\ep^3
\eeq
\end{lem}
\begin{proof}
%The proof is the consequence of  Lemma \eqref{lem-Ga-ll}and the similar arguments in the proof of \eqref{es-cLa2}. 
The proofs of \eqref{es-tla} and \eqref{es-R}
come  from applying Lemma \eqref{lem-Ga-ll} and using  arguments analogous to the ones  in the proof of \eqref{es-cLa2}, thus the details are omitted. We now explain how to get \eqref{La1-commu}.   
First, by applying Lemma \ref{lem-commutator} for $s=1,$ we 
have that 
\beqs 
\big\|\big[\chi_1(D_x),\, d_c \p_x Z_c\big] \big\|_{B(L^2, H_{*}^{\f12})} \lesssim \|\chi_1'\|_{L_{\xi}^{\infty}} \|\cF_{x\rightarrow \xi}(\p_x( d_c \p_x Z_c))\|_{L_{\xi}^1}\lesssim (K\ep)^{-1}\ep^4=K^{-1}\ep^3.
\eeqs
Note that $\|\chi_1'\|_{L_{\xi}^{\infty}} \lesssim (K\ep)^{-1}$ and 
that we have by definition that $\p_x( d_c \p_x Z_c)$
has the form $\ep^4 f_{\ep}(\ep x),$ where $f$ is a smooth and exponentially localized function. %function with exponentially localization. 
Next, it follows  again from Lemma \ref{lem-commutator} that 
\begin{align}\label{es-com-0}
  \|\p_x [\chi_1, v_c]\|_{B(L^2)}\lesssim K\ep  \|\chi_1'\|_{L_{\xi}^{\infty}} \|\cF(\p_x v_c)\|_{L_{\xi}^1}\lesssim \ep^3. 
\end{align}
Moreover, by applying \eqref{es-commutator} for $s=2,$ we have that, since $ \chi(\xi)=0 $ for $ |\xi|\geq 2K\ep,$
\begin{align*}
     \|\p_x [\chi_1(D_x), v_c] \chi(D_x)U\|_{L^2}\lesssim C_2  
    \, \|\cF(\p_x^2 v_c)\|_{L_{\xi}^1}\, \|\chi(D_x) U\|_{L^2} \, % \lesssim \ep^3.
\end{align*}
where 
\beqs 
C_2=\sup_{\xi\in \mR, \, |\xi'|\leq 2K\ep}
\bigg| \f{ \chi_1(\xi)-\chi_1(\xi')}{|\xi-\xi'|^2}  \xi\bigg|\, .
\eeqs
Since $
\chi_1(\xi)=1$ for $|\xi|\leq 3K \ep,$ it holds that
\beqs 
C_2 \lesssim \big\|\chi_1'\big\|_{L_{\xi}^{\infty}}\lesssim (K\ep)^{-1}.
\eeqs
We thus get that 
\beq \label{es-commu-lowfreq}
 \|\p_x [\chi_1(D_x), v_c] \chi(D_x)U\|_{L^2}\lesssim K^{-1}\ep^3 \|\chi(D_x) U\|_{L^2} \, .
\eeq
Writing $v_c \, \p_x\cdot =\p_x (v_c \, \cdot)-\p_x v_c \cdot,$ we have  similar estimate on the operator norm for 
$ [\chi_1(D_x), v_c \p_x]$ in ${B\big(H_{*}^{\f12}, H_{*}^{\f12}\big)}$. 
We are now left to control $[\chi_1(D_x), G_a[\zeta_c]-G_a[0]]$ in $B(H_{*}^{\f12}, L^2),$ which is based on the expansion \eqref{Ga-talor1er}.  Let us sketch the estimate for  $[\chi_1(D_x), \, \div_a[\zeta_c \na_a )], $ the other ones can be controlled in the same way upon applying Lemma \ref{lem-Ga-Ga0-g}.
%enjoy better property. 
Let us first notice that for $A^2\ep\leq 1,$
$$\|\p_y^2 \, \bI_{A\ep^2}(D_y) (\zeta_c \, \cdot)\|_{B(H_{*}^{\f12}, L^2)}\lesssim A^2\ep^4\cdot \ep^2 \cdot \ep^{-1}\lesssim \ep^4.$$
Next, write $\p_x \big(\zeta_c\, \p_x \big)= \p_x \big(\zeta_c \, \chi(D_x)\p_x \big)+\p_x^2(\zeta_c (1-\chi)(D_x))-\p_x\big(\p_x \zeta_c (1-\chi)(D_x)\big). $
We have on the one hand, as in  the proof of \eqref{es-commu-lowfreq} that
\beqs 
\|[\chi_1(D_x), \, \zeta_c ] \, \chi(D_x)\p_x U\|_{L^2}\lesssim K^{-1}\ep^3\|\chi(D_x)\p_x U\|_{L^2}\lesssim K^{-1}\ep^3 \|\chi(D_x) U\|_{H_{*}^{\f12}},
\eeqs
and on the other hand, as in  the proof of \eqref{es-com-0} that,
\beqs 
\|\big[\chi_1(D_x), \p_x^2(\zeta_c\cdot)-\p_x(\p_x \zeta_c)\big]\|_{B(L^2)}\lesssim K\ep^4+\ep^3 \lesssim \ep^3.
\eeqs
%by using the similar arguments as in the 

\end{proof}
\section{Generalized kernel of $L(0)$} %Study of the one dimensional solitary water waves}
Let $L(0)$  be the (modified) operator for the  one dimensional water waves linearized about  the solitary wave:
\beqs 
 L(0)=:\left( \begin{array}{cc}
   \p_x(d_c\cdot)  &  G_0[\zeta_{c}]   \\[5pt]
 -  w_c& d_c\p_x
\end{array}\right)\, ,
\eeqs
where $ G_0[\zeta_{c}]$ is the usual Dirichlet-Neumann operator for one dimensional water waves.
 In this appendix, we look for the generalized kernel of $L(0)$ in the weighted space $Y_a.$ This is done by using the translational and Galilean invariance of the original linear operator 
 \beq \label{relationLtL}
 \tilde{L}(0)=:\left( \begin{array}{cc}
   1   & 0 \\
   -Z_c   & 1
 \end{array}\right)^{-1}L(0)\left( \begin{array}{cc}
   1   & 0 \\
   -Z_c   & 1
 \end{array}\right)
 \eeq
and then transforms to $L(0).$ By definition, the background 1-d waves $(\zeta_c, \vp_c)^t$ solves the equation
\beqs \label{profile eq}
    \left\{ \begin{array}{l}
       \p_x \zeta_c+ G_0[\zeta_c]\varphi_c=0, \\
         \p_x\vp_c-\f12 (\p_x\varphi_c)^2+\f12 \f{(G[\zeta]\varphi_c+\p_x\vp\cdot\p_x\zeta_c)^2}{\sqrt{1+|\p_x\zeta_c|^2}}-\gamma\zeta_c=0\, . 
    \end{array}
    \right. 
\eeqs
First, by the translational invariance, one gets by differentiating the above profile equations that 
\beqs 
\tilde{L}(0)\left(\begin{array}{c}
     \zeta_c'  \\
    \vp_c'  
\end{array}\right)=0\, .
\eeqs
To find the other element in the generalized kernel, it is convenient to use the  %reverse the 
change of variable %\eqref{changeofvarible}:
\beqs 
X= hx, \quad Z=hz, \quad \tilde{\zeta}(X)=h\,\zeta(x), \quad \tilde{\vp}(X)=ch\,\vp(x)
\eeqs
and study the obtained  system where the wave speed $c$ appears explicitly: 
\beqs %\label{ww-sur-mov}
    \left\{ \begin{array}{l}
      c\p_X \tilde{\zeta}_c+ G_0[\tilde{\zeta}_c]\tilde{\varphi}_c=0\, , \\
        c\p_x\tilde{\vp}_c-\f12 (\p_X \tilde{\varphi}_c)^2+\f12 \f{(G_0[\tilde{\zeta}_c]\tilde{\varphi}_c+\p_X\tilde{\vp}_c\cdot\p_X\tilde{\zeta}_c)^2}{\sqrt{1+(\p_X\tilde{\zeta}_c)^2}}-g\tilde{\zeta}_c=0\, . 
    \end{array}
    \right. 
\eeqs
Since $\tilde{\zeta}_c,\, \tilde{\vp}_c $ depends smoothly on the wave speed $c,$ we can
differentiate the above equation in terms of $c$ to find that
\beq\label{diff-c-ori}
\begin{aligned}
\left\{ \begin{array}{l}
c\p_X\p_c \tilde{\zeta}_c+G_0[\tilde{\zeta}_c]\p_c\tilde{\varphi}-G_0[\tilde{\zeta}_c]\big(\tilde{Z}_c\,\p_c\tilde{\zeta}_c\big)-\p_X\big(\tilde{v}_c\p_c\tilde{\zeta}_c\big)=-\p_X \tilde{\zeta}_c\,,\\[3pt]
c\p_X\p_c \tilde{\vp}_c-\tilde{v}_c\p_X\p_c \tilde{\vp}_c+ \tilde{Z}_c \, G_0[\tilde{\zeta}_c]\big(\p_c\tilde{\vp}_c-\tilde{Z}_c\p_c \tilde{\zeta}_c\big)+\tilde{\zeta}_c\p_X \tilde{v}_c \p_c \tilde{\zeta}_c-g \p_c \tilde{\zeta}_c=\p_X \tilde{\vp}_c,\, 
 \end{array}
    \right. 
\end{aligned}
\eeq
where 
\beqs
\tilde{Z}_c=\f{G_0[\tilde{\zeta}_c]\tilde{\varphi}_c+\p_X\tilde{\vp}_c\cdot\p_X\tilde{\zeta}_c}{{1+(\p_X\tilde{\zeta}_c)^2}}, \qquad \tilde{v}_c=\p_X \tilde{\vp}_c-\tilde{Z}_c\p_X \tilde{\zeta}_c\, .
\eeqs
Since $\p_X\tilde{\vp}_c(X)=c\,\p_x \vp_c(x),\, \p_X\tilde{\zeta}_c(X)=\p_x \zeta_c(x),$ it holds that
$\tilde{Z}_c(X)=c Z_c(x), \, \tilde{v}_c(X) = c v_c(x).$
Therefore, changing back to the variable $x,$ we get from \eqref{diff-c-ori} that
\beqs
\tilde{L}(0)\left(\begin{array}{c}
 c\p_c \zeta_c\\ [3pt] % \p_c \tilde{\zeta}_c/h  \\
    \p_c(c {\zeta}_c)
\end{array}\right)=-\left(\begin{array}{c}
{\zeta}_c' \\ [3pt]
{\vp}_c'
\end{array}\right). 
\eeqs
It then stems from the relation \eqref{relationLtL} the basis of the generalized kernel of $L(0):$ 
\beqs
V_0=\left(\begin{array}{c}
  \zeta_c'\\ [3pt] % \p_c \tilde{\zeta}_c/h  \\
\vp'_c-Z_c{\zeta}_c'
\end{array}\right)=\left(\begin{array}{c}
  \zeta_c'\\ [3pt] % \p_c \tilde{\zeta}_c/h  \\
v_c
\end{array}\right), \qquad V_1=\left(\begin{array}{c}
 c\p_c \zeta_c\\ [3pt] % \p_c \tilde{\zeta}_c/h  \\
-cZ_c\p_c{\zeta}_c+\p_c(c\vp_c)
\end{array}\right),
\eeqs
which satisfy:
\beqs%\label{id-generalkernel}
L(0)V_0=0\,, \qquad \quad \, L(0)V_1=-V_0 \, .
\eeqs
Let us remark that the fact that the kernel of $L(0)$ is of algebraic multiplicity $2$ is related to the same property of the linearized operator of KdV equation around KdV soliton, see Appendix C in \cite{Pego-Sun}.
At this stage, it is also useful to notice that, by the  above identities, $v_c$ and $\vp_{1c}=\colon -cZ_c\p_c{\zeta}_c+\p_c(c\vp_c) $  solve respectively the equation: 
\beqs
\begin{aligned}
&\qquad\qquad\bigg(\p_x \bigg(\f{d_{c}^2}{w_{c}}\p_x\bigg)+{G}_{0}[{\zeta_{c}}] \bigg) v_c=0 \,, \\
&\p_x\bigg( \f{d_c}{w_c}(v_c+d_c\p_x \vp_{1c})\bigg)+G_0[\zeta_c]\vp_{1c}=- \f{d_c}{w_c}\p_x v_c\, . 
\end{aligned}
\eeqs
Consequently, it follows from the change of variable \eqref{scaling} that, % and unknown \eqref{def-dwep} 
  $\hat{v}_{\ep}(\hat{x})=:\ep^{-2}v_c(\hat{x}/\ep), \, \hat{\vp}_{1\ep}(\hat{x})=: \ep%^{-3}
{\vp}_{1c}(\hat{x}/\ep)$ satisfy:
\beq\label{id-twomodes}
\begin{aligned}
&\qquad \qquad \qquad \bigg(\ep^{-2}\p_{\hat{x}} \bigg(\f{\hat{d}_{\ep}^2}{\hat{w}_{\ep}}\p_{\hat{x}}\bigg)+\ep^{-4}\hat{G}_{0}[{\zeta_{c}}] \bigg) \hat{v}_{\ep}=0,\\
&\bigg(\ep^{-2}\p_{\hat{x}} \bigg(\f{\hat{d}_{\ep}^2}{\hat{w}_{\ep}}\p_{\hat{x}}\bigg)+\ep^{-4}\hat{G}_{0}[{\zeta_{c}}] \bigg) \hat{\vp}_{1\ep}=-\bigg(\p_{\hat{x}} \bigg(\f{\hat{d}_{\ep}}{\hat{w}_{\ep}}\bigg)+\f{2\hat{d}
_{\ep}}{\hat{w}_{\ep}}\p_x\bigg)\hat{v}_{\ep},
\end{aligned}
\eeq
where $\hat{d}_{\ep}, \hat{w}_{\ep}$
are defined in \eqref{def-dwep}.

\section{The error  between the resonant modes %generalized eigenfunctions
of water waves and KP-II } \label{approx-resonants}%
In this appendix, %Approximation of the
we approximate the  resonant modes for  the transformed linearized operator $L_a(\ep^2\hat{\eta})$ for the water waves by those of $L^{\hat{a}}_{\kp}(\hat{\eta})=e^{\hat{a}x} L_{\kp}(\hat{\eta}) e^{-\hat{a} x}.$ 
\subsection{Resonant modes for the linearized operator of KP-II equation} 
In this subsection, we recall some results proven in 
\cite{Mizumachi-KP-nonlinear},
pertaining to the resonant modes of   $L_{\kp}-$ the linearized operator for the  KP-II equation about  the line KdV soliton. %$\Psi_{\kdv}(\cdot)=\sech^2\big( \f{\sqrt{3}}{2}\cdot\big)$
This is  useful in approximating the water waves equation by {KP-II} in the low frequency regime. 

Let 
\beqs 
L_{\kp}(\eta)=e^{-iy\eta}\,L_{\kp}\,e^{iy\eta}=\p_x \big( 1-\f13 \p_x^2-3\Psi_{{\kdv}}\cdot \big)+\eta^2 \p_x^{-1}.
\eeqs
Performing the change of variable, 
\beqs 
\tilde{x}=\f{\sqrt{3}}{2} x, \,\quad  \tilde{\eta}=\f43 \,\eta, \quad \tilde{f}(\tilde{x}, \tilde{\eta})=f(t, x)\,.
\eeqs
Then $(L_{\kp}(\eta) f)(x,\eta)=-\f{\sqrt{3}}{8}\big( \p_{\tilde{x}}^3-4\p_{\tilde{x}}+6\p_{\tilde{x}}(\sech^2 \cdot)+3\tilde{\eta}^2\p_{\tilde{x}}^{-1}\big)\tilde{f}=\colon \f{\sqrt{3}}{8} \tilde{L}_{\kp}(\tilde{\eta})\tilde{f}.$ 
Based on the result in Lemma 2.1, \cite{Mizumachi-KP-nonlinear}, we know that the eigenmodes of $ \tilde{L}_{\kp}(\tilde{\eta}), \, \big(\tilde{L}_{\kp}(\tilde{\eta})\big)^{*}$ are 
\begin{align*}
\tilde{\lambda}_{\kp}(\tilde{\eta})=4i \tilde{\eta} \sqrt{1+i\tilde{\eta}}, \qquad \tilde{g}_0(\tilde{x}, \tilde{\eta})=\f{1}{\sqrt{1+i\tilde{\eta}}} \p_{\tilde{x}}^2\big( e^{-\sqrt{1+i\tilde{\eta}}\,\tilde{x}}\sech (\tilde{x})\big);\\
\overline{\tilde{\lambda}_{\kp}}(\tilde{\eta})=-4i \tilde{\eta} \sqrt{1-i\tilde{\eta}}, \qquad \tilde{g}^{*}_0(\tilde{x}, \tilde{\eta})=\f{i}{2\tilde{\eta}} \p_{\tilde{x}}\big( e^{\sqrt{1-i\tilde{\eta}}\,\tilde{x}}\sech (\tilde{x})\big).
\end{align*}
Changing back to the variable $(x,\eta),$ we find that
\beq\label{resomode-kp}
\begin{aligned}
  &  {\lambda}_{\kp}(\tilde{\eta})=-\f{i {\eta} }{2\sqrt{3}}\sqrt{1+4i\eta/3}\,, \qquad {g}_0(x, {\eta})=\f{1}{\sqrt{1+4i{\eta}/3}} \p_{\tilde{x}}^2\big( e^{-\sqrt{1+4i{\eta}/3}\,\tilde{x}}\sech (\tilde{x})\big)\,;\\
&\overline{{\lambda}_{\kp}}({\eta})=\f{i\eta}{2\sqrt{3}}  \sqrt{1-4i{\eta}/3}\,, \qquad \qquad {g}^{*}_0(x\, \tilde{\eta})=\f{3i}{8{\eta}} \p_{\tilde{x}}\big( e^{\sqrt{1-4i{\eta}/3}\,\tilde{x}}\sech (\tilde{x})\big)\,.
\end{aligned}
\eeq
As in \cite{Mizumachi-KP-nonlinear}, to resolve the degeneracy of $g_0(x, \eta)$ and the singularity of $g_0^{*}(x, \eta)$ when $\eta=0,$ we define the basis and dual basis 
\begin{align*}
g_{_{01}}(\cdot, \eta)= g_0(\cdot, \eta)+ g_0(\cdot, -\eta), \qquad g_{_{02}}(\cdot, \eta)=\f{1}{ i \eta} \big(g_0(\cdot, \eta)- g_0(\cdot, -\eta)\big); \\
g_{_{01}}^{*}(\cdot, \eta)=\f12 \big(g_0^{*}(\cdot, \eta)+ g_0^{*}(\cdot, -\eta)\big), \qquad g_{_{02}}^{*}(\cdot, \eta)=\f{\eta}{2 i} \big(g_0^{*}(\cdot, \eta)- g_0^{*}(\cdot, -\eta)\big)\,. 
\end{align*}
As is clear in the definitions, all these functions  are even in terms of $\eta,$ which leads to 
\begin{align*}
   \| g_{_{0k}}(\cdot, \eta)-  g_{_{0k}}(\cdot, 0)\|_{L^2}\lesssim \eta^2,\, \qquad  \| g^{*}_{_{0k}}(\cdot, \eta)-  g^{*}_{_{0k}}(\cdot, 0)\|_{L^2}\lesssim \eta^2 .
\end{align*}
Let $$v_0(x)=\Psi_{\kdv}(x)=\sech^2\big( \f{\sqrt{3}}{2} x\big)\,, \quad \vp_{10}(x)=2v_0 (x)+{x}\p_x v_0 (x)\,,$$ 
it follows from direct calculations that
\beq\label{base-kp}
\begin{aligned}
&\qquad g_{_{01}}(\cdot, 0)=v_0'\, , \qquad\qquad  g_{_{02}}(\cdot, 0)=-\f{1}{\sqrt{3}}v_0'-\f{\vp_{10}}{2}\,; \\
&g_{_{01}}^{*}(\cdot, 0)=-\f{\sqrt{3}}{4}%\p_x^{-1}
\int_{-\infty}^x \vp_{10}\,, \qquad g_{_{02}}^{*}(\cdot, 0)=-\f{\sqrt{3}}{2}v_0\, .
\end{aligned}
\eeq
\subsection{Expansions of the basis and dual basis for water waves}\label{approx-resonants-sec2}
%Now, for later use, 
We study the expansions of the basis in terms of $\ep$ and $\hat{\eta}_0.$
By the definition \eqref{defU-base}, expansion \eqref{expan-modes-1} as well as the facts $%\lambda(\eta)=\cO(\ep^3),\,
w_c=1+\cO(\ep^2),\, d_c=1+\cO(\ep^2),$ it holds that
% and the expansion \eqref{expan-modes-1}, we have that \beqs (\p_x U_2)(\ep^{-1}x, \ep^{2}\hat{{\eta}})=\ep\, \psi (x, \hat{\eta})=\ep \p_x \big(v_{\ep}- i \Lambda_{1\ep} \hat{\eta}  \,\vp_{1\ep}\big)(x)+\cO(\ep\hat{\eta}^2).\eeqs
%Moreover, since it holds that\beq \label{U1}U_1(\ep^{-1}x, \ep^{2}\hat{{\eta}})=\ep\, \psi (x, \hat{\eta})+\cO(\ep^2)=\ep \p_x \bigg(v_{\ep}-\big(i\,\Lambda_{1\ep} \hat{\eta} - \tilde{\Lambda}_{2\ep}\hat{\eta}^2\big)\,\vp_{1\ep}+\hat{\eta}^2 \tilde{\psi}_{2\ep}(\cdot, \hat{\eta}) \bigg)(x)+\cO(\ep^2)%+\ep\hat{\eta}^2).\eeq
%Consequently, %$V=\diag (1, \p_x) U,$ it then holds that
\beq\label{U}
\begin{aligned}
\big(\diag (1, \p_x) U\big)(\ep^{-1}x, \ep^2\hat{\eta})&=\ep\, \psi(x, \hat{\eta})\left(\begin{array}{c}
   1  \\
    1 
\end{array} \right)+\cO\big(\ep(\ep+\hat{\eta}^2)\big)\\
&= \ep \bigg(v'_{\ep}- i\,\Lambda_{1\ep} \hat{\eta}\, \vp'_{1\ep} \bigg) \left(\begin{array}{c}
   1  \\
    1 
\end{array} \right)+\cO\big(\ep(\ep+\hat{\eta}^2)\big) \, .
\end{aligned}
\eeq
%\beqs \Re V =\ep\, v_{\ep}' \left(\begin{array}{l}  1  \\   1 \end{array} \right)+\cO(\ep^2+\hat{\eta}^2)\,, \qquad \Im V =- \ep\,\hat{\eta} \,\Lambda_{1\ep} \, \vp_{1\ep}' \left(\begin{array}{l}  1  \\   1 \end{array} \right)+\cO(\ep^2+\hat{\eta}^2)\, .\eeqs
Next, %by the definition,
as $U^{*}(\cdot, \eta)=(U_2, U_1)(-\,\cdot, -\eta),$ we have by using the fact that $v_{\ep},\, \vp'_{1\ep}$ are even that %, $\vp_{1\ep}$ is odd that
\beq \label{U*}
\big(\diag (1, \p_x^{-1}) U^{*}\big)(\ep^{-1}x, \ep^2\hat{\eta}) = \bigg(v_{\ep}-i\,\Lambda_{1\ep} \,\hat{\eta}\,%\big(i\,\Lambda_{1\ep} \,\hat{\eta} +\tilde{\Lambda}_{2\ep}\hat{\eta}^2\big) \int_{-\infty}^x 
\vp_{1\ep}' %\vp_{1\ep} 
\bigg) \left(\begin{array}{c}
     1  \\
   - 1 
\end{array} \right)+\cO(\ep^2+\hat{\eta}^2) \, .
\eeq
We now compute 
\begin{align*}
   & \langle U(\cdot, \eta), U^{*}(\cdot, \eta)  \rangle=2 \int U_1(x, \eta)\, \overline{U_2(x, \eta)}\, \d x \\
   & =2\ep^{-1}\int U_1(\ep^{-1}x, \ep^2\hat{\eta})\, \overline{U_2(\ep^{-1}x, \ep^2\hat{\eta})}\, \d x =2 \int \psi(x, \hat{\eta}) \int_{-\infty}^x \overline{\psi(x', \hat{\eta})} \,\d x' \d x .
\end{align*}  
 In view of the expansions of $\psi(\cdot, \hat{\eta})$ in \eqref{expan-modes-1},  we find after some calculations that
 \begin{align*}
    \Re \big\langle U(\cdot, \eta), U^{*}(\cdot, \eta) \big \rangle &=2
\, \hat{\eta}^2 \bigg( \Lambda_{1\ep}^2 \int {\vp_{1\ep}'} \int_{-\infty}^x {\vp_{1\ep}'} \, \d x' \d x+2 \tilde{\Lambda}_{2\ep} \int v_{1\ep} \vp_{1\ep}'+\cO(\hat{\eta}^2)\bigg)\\
& =2\, \hat{\eta}^2 \bigg(\f12 \Lambda_{10}^2 \|v_{0}\|_{L^1(\mR)}^2+3 \tilde{\Lambda}_{20} \|v_0\|_{L^2(\mR)}^2+\cO(\ep+\hat{\eta}^2) \bigg) \\
&=\f23\, \hat{\eta}^2 \|v_{0}\|_{L^1(\mR)}^2+\cO\big(\hat{\eta}^2(\ep+\hat{\eta}^2)\big)
=\hat{\eta}^2 \bigg(\f{32}{9}+\cO\big(\ep+\hat{\eta}^2\big)\bigg),
 \end{align*}
and 
\begin{align*}
\Im  \big\langle U(\cdot, \eta), U^{*}(\cdot, \eta)  \big\rangle &=\hat{\eta} \big(-4\,\Lambda_{1\ep}  \int v_{\ep} \, \vp_{1\ep}' +\cO(\hat{\eta}^2) \big)\\
&=-6\,\hat{\eta}\,\Lambda_{10}\, \|v_0\|_{L^2(\mR)}^2+\cO\big(\hat{\eta}(\ep+\hat{\eta}^2)\big) =\hat{\eta}\bigg(-\f{16}{3}\, +\cO\big(\ep+\hat{\eta}^2\big)\bigg)\, .
\end{align*} 
The last equality comes  from the fact that since  $v_0(\cdot)=\Psi_{\kdv}(\cdot)=\sech^2(\f{\sqrt{3}}{2}\cdot)$, we have  that $\|v_0\|_{L^1}=\f{4\sqrt{3}}{3},\, \|v_0\|_{L^2}^2=\f{8\sqrt{3}}{9}.$
Consequently, we obtain
\beq \label{alpha-kpa}
\alpha(\eta)=\hat{\eta}\bigg( -\f23 +\cO\big(\ep+\hat{\eta}^2\big)\bigg)\, , \quad \kappa(\eta)=\hat{\eta}\bigg( \f{2\sqrt{3}}{3} +\cO\big(\ep+\hat{\eta}^2\big)\bigg)\, .
\eeq
Plugging \eqref{U}, \eqref{U*} and \eqref{alpha-kpa} into \eqref{defbasis-dual}, we obtain that:
\begin{align*}
&  ({2\ep})^{-1} (1, 1) \cdot \big(\diag (1, \, \p_x) g_1\big)(\ep^{-1}x, \ep^2\hat{\eta})= v_{\ep}' + \cO\big(\ep+\hat{\eta}^2\big), \\
 &  ({2\ep})^{-1} (1, 1) \cdot \big(\diag (1, \, \p_x) g_2\big)(\ep^{-1}x, \ep^2\hat{\eta})= -\f{1}{\sqrt{3}}v'_{\ep}-\f{\vp_{1\ep}}{2} + \cO\big(\ep+\hat{\eta}^2\big), \\
&(1,1) \cdot \big(\diag (1, \, -\p_x^{-1}) g_1^{*}\big)(\ep^{-1}x, \ep^2\hat{\eta})=-\f{\sqrt{3}}{4}
\int_{-\infty}^x \vp_{1\ep} + \cO\big(\ep+\hat{\eta}^2\big), \\
&(1,1) \cdot \big(\diag (1, \, -\p_x^{-1}) g_2^{*}\big)(\ep^{-1}x, \ep^2\hat{\eta})=-\f{\sqrt{3}}{2}v_{\ep}+\cO\big(\ep+\hat{\eta}^2\big).
\end{align*}
In view of the definitions \eqref{base-kp}, we finally obtain  that
\begin{align*}
& \big\|({2\ep})^{-1} (1, 1) \cdot \big(\diag (1, \, \p_x) g_k\big)(\ep^{-1}\cdot, \ep^2\hat{\eta}) -g_{0k}(\cdot, \hat{\eta})\big\|_{L_{\hat{a}}^2(\mR)}\lesssim \cO\big(\ep+\hat{\eta}^2\big), \, \\
&\big\|( (1, 1) \cdot \big(\diag (1, \, -\p_x^{-1}) g_k^{*}\big)(\ep^{-1}\cdot, \ep^2\hat{\eta}) -g_{0k}^{*}(\cdot, \hat{\eta})\big\|_{L_{-\hat{a}}^2(\mR)}\lesssim \cO\big(\ep+\hat{\eta}^2\big)\, .
\end{align*} 

\section*{Acknowledgement}
The research of the first author is supported by the ANR BOURGEONS project,  ANR-23-CE40-0014-01.
Part of this work was conducted during the postdoctoral stay of the second author at Institute of Mathematics of Toulouse and math department of INSA Toulouse. We express our sincere appreciation for the outstanding research environment provided by these institutions.

\bibliographystyle{abbrv}

\nocite{*}
\bibliography{ref}
\end{document}